\documentclass[12pt]{amsart}
\usepackage[english]{babel}
\usepackage[pdftex,paper=a4paper,portrait=true,textwidth=450pt,textheight=675pt,tmargin=3cm,marginratio=1:1]{geometry}
\usepackage{amsfonts}
\usepackage[dvips]{graphics}
\usepackage[colorinlistoftodos]{todonotes}
\usepackage{amsmath}
\usepackage{amsthm}
\usepackage{amssymb}
\usepackage{bbm}
\usepackage{cancel}
\usepackage{color}
\usepackage[curve]{xypic}
\usepackage{graphicx}
\usepackage{verbatim}
\usepackage{version}
\usepackage{color}

\newtheorem{theorem}{Theorem}[section]
\newtheorem{corollary}[theorem]{Corollary}
\newtheorem{proposition}[theorem]{Proposition}
\newtheorem{lemma}[theorem]{Lemma}
\newtheorem{conjecture}[theorem]{Conjecture}

\theoremstyle{definition}
\newtheorem{definition}[theorem]{Definition}

\newtheorem{remark}[theorem]{Remark}
\newtheorem*{acknowledgements}{Acknowledgements}
\theoremstyle{property}

\DeclareFontFamily{OT1}{rsfs}{}
\DeclareFontShape{OT1}{rsfs}{n}{it}{<-> rsfs10}{}
\DeclareMathAlphabet{\curly}{OT1}{rsfs}{n}{it}
\newcommand\A{\mathcal A}

\newcommand\sfT{\mathsf T}
\newcommand\sfD{\mathsf D}

\newcommand\SW{\mathrm{SW}}

\newcommand\Eu{\mathrm{Eu}}

\renewcommand\O{\mathcal O}
\newcommand\PP{\mathbb P}

\newcommand\cA{\mathcal A}

\newcommand\E{\mathbb E}

\newcommand\Coeff{\mathrm{Coeff}}

\newcommand\vir{\mathrm{vir}}
\newcommand\odd{\mathrm{odd}}

\newcommand\td{\mathrm{td}}
\newcommand\bslambda{\boldsymbol \lambda}
\newcommand\bsmu{\boldsymbol \mu}
\newcommand\C{\mathbb C}

\newcommand\FF{\mathbb F}

\newcommand\sfZ{\mathsf Z}
\newcommand\sfA{\mathsf A}
\newcommand\Q{\mathbb Q}

\newcommand\Z{\mathbb Z}
\newcommand\cZ{\mathcal Z}

\newcommand\s{\mathfrak s}
\newcommand\I{\mathcal I}

\newcommand\INTO{\ar@{^{(}->}[r]}

\newcommand\rk{\operatorname{rk}}

\newcommand\ch{\operatorname{ch}}

\newcommand\vd{\operatorname{vd}}

\renewcommand\hom{\mathcal{H}{\it{om}}}
\newcommand\Hom{\operatorname{Hom}}
\newcommand\Ext{\operatorname{Ext}}
\newcommand\ext{\curly Ext}

\newcommand\Ob{\operatorname{Ob}}
\newcommand\Bl{\operatorname{Bl}}
\newcommand\Pic{\operatorname{Pic}}

\newcommand\Spec{\operatorname{Spec}\,}
\newcommand\Hilb{\operatorname{Hilb}}

\newcommand\beq[1]{\begin{equation}\label{#1}}
\newcommand\eeq{\end{equation}}
\newcommand\beqa{\begin{eqnarray*}}
\newcommand\eeqa{\end{eqnarray*}}
\newcommand\<{\langle}
\renewcommand\>{\rangle}

\begin{document}
\title[Virtual refinements of the Vafa-Witten formula]{Virtual refinements of the Vafa-Witten formula}
\author[L.~G\"ottsche and M.~Kool]{Lothar G\"ottsche and Martijn Kool}
\maketitle
\centerline{\emph{with an appendix by Lothar G\"ottsche and Hiraku Nakajima}}

\begin{abstract}
We conjecture a formula for the generating function of virtual $\chi_y$-genera of moduli spaces of rank 2 sheaves on arbitrary surfaces with holomorphic 2-form. Specializing the conjecture to minimal surfaces of general type and to virtual Euler characteristics, we recover (part of) a formula of C.~Vafa and E.~Witten. 

These virtual $\chi_y$-genera can be written in terms of descendent Donaldson invariants. Using T.~Mochizuki's formula, the latter can be expressed in terms of Seiberg-Witten invariants and certain explicit integrals over Hilbert schemes of points. These integrals are governed by seven universal functions, which are determined by their values on $\PP^2$ and $\PP^1 \times \PP^1$. Using localization we calculate these functions up to some order, which allows us to check our conjecture in many cases. 

In an appendix by H.~Nakajima and the first named author, the virtual Euler characteristic specialization of our conjecture is extended to include $\mu$-classes, thereby interpolating between Vafa-Witten's formula and 
Witten's conjecture for Donaldson invariants.
\end{abstract}

\section{Introduction} 

Let $S$ be a smooth projective complex surface with $b_1(S) = 0$ and polarization $H$. We denote by $$M:=M_{S}^{H}(r,c_1,c_2)$$ the moduli space of rank $r$ Gieseker $H$-stable torsion free sheaves on $S$ with Chern classes $c_1 \in H^2(S,\Z)$, $c_2 \in H^4(S,\Z)$. Suppose that no rank $r$ strictly Gieseker $H$-semistable sheaves with Chern classes $c_1, c_2$ exist. Then $M_{S}^{H}(r,c_1,c_2)$ is projective. 
T.~Mochizuki \cite{Moc} studied a perfect obstruction theory on $M$ with 
\begin{equation} \label{Tvir}
T^\vir = R\pi_* R\hom(\E,\E)_0[1],
\end{equation}
where $\E$ denotes the universal sheaf
on $M \times S$, $\pi : M \times S \rightarrow M$ is projection, and $(\cdot)_0$ denotes the trace-free part.\footnote{Although $\E$ may only exist \'etale locally, $R\pi_* R\hom(\E,\E)_0$ exists globally \cite[Sect.~10.2]{HL}.}  

This leads to a virtual cycle on $M$ of degree equal to the virtual dimension
\begin{equation} \label{vd}
\vd(M)=2rc_2-(r-1)c_1^2-(r^2-1)\chi(\O_S).
\end{equation}
The (algebraic) Donaldson invariants are then obtained by capping certain classes with the virtual cycle
\vspace{-0.1cm}
\begin{equation} \label{invgeneral}
\int_{[M]^{\vir}}  \tau_{\alpha_1}(\sigma_1) \cdots \tau_{\alpha_m}(\sigma_m),
\end{equation}
where $\sigma_1, \ldots \sigma_m \in H^*(S,\Q)$, $\alpha_1, \ldots, \alpha_m \geq 0$ are the descendence degrees, and the insertions $\tau_{\alpha_i}(\sigma_i)$ are defined in Section \ref{sec1}. One of the main achievements of \cite{Moc} is a beautiful formula expressing \eqref{invgeneral} for $r=2$ in terms of Seiberg-Witten invariants of $S$ and certain explicit integrals over $S^{[n_1]} \times S^{[n_2]}$, where $S^{[n_i]}$ denotes the Hilbert scheme of $n_i$ points on $S$. This formula was used by the first named author, H.~Nakajima, and K.~Yoshioka to prove the Witten conjecture for algebraic surfaces \cite{GNY3}. 

We are interested in the virtual $\chi_y$-genus of $M$ defined in \cite{FG}
$$
\chi_{-y}^{\vir}(M) := \sum_{p \geq 0} (-y)^p \chi^\vir(M,\Omega_{M}^{p,\vir}) \in \Z[y],
$$
where $\chi^\vir(M, \cdot)$ is virtual holomorphic Euler characteristic and $\Omega_{M}^{p, \vir} = \Lambda^p (T_{M}^{\vir})^{\vee}$.\footnote{We denote the $K$-group generated by locally free sheaves on $M$ by $K^0(M)$. Virtual holomorphic Euler characteristic and the definition of $\Lambda^p E \in K^0(M)$ for an arbitrary element $E \in K^0(M)$ are explained in \cite{FG}. The notation $(\cdot)^\vee$ is for derived dual.} We will usually use its shifted version
$$
\overline \chi_{-y}^{\vir}(M):=y^{-\frac{\vd(M)}{2}}\chi_{-y}^{\vir}(M),
$$ 
which is a \emph{symmetric} Laurent polynomial in $y^{\pm \frac{1}{2}}$ \cite[Cor.~4.9]{FG}. The virtual Euler characteristic is defined as
$$
e^\vir(M) :=\overline \chi_{-1}^{\vir}(M)= \int_{[M]^\vir} c_{\vd(M)}(T^\vir),
$$
where the last equality is the virtual Hopf index theorem \cite[Cor.~4.8]{FG}. We are interested in the coefficients of the generating function 
 $$
 \sfZ_{S,H,c_1}(x,y) := \sum_{c_2} \overline{\chi}_{-y}^\vir(M_{S}^{H}(2,c_1,c_2)) \, x^{4c_2 - c_{1}^{2} - 3 \chi(\O_S)},
 $$
where the power is $\vd(M_{S}^{H}(2,c_1,c_2))$. We will encounter the Dedekind eta function and three of the Jacobi theta functions
 \begin{align}
 \begin{split} \label{thetas}
\eta(q) &= q^{\frac{1}{24}} \prod_{n=1}^{\infty} (1-q^n), \quad \ \ \theta_1(q,y) = \sum_{n \in \Z} (-1)^n q^{\big(n+\frac{1}{2}\big)^2} y^{n+\frac{1}{2}} \\
\theta_2(q,y) &= \sum_{n \in \Z} q^{\big(n+\frac{1}{2}\big)^2} y^{n+\frac{1}{2}}, \quad \theta_3(q,y) = \sum_{n \in \Z} q^{n^2} y^{n}.
\end{split}
\end{align}
We define $$\overline{\eta}(q) := q^{-\frac{1}{24}} \eta(q)$$ and write the corresponding ``Nullwerte'' by $\theta_i(q) = \theta_i(q,1)$.

Seiberg-Witten invariants are oriented diffeomorphism invariants of 4-manifolds. For smooth projective surfaces over $\C$ satisfying $b_1(S)=0$ and $p_g(S)>0$, Seiberg-Witten invariants of an algebraic class $a \in H^2(S,\Z)$ can be constructed in terms of a natural virtual class on the linear system $|\O_S(a)|$ \cite[Sect.~6.3]{Moc}.\footnote{In \cite{DKO},  M.~D\"urr, A.~Kabanov, and C.~Okonek defined so-called Poincar\'e invariants of arbitrary smooth projective surfaces, which they conjectured to concide with Seiberg-Witten invariants from differential geometry. This conjecture is now fully established by the work of H.-l.~Chang and Y.-H.~Kiem \cite{CK}. More precisely \cite{DKO}, the Seiberg-Witten invariant $\SW(a)$ with $a \in H^2(S,\Z)$ constructed in algebraic geometry coincides with $\widetilde{\SW}(2a-K_S)$, where $\widetilde{\SW}(b)$ denotes the Seiberg-Witten invariant for class $b \in H^2(S,\Z)$ constructed in differential geometry.}
 Either the virtual class is zero, in which case one defines $\SW(a) = 0$, or it has virtual dimension zero, in which case its degree is denoted by $\SW(a) \in \Z$. A class $a \in H^2(S,\Z)$ is called a \emph{Seiberg-Witten basic class} when $\SW(a) \neq 0$.
Many surfaces only have Seiberg-Witten basic classes $0$ and $K_S \neq 0$ with corresponding Seiberg-Witten invariants $1$ and $(-1)^{\chi(\O_S)}$. Such surfaces are for example minimal surfaces of general type with $p_g(S)>0$ and $b_1(S)=0$ \cite[Thm.~7.4.1]{Mor}.
\begin{conjecture} \label{conj}
Let $S$ be a smooth projective surface with $b_1(S) = 0$ and $p_g(S)>0$. Suppose the Seiberg-Witten basic classes of $S$ are $0$ and $K_S \neq 0$. Let $H, c_1,c_2$ be chosen such that there exist no rank 2 strictly Gieseker $H$-semistable sheaves with Chern classes $c_1,c_2$. Then $\overline{\chi}_{-y}^{\vir}(M^H_S(2,c_1,c_2))$ is the coefficient of $x^{4c_2-c_1^2-3\chi(\O_S)}$ of 
\begin{align*}
\psi_S(x,y):=8 \Bigg( \frac{1}{2} \prod_{n=1}^{\infty} \frac{1}{(1-x^{2n})^{10}(1-x^{2n} y)(1-x^{2n} y^{-1})} \Bigg)^{\chi(\O_S)} \Bigg( \frac{2 \overline{\eta}(x^4)^2}{\theta_3(x,y^{\frac{1}{2}})} \Bigg)^{K_{S}^2}.
\end{align*}
\end{conjecture}
In  this conjecture, we can view $\psi_S(x,y)$ as generating function for $\overline{\chi}_{-y}^{\vir}(M^H_S(2,c_1,c_2))$ for all $c_1,c_2$. As a consequence of the conjecture, $\sfZ_{S,H,c_1}(x,y)$ is independent of the choice of polarization $H$, which we often omit from our notation. There is actually a closed formula for $\sfZ_{S,c_1}(x,y)$, which we give in Proposition \ref{fixedc1prop}. Note that the first factor is related to $\theta_1(q,y)$ by the Jacobi triple product formula
$$
\theta_1(q,y) =q^{\frac{1}{4}}(y^{\frac{1}{2}}-y^{-\frac{1}{2}})\prod_{n=1}^\infty(1-q^{2n})(1-q^{2n} y)(1-q^{2n} y^{-1}).
$$
We also present a generalization of this conjecture to \emph{arbitrary} smooth projective surfaces $S$ with $b_1(S) = 0$ and $p_g(S)>0$ (Conjecture \ref{generalsurfconj}), but the above conjecture is easier to state. 

\begin{remark} \label{bar} Replacing $x$ by $x y^{\frac{1}{2}}$, we go from generating functions for $\overline{\chi}_{-y}^{\vir}$ to generating functions for $\chi_{-y}^{\vir}$. Therefore we also get a conjectural generating function for the non-shifted virtual $\chi_y$-genera of the moduli spaces. Under the same assumptions, Conjecture \ref{conj} gives that $\chi^{\vir}_{-y}(M^H_S(2,c_1,c_2))$ is the coefficient of $x^{4c_2-c_1^2-3\chi(\O_S)}$ of 
$$\psi_S(xy^{\frac{1}{2}},y)=8 \Bigg( \frac{1}{2} \prod_{n=1}^{\infty} \frac{1}{(1-x^{2n}y^n)^{10}(1-x^{2n} y^{n+1})(1-x^{2n}y^{n-1})} \Bigg)^{\chi(\O_S)} \Bigg( \frac{2 \overline{\eta}(x^4y^2)^2}{\theta_3(xy^{\frac{1}{2}},y^{\frac{1}{2}})} \Bigg)^{K_{S}^2}$$
with $$\overline\eta(x^4y^2)=\prod_{n=1}^\infty (1-x^{4n}y^{2n}), \quad \theta_3(xy^{\frac{1}{2}},y^{\frac{1}{2}})=\sum_{n\in \Z} y^{\binom{n+1}{2}}x^{n^2}.$$
This formula also makes it evident that $\chi^{\vir}_{-y}(M^H_S(2,c_1,c_2))$ is a polynomial.
\end{remark}

Next, denote by 
$$
\sfZ_{S,c_1}(x) := \sfZ_{S,c_1}(x,1)
$$ 
the generating function of virtual Euler characteristics. When $-K_S H > 0$ or $K_S=0$ (and some other cases including elliptic surfaces), these are just ordinary topological Euler characteristics because the obstructions vanish and the moduli space is smooth. Then $\sfZ_{S,c_1}(x)$ was studied by many people, e.g.~\cite{Kly, Got1, Got2, GH, Yos1, Yos2, Yos3}. Conjecture \ref{conj} implies the following formula which follows by specializing to $y=1$.
\begin{corollary} [Proposition \ref{fixedc1prop}] \label{cor}
Assume Conjecture \ref{conj}. Let $S$ be a smooth projective surface with $b_1(S) = 0$ and $p_g(S)>0$. Suppose the Seiberg-Witten basic classes of $S$ are $0$ and $K_S \neq 0$. Let $H, c_1$ be chosen such that there exist no rank 2 Gieseker strictly $H$-semistable sheaves with first Chern class $c_1$. Then
\begin{align*} 
\sfZ_{S,c_1}(x) =2\sum_{k=0}^3 \frac{(i^k)^{c_1^2 - \chi(\O_S)}}{(2\overline{\eta}((-1)^kx^2)^{12})^{\chi(\O_S)}}\Bigg(\frac{2\overline{\eta}(x^4)^2}{\theta_3(i^kx)} \Bigg) ^{K_S^2},
\end{align*}
where $i = \sqrt{-1}$.
\end{corollary}

From definition \eqref{thetas}, we see that
$\theta_3(i^k x)=\theta_3(x^4)+i^k\theta_2(x^4)$.
Therefore we can rewrite the formula for $\sfZ_{S,c_1}(x)$ of Corollary \ref{cor} as 
\begin{align} 
\begin{split} \label{evirgenfunv2}
&\frac{2}{(2\overline{\eta}(x^2)^{12})^{\chi(\O_S)}} \Bigg\{ \Bigg( \frac{\theta_3(x^4)+\theta_2(x^4)}{2\overline{\eta}(x^4)^2} \Bigg)^{-K_S^2} + (-1)^{c_1^2 - \chi(\O_S)} \Bigg( \frac{\theta_3(x^4)-\theta_2(x^4)}{2\overline{\eta}(x^4)^2} \Bigg)^{-K_S^2} \Bigg\} \\
&+\frac{2 (-i)^{c_1^2 - \chi(\O_S)}}{(2\overline{\eta}(-x^2)^{12})^{\chi(\O_S)}} \Bigg\{ \Bigg( \frac{\theta_3(x^4)-i\theta_2(x^4)}{2\overline{\eta}(x^4)^2} \Bigg)^{-K_S^2} + (-1)^{c_1^2 - \chi(\O_S)} \Bigg( \frac{\theta_3(x^4)+i\theta_2(x^4)}{2\overline{\eta}(x^4)^2} \Bigg)^{-K_S^2} \Bigg\}.
\end{split}
\end{align}
In \cite{VW}, C.~Vafa and E.~Witten study certain invariants related to $S$-duality. On $\PP^2$ their invariants are topological Euler characteristics $e(M)$. For surfaces $S$ with smooth connected canonical divisor, they give a formula \cite[Eqn.~(5.38)]{VW}. Equation \eqref{evirgenfunv2} coincides with \emph{part of} their formula, namely all except the first two terms of \cite[Eqn.~(5.38)]{VW} and up to an overall factor $x^{-\chi(\O_S) + K_S^2/3}$ coming from our choice of normalization.
Likewise, for $y=1$, the more general Conjecture \ref{generalsurfconj} specializes to (part of) a formula from the physics literature due to R.~Dijkgraaf, J.-S.~Park, and B.~Schroers \cite{DPS}.

\begin{remark} By work of S.~Donaldson, D.~Gieseker and J.~Li, and others, if we fix $S$ with $b_1(S)=0$ and ample class $H$, then 
$M:=M^H_S(2,c_1,c_2)$ is irreducible and generically smooth of the expected dimension for sufficiently large $\vd=4c_2-c_1^2-3\chi(\O_S)$. It is then also normal and a local complete intersection. See \cite[Chap.~9]{HL} for references.
In this case \cite[Thm.~4.15]{FG} implies that $e^\vir(M)$ is the degree of the Fulton Chern class  $c_{F}(M)$, which agrees with 
the Euler characteristic of any smoothening of $M$, and Corollary \ref{cor} predicts this number. More generally, all virtual Chern numbers of $M$ coincide with the corresponding Chern numbers of any smoothening \cite[Rem.~4.16]{FG}. In particular the virtual $\chi_y$-genus of $M$ equals the $\chi_y$-genus of any smoothening and Conjecture \ref{conj} predicts these genera.
\end{remark}
\begin{remark}
Recently, Y.~Tanaka and R.~P.~Thomas \cite{TT1} defined a symmetric perfect obstruction theory on the moduli space of stable Higgs pairs $(E,\phi)$ on $S$, where $E$ has fixed determinant and $\phi$ is trace-free.\footnote{In \cite{TT1}, Tanaka-Thomas consider the case where semistability and stability of Higgs pairs coincide. They treat the semistable case in a separate paper \cite{TT2}.} Stable Higgs pairs are related by a Hitchin-Kobayashi correspondence to solutions of the Vafa-Witten equations. There is a $\C^*$-scaling action on the Higgs field and Tanaka-Thomas define ``$\mathrm{SU}(r)$ Vafa-Witten invariants'' by virtual localization with respect to this action. They show that the contribution to the invariant of the components corresponding to $\phi = 0$ are precisely the virtual Euler characteristics $\sfZ_{S,c_1}(x)$ that we study (though Tanaka-Thomas's invariants are defined for any rank). Moreover in the rank 2 case and for $S$ with smooth connected canonical divisor and $b_1(S) = 0$, they conjecture that the contribution of the \emph{other} components of the $\C^*$-fixed locus corresponds to the first two terms of \cite[Eqn.~(5.38)]{VW}. Recall that these are precisely the two terms that we do not see. They gather evidence for this by computing the contributions of other components for some low orders. Therefore by the calculations of this paper and their conjecture, their invariant indeed appears to be the correct mathematical definition of Vafa-Witten's invariant \cite[Eqn.~(5.38)]{VW}. Also recently, A.~Gholampour, A.~Sheshmani, and S.-T.~Yau studied Donaldson-Thomas invariants of local surfaces \cite{GSY}. Their invariants are closely related to Tanaka-Thomas's invariants. The virtual Euler characteristics that we calculate are part of their invariants.
\end{remark}

We approach Conjecture \ref{conj} as follows:
\begin{itemize}
\item We use the virtual Hirzebruch-Riemann-Roch formula\footnote{See \cite[Cor.~3.4]{FG}, or \cite{CFK} in the context of $[0,1]$-manifolds.} to express $\overline{\chi}_{-y}(M)$ in terms of certain descendent Donaldson invariants (Proposition \ref{chiyinsert}).
\item We apply Mochizuki's formula to these descendent Donaldson invariants. This expresses $\sfZ_{S,c_1}(x,y)$ in terms of Seiberg-Witten invariants of $S$ and certain integrals over $S^{[n_1]} \times S^{[n_2]}$. Although Mochizuki's formula requires $p_g(S)>0$, these integrals make sense for \emph{any} surface $S$.
\item We show that the integrals over $S^{[n_1]} \times S^{[n_2]}$ can be expressed in terms of seven universal series $A_1, \ldots, A_7 \in 1+q \, \Q[y](\!(s)\!)[[q]]$ (Proposition \ref{univ}).
\item The universal series $A_i$ are entirely determined by their values on $S = \PP^2$ and $\PP^1 \times \PP^1$. We calculate $A_i(s,y,q)$ to order $q^{7}$ and $A_i(s,1,q)$ to order $q^{30}$ by Atiyah-Bott localization (Section \ref{toricsec} and Appendix \ref{app1}). 
\item We then verify Conjecture \ref{conj} up to a certain order in $x$ for examples of the following types: blow-ups $\Bl_p K3$, double covers of $\PP^2$, double covers of $\PP^1 \times \PP^1$, double covers of Hirzebruch surfaces $\FF_a$, surfaces in $\PP^3$, $\PP^2 \times \PP^1$ and $\PP^1 \times \PP^1 \times \PP^1$, and complete intersections in $\PP^4$ and $\PP^5$ (Section \ref{evidencesec}). 
\end{itemize} 
The reduction to toric surfaces also allows us to relate $\sfZ_{S,c_1}(x,1)$ to the Nekrasov partition function with one adjoint matter and one fundamental matter (Appendix \ref{Nekrasovsec}). 
As a consequence, $\sfZ_{S,c_1}(x,1)$ can be expressed in terms of four universal series $F_0, H, G_1, G_2$ (see Remark \ref{4univ} for details). This is not used elsewhere in the paper.

We present two generalizations of Conjecture \ref{conj}: 
\begin{itemize}
\item Conjecture \ref{numconj} is a statement purely about intersection numbers on Hilbert schemes of points. Together with a strong form of Mochizuki's formula (Remark \ref{strongform}), it implies Conjecture \ref{conj} (see Proposition \ref{numconjimpliesconj}). In addition, the strong form of Mochizuki's formula and Conjecture \ref{numconj} imply a version of Conjecture \ref{conj} for arbitrary blow-ups of surfaces $S$ with $b_1(S)=0$, $p_g(S)>0$, and Seiberg-Witten basic classes $0$ and $K_S \neq 0$ (Proposition \ref{nonmingt}). We test Conjecture \ref{numconj} in many cases in Section \ref{evidencesec}.
\item Conjecture \ref{generalsurfconj} generalizes Conjecture \ref{conj} to \emph{arbitrary} surfaces with $b_1(S) = 0$ and $p_g(S)>0$. This conjecture has two further applications. (1) It implies a blow-up formula for virtual $\chi_y$-genera, which is reminiscent of the blow-up formula of  W.-P.~Li and Z.~Qin (Proposition \ref{blowupcor}). (2) It implies a formula for surfaces with canonical divisor with irreducible reduced connected components (Proposition \ref{disconn}). For $y=1$, the latter recovers another formula of Vafa-Witten \cite[(5.45)]{VW}. We check Conjecture \ref{generalsurfconj} up to a certain order in $x$ in the following cases (other than the cases above): $K3$ surfaces, blow-ups $\Bl_p \Bl_q K3$, and elliptic surfaces $E(n)$ for various $n \geq 4$. 
\end{itemize}

In Appendix \ref{appC}, the first named author and H.~Nakajima conjecture a formula unifying the virtual Euler characteristic specialization of Conjecture \ref{generalsurfconj} and Witten's conjecture for Donaldson invariants. 

In \cite{GK} we extend these results and conjectures to virtual elliptic genera and virtual cobordism classes. Besides Mochizuki's formula \cite{Moc}, this paper uses ideas from \cite{GNY1,GNY2,GNY3}. The physics approach to the calculation of elliptic genera of instanton moduli spaces was discussed in N.~Nekrasov's PhD thesis \cite{Nek1} and the papers \cite{LNS, BLN}. We refer to \cite{LLZ} for applications to Gopakumar-Vafa invariants.

\begin{acknowledgements}
We thank Richard Thomas for suggesting to look at Mochizuki's book in the context of the virtual Euler characteristic many years ago. The first named author thanks Hiraku Nakajima and K$\bar{\textrm{o}}$ta Yoshioka for collaboration and many useful discussions on instanton counting and Mochizuki's formula over many years, and in particular to Hiraku Nakajima for very useful discussions and comments. The second named author was supported by Marie Sk{\l}odowska-Curie IF 656898.
\end{acknowledgements}

\section{Mochizuki's formula} \label{sec1}

Let $S$ be a smooth projective surface with $b_1(S) = 0$ and polarization $H$. Denote by $M:=M_{S}^{H}(r,c_1,c_2)$ the moduli space of rank $r$ Gieseker $H$-stable torsion free sheaves on $S$ with Chern classes $c_1,c_2$. We assume there are no rank $r$ strictly Gieseker $H$-semistable sheaves with Chern classes $c_1,c_2$. Then $M_{S}^{H}(r,c_1,c_2)$ is projective. In this section, we first show that ${\chi}_{-y}^{\vir}(M)$ can be expressed in terms of descendent Donaldson invariants. We then recall Mochizuki's formula \cite{Moc} and apply it to our setting.

We start with some notation. Assume there exists a universal sheaf $\E$ on $M \times S$ (in fact, we get rid of this assumption at the end of this section in Remark \ref{univexists}). Let $\sigma \in H^*(S,\Q)$ and $\alpha \geq 0$, then we define
$$
\tau_{\alpha}(\sigma) := \pi_{M*} \big( \ch_{2+\alpha}(\E) \cap \pi_{S}^{*} \, \sigma \big),
$$
where 
$\pi_M : M \times S \rightarrow M$ and $\pi_S : M \times S \rightarrow S$ denote projections. We refer to $\tau_{\alpha}(\sigma)$ as a descendent insertion of descendence degree $\alpha$. The insertions $\tau_0(\sigma)$ are called primary insertions.

We introduce some further notation which will be useful later. Suppose $X$ is any projective $\C$-scheme and $E$ a vector bundle on $X$, then we define
$$
\Lambda_y E := \sum_{p=0}^{\rk(E)} [\Lambda^p E] \, y^p \in K^0(X)[[y]].
$$
This element is invertible in $K^0(X)[[y]]$, so we can define $\Lambda_y (-E) = 1/ \Lambda_y E$. Hence we can also define $\Lambda_y E$ for any element $E$ of $K^0(X)$. Next, for any element $E$ of $K^0(X)$ we define
\begin{align} \label{defT}
\sfT_{y}(E,t) := t^{-\rk E} \sum_k \left\{ \ch(\Lambda_{y} E^{\vee}) \, \td(E) \right\}_k t^{k},
\end{align}
where $\{\cdot\}_k \in A^k(X)_{\Q}$ selects the degree $k$ part in the Chow ring. Since we have
\begin{equation*} 
\Lambda_y (E_1 + E_2) = \Lambda_y E_1 \otimes \Lambda_y E_2,
\end{equation*}
the standard properties of Chern characters and Todd classes give
\begin{equation} \label{Tprod}
\sfT_y(E_1 + E_2,t) = \sfT_y(E_1,t) \, \sfT_y(E_2,t).
\end{equation}
This multiplicative property will be crucial in Section \ref{secuniv}. We also note that for a line bundle $L$ on $X$ with $c_1(L) = x$, we have
\begin{equation} \label{Tlinebd}
\sfT_y(L,t) = \frac{x (1+y e^{-x t})}{1-e^{-x t}}.
\end{equation}
The variable $t$ can be exploited for a convenient normalization. Indeed if we take $t = 1+y$, then equations \eqref{Tprod} and \eqref{Tlinebd} imply
\begin{equation} \label{trivialshift}
\sfT_{y}(E - \O_X^{\oplus r},1+y) = \sfT_{y}(E,1+y)
\end{equation}
for all $r \geq 0$. This will be used in Section \ref{secuniv} as well. Another convenient consequence of the specialization $t = 1+y$ is that $\sfT_{y}(E,1+y)$ is a \emph{polynomial} in $1+y$. Moreover, its leading coefficient is given by
\begin{equation} \label{specialize}
\sfT_{y}(E,1+y) \Big|_{y=-1} = c(E),
\end{equation}
where $c(\cdot)$ denotes total Chern class. This essentially follows from \cite[Thm.~4.5(c)]{FG}. 

We come back to $M:=M_{S}^{H}(r,c_1,c_2)$. The next proposition involves an argument that appears more generally in the context of stable pairs on 3-folds in \cite{She} (see also \cite{Pan}).
\begin{proposition} \label{chiyinsert}
For every $S,H,r,c_1,c_2$ as above, there exists a polynomial expression $P(\E)$ in certain descendent insertions $\tau_{\alpha}(\sigma)$ and $y$ such that
$$
\chi^{\vir}_{-y}(M_{S}^{H}(r,c_1,c_2)) = \int_{[M_{S}^{H}(r,c_1,c_2)]^{\vir}} P(\E).
$$
\end{proposition}
\begin{proof}
By the virtual Hirzebruch-Riemann-Roch theorem \cite[Cor.~3.4]{FG}
\begin{align*}
\chi^{\vir}_{-y}(M) &= \int_{[M]^{\vir}} \sfT_{-y}(T^{\vir},1-y),
\end{align*}
where $\sfT_{-y}(T^{\vir},1-y)$ can be expressed as a $\Q$-linear combination of monomials in $c_i(T^{\vir})$ and $y$.
By \eqref{Tvir} and Grothendieck-Riemann-Roch, each $c_{i}(T^{\vir})$ can be expressed as a $\Q$-linear combination of monomials in 
\begin{equation} \label{twoterms}
\pi_{M*} \big( \ch_\alpha(\E) \ch_\beta(\E) \cap \pi_{S}^{*} \sigma \big),
\end{equation}
where $\sigma$ is one of the components of $\td(S)$. Therefore it suffices to show that every expression of the form \eqref{twoterms} is a polynomial in descendent insertions. This will then define the universal polynomial $P(\E)$. 

Let $\pi_{ij}$ and $\pi_i$ be the projections from $M \times S \times S$ to factors $(i,j)$ and $i$ respectively. Then \eqref{twoterms} equals
\begin{equation} \label{MxSxS}
\pi_{1*} \big( \pi_{12}^{*} \ch_\alpha(\E) \ \pi_{13}^{*} \ch_\beta(\E) \cap \pi_2^* \sigma \ \pi_{23}^{*} \Delta \big),
\end{equation}
where $\Delta \in H^4(S \times S,\Q)$ is (Poincar\'e dual to) the class of the diagonal. Next we consider the K\"unneth decomposition
$$
\Delta = \sum_{i+j=4} \theta_{1}^{(i)} \otimes \theta_{2}^{(j)},
$$
where $\theta_{1}^{(i)} \in H^i(S,\Q)$ and $\theta_{2}^{(j)} \in H^j(S,\Q)$. Substituting this decomposition into \eqref{MxSxS}, factoring the push-forward as $\pi_1 = \pi_M \circ \pi_{12}$, and applying the projection formula gives
\begin{equation*} 
\pi_{M*} \big( \ch_\alpha(\E) \, \ch_\beta(\E) \cap \pi_{S}^{*} \sigma \big) = \sum_{i+j=4} \tau_\alpha(\sigma \theta_{1}^{(i)}) \, \tau_\beta(\theta_{2}^{(j)}). \qedhere
\end{equation*}
\end{proof}
\begin{remark}
Note that in this section (and the next) we use $\chi_{-y}^{\vir}$ instead of $\overline{\chi}_{-y}^{\vir}$ as in Conjecture \ref{conj}. The reason is that the intermediate formulae of this and the next section look slightly easier for $\chi_{-y}^{\vir}$ whereas the final formula of Conjecture \ref{conj} looks more elegant for $\overline{\chi}_{-y}^{\vir}$. One can easily pass from one to the other by Remark \ref{bar}.
\end{remark}

Next we recall Mochizuki's formula \cite[Thm.~1.4.6]{Moc}. His formula is derived by a beautiful argument involving \emph{geometric wall-crossing} for moduli spaces of so-called Bradlow pairs depending on a stability parameter $\alpha \in \Q_{>0}$. For $\alpha \rightarrow \infty$ the moduli spaces are empty and for $\alpha \rightarrow 0^+$ the moduli space is a projective bundle over $M$. On a wall, Mochizuki uses a ``master space'' with $\C^*$-action, whose fixed locus contains components corresponding to the moduli spaces of Bradlow pairs on either side of the wall. Other components of the fixed locus can be seen as contributions from wall-crossing. They lead to the Seiberg-Witten invariants and integrals over $S^{[n_1]} \times S^{[n_2]}$ described below. These wall-crossing terms are computed by Graber-Pandharipande's virtual localization on the master space. This geometric wall-crossing is very different from \emph{motivic wall-crossing}, as in \cite{Joy}, which does not work for ample $K_S$ due to non-vanishing $\Ext^2$. Among other things, Mochizuki's formula was used in \cite{GNY3} to prove Witten's conjecture for algebraic surfaces. 

There are two ingredients for Mochizuki's formula. The first ingredient is the Seiberg-Witten invariants $\SW(a)$ of $S$ in class $a \in H^2(S,\Z)$ mentioned in the introduction. 

The second ingredient is certain integrals over products of Hilbert schemes of points. On $S^{[n_1]} \times S^{[n_2]} \times S$ we have the pull-backs of the universal ideal sheaves $\I_1$, $\I_2$ from $S^{[n_1]} \times S$, $S^{[n_2]} \times S$. For any line bundle $L \in \Pic(S)$ we denote by $L^{[n_i]}$ the tautological vector bundle on $S^{[n_i]}$ defined by
$$
L^{[n_i]} := p_* q^* L,
$$
with $p : \cZ_i \rightarrow S^{[n_i]}$, $q : \cZ_i \rightarrow S$ projections from the universal subscheme $\cZ_i \subset S^{[n_i]} \times S$. 

We consider $S^{[n_1]} \times S^{[n_2]}$ to be endowed with a trivial $\C^*$-action and we denote the generator of the character group by $\s$.\footnote{This action originates from the $\C^*$-action on the master space mentioned above.} Moreover we write $s$ for the generator of 
$$
H^*(B\C^*,\Q) = H^*_{\C^*}(pt,\Q) \cong \Q[s].
$$
Let $P(\E)$ be any polynomial in descendent insertions $\tau_{\alpha}(\sigma)$, which arises from a polynomial in Chern numbers of $T^\vir$ (e.g.~such as in Proposition \ref{chiyinsert}). For any $a_1, a_2 \in A^1(S)$ and $n_1, n_2 \in \Z_{\geq 0}$, Mochizuki defines $\Psi(a_1,a_2,n_1,n_2)$ as follows
\begin{equation} \label{Psi}
\Coeff_{s^0} \Bigg( \frac{P(\I_1(a_1) \otimes \s^{-1} \oplus \I_2(a_2) \otimes \s)}{Q(\I_1(a_1) \otimes \s^{-1}, \I_2(a_2) \otimes \s)} \frac{\Eu(\O(a_1)^{[n_1]}) \, \Eu(\O(a_2)^{[n_2]} \otimes \s^2)}{(2s)^{n_1+n_2 - \chi(\O_S)}} \Bigg).
\end{equation}
Let us explain this notation. Here $\Eu(\cdot)$ denotes $\C^*$-equivariant Euler class and $\Coeff_{s^0}$ refers to taking the coefficient of $s^0$.\footnote{This differs slightly from Mochizuki who uses $p_g(S)$ instead of $\chi(\O_S)$ and takes a residue. Consequently our definition differs by a factor $2$ from Mochizuki's. The difference is accounted for by the fact that Mochizuki works on the moduli stack of oriented sheaves which maps to $M$ via a degree $\frac{1}{2}:1$ \'etale morphism \cite[Rem.~4.2]{GNY3}.} The notation $\I_i(a_i)$ is short-hand for $\I_i \otimes \pi_S^* \O(a_i)$. Furthermore, for any $\C^*$-equivariant sheaves $\E_1$, $\E_2$ on $S^{[n_1]} \times S^{[n_2]} \times S$ flat over $S^{[n_1]} \times S^{[n_2]}$, Mochizuki defines
$$
Q(\E_1,\E_2) :=\Eu(- R\hom_{\pi}(\E_1,\E_2) - R\hom_{\pi}(\E_2,\E_1)), 
$$
where $\pi : S^{[n_1]} \times S^{[n_2]} \times S \rightarrow S^{[n_1]} \times S^{[n_2]}$ denotes projection and 
$$
R\hom_\pi(\cdot, \cdot) := R\pi_* R\hom(\cdot,\cdot).
$$ 
Finally, $P(\cdot)$ is the expression obtained from $P(\E)$ by formally replacing $\E$ by $\cdot$. We define 
$$
\widetilde{\Psi}(a_1,a_2,n_1,n_2,s)
$$
by expression \eqref{Psi} but \emph{without} applying $\Coeff_{s^0}$. 

Next, let $c_1,c_2$ be a choice of Chern classes and let $\ch = (2,c_1,\frac{1}{2}c_1^2 - c_2)$ denote the corresponding Chern character. For any decomposition $c_1 = a_1 + a_2$, Mochizuki defines
\begin{equation} \label{cA}
\cA(a_1,a_2,c_2) := \sum_{n_1 + n_2 = c_2 - a_1 a_2} \int_{S^{[n_1]} \times S^{[n_2]}} \Psi(a_1,a_2,n_1,n_2).
\end{equation}
We denote by $\widetilde{\cA}(a_1,a_2,c_2,s)$ the same expression with $\Psi$ replaced by $\widetilde{\Psi}$.

\begin{theorem}[Mochizuki] \label{mocthm}
Let $S$ be a smooth projective surface with $b_1(S) = 0$ and $p_g(S) >0$. Let $H, c_1,c_2$ be chosen such that there exist no rank 2 strictly Gieseker $H$-semistable sheaves with Chern classes $c_1,c_2$. Suppose the universal sheaf $\E$ exists on $M_{S}^{H}(2,c_1,c_2) \times S$. Suppose the following conditions hold:
\begin{enumerate}
\item[(i)] $\chi(\ch) > 0$, where $\chi(\ch) := \int_S \ch \cdot \td(S)$ and $\ch = (2,c_1,\frac{1}{2}c_1^2 - c_2)$.
\item[(ii)] $p_{\ch} > p_{K_S}$, where $p_{\ch}$ and $p_{K_S}$ are the reduced Hilbert polynomials associated to $\ch$ and $K_S$.
\item[(iii)] For all Seiberg-Witten basic classes $a_1$ satisfying $a_1 H \leq (c_1 -a_1) H$ the inequality is strict. 
\end{enumerate}
Let $P(\E)$ be any polynomial in descendent insertions, which arises from a polynomial in Chern numbers of $T^\vir$ (e.g.~such as in Proposition \ref{chiyinsert}). Then
$$
\int_{[M_{S}^{H}(2,c_1,c_2)]^{\vir}} P(\E) = -2^{1-\chi(\ch)} \sum_{{\scriptsize{\begin{array}{c} c_1 = a_1 + a_2 \\ a_1 H < a_2 H \end{array}}}} \SW(a_1) \, \cA(a_1,a_2,c_2).
$$
\end{theorem}

\begin{remark} \label{univexists}
The assumption that $\E$ exists on $M \times S$, where $M:=M_{S}^{H}(2,c_1,c_2)$, is unnecessary. As remarked in the introduction, $T^{\vir} = -R\hom_{\pi}(\E,\E)_0$ always exists globally so the left-hand side of Mochizuki's formula always makes sense. Moreover, Mochizuki \cite{Moc} works over the Deligne-Mumford stack of oriented sheaves,
which always has a universal sheaf. This can be used to show that global existence of $\E$ on $M \times S$ can be dropped from the assumptions. In fact, when working on the stack, $P$ can be \emph{any} polynomial in descendent insertions defined using the universal sheaf on the stack.
\end{remark}

\begin{remark}
Mochizuki's formula (Theorem \ref{mocthm}) holds without the assumption that there are no strictly $H$-Gieseker semistable sheaves and without assumption (iii). Then $\int_{[M]^{\vir}} P(\E)$ is defined via the moduli space of oriented reduced Bradlow pairs \cite[Def.~7.3.2]{Moc}. In this more general setup the definition of $\cA(a_1,a_2,c_2)$ has to be modified slightly: when $a_1 H = a_2 H$ the sum in \eqref{cA} is over all $n_1>n_2$ satisfying $n_1 + n_2 = c_2 - a_1 a_2$. 
\end{remark}
\begin{remark} \label{assumpmocthm}
It is conjectured in \cite{GNY3} that assumptions (ii) and (iii) can be dropped from Theorem \ref{mocthm} and the sum can be replaced by a sum over \emph{all} Seiberg-Witten basic classes. We will see in our computations that assumption (i) is necessary. 
\end{remark}
\begin{remark} \label{Mochcor}
Let the setup be as in Theorem \ref{mocthm}. Then 
$$
{\chi}_{-y}^\vir(M_{S}^{H}(2,c_1,c_2)) = -2^{1-\chi(\ch)}  \sum_{{\scriptsize{\begin{array}{c} c_1 = a_1 + a_2 \\ a_1 H < a_2 H \end{array}}}} \SW(a_1) \, \cA(a_1,a_2,c_2),
$$
where $P(\E)$ is determined by the following expression
$$
P(\E) := \sfT_{-y}(-R\hom_{\pi}(\E,\E)_0,1-y)
$$ 
with $\E = \I_1(a_1) \otimes \s^{-1} \oplus \I_2(a_2) \otimes \s$. We note that the rank of 
$$
-R\hom_{\pi}(\I_1(a_1) \otimes \s^{-1} \oplus \I_2(a_2) \otimes \s,\I_1(a_1) \otimes \s^{-1} \oplus \I_2(a_2) \otimes \s)_0
$$ 
is given by 
\begin{align*}
&\chi(\O_S)+(2n_1+2n_2 - 2\chi(\O_S)) + (n_1+n_2-\chi(\O(a_2-a_1))) \\
&\qquad \qquad \qquad \qquad \qquad \qquad \quad \ \ \! + (n_1+n_2-\chi(\O(a_1-a_2))) \\
&= 4c_2-c_1^2-3\chi(\O_S),
\end{align*}
which equals the rank of $T^\vir$ given by \eqref{vd}.
\end{remark}

\section{Universality} \label{secuniv}

In this section $S$ is \emph{any} smooth projective surface. We start with a well-known lemma, which we include for completeness.
\begin{lemma} \label{lem1}
Let $\pi : S^{[n]} \times S \rightarrow S^{[n]}$ denote the projection. Then
$$
-R\hom_\pi(\I,\I)_0 \cong \ext_{\pi}^{1}(\I,\I)_0 \cong T_{S^{[n]}},
$$
where $\I$ denotes the universal ideal sheaf and $T_{S^{[n]}}$ denotes the tangent bundle.
\end{lemma}
\begin{proof}
Since $\ext_{\pi}^{1}(\I,\I)_0 \cong T_{S^{[n]}}$, it suffices to show that for any ideal sheaf $I  = I_Z\subset \O_S$ we have $\Hom(I,I)_0 = \Ext^2(I,I)_0 = 0$. Clearly $\Hom(I,I)_0 = 0$ because $I$ is simple. 

Next we consider the trace map
$$
\Ext^2(I,I) \rightarrow H^2(\O_S).
$$
First applying $- \otimes K_S$ and then $\Hom(I,\cdot)$ to  
\begin{equation} \label{ses}
0 \rightarrow I \rightarrow \O_S \rightarrow \O_Z \rightarrow 0
\end{equation}
gives a long exact sequence
$$
0 \rightarrow \Hom(I,I \otimes K_S) \rightarrow \Hom(I,K_S) \rightarrow \Hom(I,K_S|_Z) \rightarrow \cdots.
$$
The natural map $H^0(K_S) \rightarrow \Hom(I,I \otimes K_S)$ is Serre dual to the trace map, so our goal is to show that this map is an isomorphism. It is enough to show that the restriction map $H^0(K_S) \rightarrow \Hom(I,K_S)$, which factors through $\Hom(I,I \otimes K_S)$, is an isomorphism. This in turn can be seen by applying $\Hom(\cdot, K_S)$ to \eqref{ses} and using that $Z$ is 0-dimensional.
\end{proof}

Our main object of study is the following generating function.
\begin{definition} \label{defZX}
For any $a$ in the Chow group  $A^1(S)$ we abbreviate $\chi(a) := \chi(\O(a))$. For any $a_1, c_1 \in A^1(S)$, we define 
\begin{align*} 
&\sfZ_{S}(a_1,c_1,s,y, q) := \\
&\sum_{n_1, n_2 \geq 0} q^{n_1+n_2} \int_{S^{[n_1]} \times S^{[n_2]}} \frac{\sfT^{\C^*}_{-y}(E_{n_1,n_2},1-y) \, \Eu(\O(a_1)^{[n_1]}) \, \Eu(\O(c_1-a_1)^{[n_2]} \otimes \s^2)}{\Eu(E_{n_1,n_2} - \pi_1^* \, T_{S^{[n_1]}} - \pi_2^* \, T_{S^{[n_2]}})}, 
\end{align*}
where $\sfT^{\C^*}$ denotes the $\C^*$-equivariant analog of \eqref{defT} and
\begin{align*}
E_{n_1,n_2} := -R\hom_{\pi}(\E,\E)_0+\chi(\O_S) \otimes \O+ \chi(c_1-2a_1) \otimes \O \otimes \s^2+ \chi(2a_1-c_1) \otimes \O \otimes \s^{-2}
\end{align*}
with $$\E = \I_1(a_1) \otimes \s^{-1} \oplus \I_2(a_2) \otimes \s.$$ This can be rewritten as
\begin{align*}
E_{n_1,n_2} 
= \pi_1^* \, T_{S^{[n_1]}} + \pi_2^* \, T_{S^{[n_2]}} &+ \O(c_1-2a_1)^{[n_2]} \otimes \s^2 + \O(2a_1-c_1 + K_S)^{[n_1] *} \otimes \s^2 \\
&- R\hom_{\pi}(\O_{\cZ_1}(a_1),\O_{\cZ_2}(c_1-a_1)) \otimes \s^2 \\
&+ \O(2a_1-c_1)^{[n_1]} \otimes \s^{-2} + \O(c_1-2a_1+K_S)^{[n_2] *} \otimes \s^{-2} \\
&- R\hom_{\pi}(\O_{\cZ_2}(c_1-a_1),\O_{\cZ_1}(a_1)) \otimes \s^{-2}.
\end{align*}
Here we used Serre duality and $(\cdot)^*$ denotes the dual vector bundle. The complex $E_{n_1,n_2}$ has rank $4n_1+4n_2$. Note that
\begin{equation} \label{leadingcoeff}
\sfZ_{S}(a_1,c_1,s,y,q) \in 1 + q \, \Q[y](\!(s)\!)[[q]]. 
\end{equation}
For $y=1$, using the $\C^*$-equivariant version of \eqref{specialize}, we obtain
$$
\sfT^{\C^*}_{-y}(E_{n_1,n_2},1-y) \Big|_{y=1} = c^{\C^*}(E_{n_1,n_2}).
$$
\end{definition}

Suppose we have a decomposition $c_1 = a_1 + a_2$. Then Remark \ref{Mochcor} implies
\begin{align} 
\begin{split} \label{keyexpr}
&\sum_{c_2 \in \Z} \widetilde{\cA}(a_1,c_1-a_1,c_2,s) \, q^{c_2} =  \sfZ_{S}\Big(a_1,c_1,s,y,\frac{q}{2s}\Big) \\
&\times (2s)^{\chi(\O_S)} \Bigg( \frac{1-e^{-2s(1-y)}}{1-y e^{-2s(1-y)}} \Bigg)^{\chi(c_1-2a_1)} \Bigg( \frac{1-e^{2s(1-y)}}{1-y e^{2s(1-y)}} \Bigg)^{\chi(2a_1-c_1)} q^{a_1 (c_1-a_1)}.
\end{split}
\end{align}
Besides Remark \ref{Mochcor}, this equality uses the following facts: Lemma \ref{lem1}, equation \eqref{Tprod}, equation \eqref{Tlinebd}, and equation \eqref{trivialshift}. For $y=1$ the second line of \eqref{keyexpr} simplifies to 
$$
(2s)^{\chi(\O_S)} \Bigg( \frac{2s}{1+2s} \Bigg)^{\chi(c_1-2a_1)} \Bigg( \frac{-2s}{1-2s} \Bigg)^{\chi(2a_1-c_1)} q^{a_1 (c_1-a_1)}.
$$

\begin{proposition} \label{univ}
There exist universal functions 
$$
A_1(s,y,q), \ldots, A_7(s,y,q) \in 1 + q \, \Q[y](\!(s)\!)[[q]]
$$ 
such that for any smooth projective surface $S$ and $a_1, c_1 \in A^1(S)$ we have
$$
\sfZ_{S}(a_1,c_1,s,y,q) = A_1^{a_1^2} \, A_2^{a_1 c_1} \, A_3^{c_1^2} \, A_4^{a_1 K_S} \, A_5^{c_1 K_S} \, A_6^{K_S^2} \, A_7^{\chi(\O_S)}.
$$
\end{proposition}
\begin{proof}
\noindent \textbf{Step 1: Universality.} For any smooth projective surface $S$, define $S_2 := S \sqcup S$ and denote by $S_2^{[n]}$ the Hilbert scheme of $n$ points on $S_2$. Then
$$
S_2^{[n]} = \bigsqcup_{n_1+n_2=n} S^{[n_1]} \times S^{[n_2]}.
$$
We endow $S_2^{[n]}$ with trivial $\C^*$-action. Denote by $\I_1$, $\I_2$ the sheaves on $S_2^{[n]} \times S$ whose restrictions to $S^{[n_1]} \times S^{[n_2]} \times S$ are the sheaves $\I_1$, $\I_2$ pulled back along projection to $S^{[n_1]} \times S$, $S^{[n_2]} \times S$ respectively. Let $L(a_1,c_1)$ be the vector bundle whose restriction to $S^{[n_1]} \times S^{[n_2]}$ is
$$
\pi_1^* \O(a_1)^{[n_1]} \oplus \pi_2^*\O(c_1-a_1)^{[n_2]} \otimes \s^2.
$$ 
Let $X$ be any rational function in the following list of Chern classes and with coefficients in $\Q[y](\!(s)\!)$
\begin{align*}
&c_{i_1}^{\C^*}(R\hom_{\pi}(\I_1(a_1),\I_2(c_1-a_1)) \otimes \s^2), \\
&c_{i_2}^{\C^*}(R\hom_{\pi}(\I_2(c_1-a_1),\I_1(a_1)) \otimes \s^{-2}), \\ 
&c_{i_3}(T_{S^{[n]}_2}), \ c_{i_4}^{\C^*}(L(a_1,c_1)).
\end{align*}
Of course we assume that only $\C^*$-moving terms appear in the denominator of $X$. Then there exists a polynomial $Y$ in $a_1^2$, $a_1 c_1$, $c_1^2$, $a_1 K_S$, $c_1 K_S$, $K_S^2$, $\chi(\O_S)$ with coefficients in $\Q[y](\!(s)\!)$ such that
$$
\int_{S_2^{[n]}} X = Y, 
$$
for any smooth projective surface $S$ and $a_1, c_1 \in A^1(S)$. This is essentially \cite[Lem.~5.5]{GNY1}, which in turn is an adaptation of \cite{EGL}. 

We conclude that for each $n \geq 0$, $i \in \Z$, there exists a universal polynomial $Y_{n,i}$ in $a_1^2$, $a_1 c_1$, $c_1^2$, $a_1 K_S$, $c_1 K_S$, $K_S^2$, $\chi(\O_S)$ such that 
$$
\sfZ_{S}(a_1,c_1,s,y,q) = \sum_{n \geq 0} \sum_{i} Y_{n,i} s^i q^n,
$$
for all $S,a_1,c_1$. The coefficient of $q^0$ is 1 by \eqref{leadingcoeff}. Hence there exists a universal power series
$$
G \in \Q[y,x_1, \ldots, x_7](\!(s)\!)[[q]],
$$
such that
\begin{equation} \label{expG}
\sfZ_{S}(a_1,c_1,s,y,q) = \exp \, G(a_1^2, a_1 c_1, c_1^2, a_1 K_S, c_1 K_S, K_S^2, \chi(\O_S)) 
\end{equation}
for all $S,a_1,c_1$. \\

\noindent \textbf{Step 2: Multiplicativity.} Let $S = S' \sqcup S''$, where $S', S''$ are not necessarily connected smooth projective surfaces. Let $a_1, c_1 \in A^1(S)$ be such that
$$
a_1|_{S'} = a_1', \ c_1|_{S'} = c_1', \ a_1|_{S''} = a_1'', \ c_1|_{S''} = c_1''.
$$ 
Then
$$
S_{2}^{[n]} = \bigsqcup_{n_1+n_2 = n} S_{2}^{\prime [n_1]} \times S_{2}^{\prime \prime [n_2]}.
$$
Let $\I_1 = \I_1' \boxplus \I_1''$, $\I_2 = \I_2' \boxplus \I_2''$. As observed in \cite{GNY1}, we then have
\begin{align*}
&R \hom_\pi (\I_1(a_1), \I_2(c_1-a_1)) = \\
&R \hom_\pi (\I_1'(a_1'), \I_2'(c_1'-a_1')) \boxplus R \hom_\pi (\I_1''(a_1''), \I_2''(c_1''-a_1''))
\end{align*}
and similarly with $\I_1$ and $\I_2$ interchanged. Furthermore
\begin{align*}
T_{S_{2}^{[n]}} &= \bigoplus_{n_1+n_2=n} T_{S_{2}^{\prime [n_1]}} \boxplus T_{S_{2}^{\prime \prime [n_2]}} \\
L(a_1,c_1) &= L(a_1',c_1') \boxplus L(a_1'',c_1''), \\
\chi(c_1-2a_1) &= \chi(c_1'-2a_1') + \chi(c_1''-2a_1''), \\
\chi(2a_1-c_1) &= \chi(2a_1'-c_1') + \chi(2a_1''-c_1''), \\
a_1 (c_1-a_1) &= a_1' (c_1' - a_1') + a_1'' (c_1''-a_1'').
\end{align*}
Note that $\Eu(\cdot)$ and $\sfT^{\C^*}_{-y}(\cdot,1-y)$ are both group homomorphisms from $(K^{0}_{\C^*}(M),+)$ to $(A^*(M)_{\Q}[y](\!(s)\!),\cdot)$, where ``$\C^*$'' stands for ``$\C^*$-equivariant''. For $\sfT^{\C^*}_{-y}(\cdot,1-y)$ this follows from the crucial multiplicative property \eqref{Tprod}. Therefore
\begin{equation} \label{mult}
\sfZ_{S}(a_1,c_1,s,y,q) = \sfZ_{S'}(a_1',c_1',s,q) \, \sfZ_{S''}(a_1'',c_1'',s,y,q).
\end{equation}

We combine \eqref{expG} and \eqref{mult} in order to construct the universal functions $A_i(s,y,q)$. This follows from a cobordism argument similar to \cite[Lem.~5.5]{GNY1}. Take seven triples $(S^{(i)},a_1^{(i)},c_1^{(i)})$ such that the vectors 
$$
w_i := ((a_1^{(i)})^2,a_1^{(i)}c_1^{(i)},(c_1^{(i)})^2,a_1^{(i)}K_{S^{(i)}},c_1^{(i)}K_{S^{(i)}},K_{S^{(i)}}^2,\chi(\O_{S^{(i)}})) \in \Q^{7}
$$
form a $\Q$-basis. Now consider an arbitrary triple $(S,a_1,c_1)$. Then we can decompose $w = (a_1^2, \ldots, \chi(\O_S))$ as
$w = \sum_i n_i w_i$ for some $n_i \in \Q$.
If all $n_i \in \Z_{\geq 0}$, then \eqref{mult} implies 
\begin{equation} \label{intermed}
\sfZ_S(a_1,c_1,s,y,q) = \prod_{i=1}^{7} \big(\exp G(w_i) \big)^{n_i} = \exp \Big( \sum_{i=1}^{7} n_i G(w_i) \Big).
\end{equation}
Denote by $W$ the matrix with column vectors $w_1, \ldots, w_{7}$ and let $M = (m_{ij})$ be its inverse. We define 
$$
A_j(s,y,q) := \exp\Big(\sum_i m_{ij} G(w_i)\Big), \quad \forall j=1, \ldots, 7.
$$ 
From \eqref{intermed} we obtain
\begin{align} 
\begin{split} \label{ZA}
\sfZ_S(a_1,c_1,s,y,q) = 
A_1^{a_1^2} \, A_2^{a_1 c_1} \, A_3^{c_1^2} \, A_4^{a_1 K_S} \, A_5^{c_1 K_S} \, A_6^{K_S^2} \, A_7^{\chi(\O_S)}.
\end{split}
\end{align}
Since the points $w = \sum_i n_i w_i$, with $n_i \in \Z_{\geq 0}$, lie Zariski dense in $\Q^{7}$, we conclude that \eqref{ZA} holds for \emph{all} triples $(S,a_1,c_1)$.
\end{proof}

For a $7$-tuple $\underline \alpha=(\alpha_1,\alpha_2,\alpha_3,\alpha_4,\alpha_5,\alpha_6,\alpha_7) \in \Z^7$ we denote
\begin{align}
\begin{split} \label{defsfA}
\sfA_{\underline \alpha }(s,y,q)
:=&- 2 \Bigg( 2^{-1} \Bigg( \frac{1-e^{-2s(1-y)}}{1-y e^{-2s(1-y)}} \Bigg)^{2} \Bigg( \frac{1-e^{2s(1-y)}}{1-y e^{2s(1-y)}} \Bigg)^{2} q^{-1} A_1(s,y,q/s) \Bigg)^{\alpha_1}   \\
&\times \Bigg( 2 \Bigg( \frac{1-e^{-2s(1-y)}}{1-y e^{-2s(1-y)}} \Bigg)^{-2} \Bigg( \frac{1-e^{2s(1-y)}}{1-y e^{2s(1-y)}} \Bigg)^{-2} q A_2(s,y,q/s)  \Bigg)^{\alpha_2}   \\
&\times \Bigg( 2^{-\frac{1}{2}} \Bigg( \frac{1-e^{-2s(1-y)}}{1-y e^{-2s(1-y)}} \Bigg)^{\frac{1}{2}} \Bigg( \frac{1-e^{2s(1-y)}}{1-y e^{2s(1-y)}}\Bigg)^{\frac{1}{2}} A_3(s,y,q/s)  \Bigg)^{\alpha_3} \\
&\times \Bigg(  \Bigg( \frac{1-e^{-2s(1-y)}}{1-y e^{-2s(1-y)}} \Bigg) \Bigg( \frac{1-e^{2s(1-y)}}{1-y e^{2s(1-y)}} \Bigg)^{-1} A_4(s,y,q/s)  \Bigg)^{\alpha_4}    \\
&\times \Bigg( 2^{\frac{1}{2}} \Bigg( \frac{1-e^{-2s(1-y)}}{1-y e^{-2s(1-y)}} \Bigg)^{-\frac{1}{2}} \Bigg( \frac{1-e^{2s(1-y)}}{1-y e^{2s(1-y)}} \Bigg)^{\frac{1}{2}} A_5(s,y,q/s)  \Bigg)^{\alpha_5}  \\
&\times A_6(s,y,q/s) ^{\alpha_6}   \\
&\times \Bigg( \frac{s}{2} \Bigg( \frac{1-e^{-2s(1-y)}}{1-y e^{-2s(1-y)}} \Bigg) \Bigg( \frac{1-e^{2s(1-y)}}{1-y e^{2s(1-y)}} \Bigg) A_7(s,y,q/s)  \Bigg)^{\alpha_7}.
\end{split}
\end{align}

\begin{corollary} \label{univcor}
Suppose $S$ satisfies $b_1(S)=0$ and $p_g(S) > 0$. Let $H, c_1, c_2$ be chosen such that there exist no rank 2 strictly Gieseker $H$-semistable sheaves with Chern classes $c_1, c_2$. Assume furthermore that: 
\begin{itemize}
\item[(i)] $c_2 <  \frac{1}{2} c_{1}(c_1-K_S) + 2\chi(\O_S)$.
\item[(ii)] $p_{\ch} > p_{K_S}$, where $p_{\ch}$ and $p_{K_S}$ are the reduced Hilbert polynomials associated to $\ch = (2,c_1,\frac{1}{2}c_1^2-c_2)$ and $K_S$.
\item[(iii)] For all SW basic classes $a_1$ satisfying $a_1 H \leq (c_1 -a_1) H$ the inequality is strict. 
\end{itemize}
Then
$$\chi_{-y}^{\vir}(M_{S}^{H}(2,c_1,c_2))=\Coeff_{s^0 q^{c_2}}
\Big[ \sum_{{\scriptsize{\begin{array}{c} a_1 \in H^2(S,\Z) \\ a_1 H < (c_1-a_1) H \end{array}}}} \SW(a_1) \, 
\sfA_{(a_1^2,a_1 c_1,c_1^2,a_1 K_S, c_1 K_S, K_S^2, \chi(\O_S))}(s,y,q)\Big].$$
\end{corollary}
\begin{proof}
This follows from Remark \ref{Mochcor}, equation \eqref{keyexpr}, and Proposition \ref{univ}. 
\end{proof}
\begin{remark} \label{strongform}
By Remark \ref{assumpmocthm}, we conjecture that this corollary holds without assuming (ii) and (iii) and that the sum can be replaced by a sum over \emph{all} Seiberg-Witten basic classes. We refer to this as ``the strong form of Mochizuki's formula''.
\end{remark}

\section{Toric calculation} \label{toricsec}

The universal functions $A_1, \ldots, A_7$ are entirely determined by the generating function $\sfZ_S(a_1,c_1,s,y,q)$ in the following cases
\begin{align*}
(S,a_1,c_1) = \, &(\PP^2,\O,\O), (\PP^2,\O(1),\O(1)), (\PP^2,\O,\O(1)), (\PP^2,\O(1),\O(2)), \\
&(\PP^1 \times \PP^1, \O,\O), (\PP^1 \times \PP^1, \O(1,0),\O(1,0)), (\PP^1 \times \PP^1, \O,\O(1,0)).
\end{align*}
For these choices the $7 \times 7$ matrix with rows 
$$
(a_1^2,a_1 c_1,c_1^2,a_1 K_S, c_1 K_S, K_S^2, \chi(\O_S))
$$
has full rank. In each of these cases $S$ is a toric surface. 

Suppose $S$ is a toric surface with torus $T = \C^{*2}$. Let $\{U_\sigma\}_{\sigma = 1, \ldots, e(S)}$ be the cover of maximal $T$-invariant affine open subsets of $S$. On $U_\sigma$ we use coordinates $x_\sigma, y_\sigma$ such that $T$ acts with characters of weight $v_\sigma, w_\sigma \in \Z^2$
$$
t \cdot (x_\sigma,y_\sigma) = (\chi(v_\sigma)(t) \, x_\sigma, \chi(w_\sigma)(t) \, y_\sigma).
$$
Here $\chi(m) : T \rightarrow \C^*$ denotes the character of weight $m \in \Z^2$. Consider the integral over $S^{[n_1]} \times S^{[n_2]}$ of Definition \ref{defZX}
\begin{equation} \label{HilbxHilb}
\int_{S^{[n_1]} \times S^{[n_2]}} \frac{\sfT_{-y}^{\C^*}(E_{n_1,n_2},1-y) \, \Eu(\O(a_1)^{[n_1]}) \, \Eu(\O(c_1-a_1)^{[n_2]} \otimes \s^2)}{\Eu\big(E_{n_1,n_2}-\pi_1^*T_{S^{[n_1]}}-\pi_2^*T_{S^{[n_2]}}\big)}.
\end{equation}
Define $\widetilde{T} := T \times \C^*$, where $\C^*$ denotes the trivial torus factor of Section \ref{sec1}. The $T$-fixed locus of $S^{[n_1]} \times S^{[n_2]}$ is indexed by pairs $(\bslambda,\bsmu)$ with
$$
\bslambda = \{\lambda^{(\sigma)}\}_{\sigma=1, \ldots, e(S)}, \ \bsmu = \{\mu^{(\sigma)}\}_{\sigma=1, \ldots, e(S)},
$$
where $\lambda^{(\sigma)}$, $\mu^{(\sigma)}$ are partitions such that
\begin{equation} \label{size}
\sum_\sigma |\lambda^{(\sigma)}| = n_1, \ \sum_\sigma |\mu^{(\sigma)}| = n_2.
\end{equation}
Here $|\lambda|$ denotes the size of $\lambda$. A partition $\lambda = (\lambda_1 \geq  \cdots \geq  \lambda_\ell)$ corresponds to a monomial ideal of $\C[x,y]$
$$
I_{Z_\lambda} := (y^{\lambda_1}, xy^{\lambda_2}, \ldots, x^{\ell-1}y^{\lambda_\ell}, x^{\ell}),
$$
where $\ell(\lambda) = \ell$ denotes the length of $\lambda$. For $\lambda^{(\sigma)}$ we denote the subscheme defined by the corresponding monomial ideal in variables $x_\sigma,y_\sigma$ by $Z_{\lambda^{(\sigma)}}$ and similarly for $\mu^{(\sigma)}$. 

In order to apply localization, we make a choice of $T$-equivariant structure on the line bundles $\O(a_1)$, $\O(c_1-a_1)$. For any $T$-equivariant divisor $a$, the restriction $\O(a)|_{U_{\sigma}}$
is trivial with $T$-equivariant structure determined by some character of weight $a_\sigma \in \Z^2$. By Atiyah-Bott localization, the integral \eqref{HilbxHilb} equals
\begin{align*}
&\sum_{(\bslambda,\bsmu)}  \prod_\sigma \frac{\Eu(H^0(\O(a_1)|_{Z_{\lambda^{(\sigma)}}}))}{\Eu(T_{Z_{\lambda^{(\sigma)}}})} \frac{\Eu(H^0(\O(c_1-a_1)|_{Z_{\mu^{(\sigma)}}}) \otimes \s^2)}{\Eu(T_{Z_{\mu^{(\sigma)}}})} \\
&\times \frac{\sfT^{\widetilde{T}}_{-y}(E_{n_1,n_2}|_{(Z_{\lambda^{(\sigma)}},Z_{\mu^{(\sigma)}})},1-y)}{\Eu(E_{n_1,n_2}|_{(Z_{\lambda^{(\sigma)}},Z_{\mu^{(\sigma)}})} - T_{Z_{\lambda^{(\sigma)}}} - T_{Z_{\mu^{(\sigma)}}})}.
\end{align*}
Here $\Eu(\cdot)$ denotes $\widetilde{T}$-equivariant Euler class, $\sfT^{\widetilde{T}}$ is the $\widetilde{T}$-equivariant version of \eqref{defT}, and the sum is over all $(\bslambda,\bsmu)$ satisfying \eqref{size}. Moreover, $T_{Z_{\lambda}}$ denotes the $T$-equivariant Zariski tangent space of $(\C^{2})^{[n]}$ at $Z_\lambda$ where $n = |\lambda|$. 
The calculation of the above product reduces to the computation of the following elements of the $T$-equivariant $K$-group $K_0^{T}(pt)$
\begin{align*}
&H^0(\O(a)|_{Z_{\lambda^{(\sigma)}}}), \\
&R\Hom_S(\O_{Z_{\lambda^{(\sigma)}}}, \O_{Z_{\lambda^{(\sigma)}}}), \\
&R\Hom_S(\O_{Z_{\lambda^{(\sigma)}}}, \O_{Z_{\mu^{(\sigma)}}}(a)),
\end{align*}
for various $T$-equivariant divisors $a$. By definition we have
$$
Z_{\lambda^{(\sigma)}} = \sum_{i=0}^{\ell(\lambda^{(\sigma)}) - 1} \sum_{j=0}^{\lambda^{(\sigma)}_{i+1}-1} \chi(v_\sigma)^i \, \chi(w_\sigma)^j
$$
and similarly for $Z_{\mu^{(\sigma)}}$. Multiplying by $\chi(a_\sigma)$ gives $H^0(\O(a)|_{Z_{\lambda^{(\sigma)}}})$. Define 
$$
\overline{\chi(m)} := \chi(-m) = \frac{1}{\chi(m)},
$$
for any $m \in \Z^2$. This defines an involution on $K_0^{T}(pt)$ by $\Z$-linear extension. 
\begin{proposition}
Let $W,Z \subset S$ be 0-dimensional $T$-invariant subschemes supported on a chart $U_\sigma \subset S$ and let $a$ be a $T$-equivariant divisor on $S$ corresponding to a character of weight $a_\sigma \in \Z^2$ on $U_\sigma$. Then we have the following equality in $K_0^{T}(pt)$
$$
R\Hom_S(\O_W, \O_Z(a)) = \chi(a_\sigma) \, \overline{W} Z \frac{(1-\chi(v_\sigma)) (1- \chi(w_\sigma))}{\chi(v_\sigma) \chi(w_\sigma)}.
$$
\end{proposition}
\begin{proof}
The proof we present is similar to the calculation in \cite[Sect.~4.7]{MNOP}. Let $v := v_\sigma$, $w := w_\sigma$, and $a := a_\sigma$. Write $U_\sigma = \Spec R$ with $R = \C[x_\sigma,y_\sigma]$. Then
$$
R\Hom_S(\O_W, \O_Z(a)) = R \Hom_{U_\sigma}(\O_{W}, \O_{Z}(a)),
$$
because $W,Z$ are supported on $U_\sigma$. We claim
\begin{align*}
&\Gamma(U_\sigma, \O(a)) - R \Hom_{U_\sigma}(I_W,I_Z(a)) = \chi(a) \Big( Z + \frac{\overline{W}}{\chi(v) \chi(w)} - \overline{W} Z \frac{(1-\chi(v)) (1- \chi(w))}{\chi(v) \chi(w)}  \Big).
\end{align*}
The result of the proposition follows from this using $I_Z = \O_{U_\sigma} - \O_{Z}$, $I_W = \O_{U_\sigma} - \O_{W}$, because the first term on the right-hand side is $H^0(U_\sigma,\O_Z(a))$ and the second term is
\begin{align*}
R\Hom_{U_\sigma}(\O_W,\O_{U_\sigma}(a)) &= H^0(U_\sigma,\O_W(-a) \otimes K_{U_\sigma})^* \\
&= H^0(U_\sigma,\O_W)^* \otimes \frac{\chi(a)}{\chi(v)\chi(w)}
\end{align*}
by $T$-equivariant Serre duality. In order to prove the claim, choose $T$-equivariant graded free resolutions
\begin{align*}
0 \rightarrow E_r &\rightarrow \cdots \rightarrow E_0 \rightarrow I_{W} \rightarrow 0, \\
0 \rightarrow F_s &\rightarrow \cdots \rightarrow F_0 \rightarrow I_{Z} \rightarrow 0,
\end{align*}
where 
\begin{align*}
E_i = \bigoplus_j R(d_{ij}), \ F_i = \bigoplus_j R(e_{ij}).
\end{align*}
Then we have Poincar\'e polynomials
$$
P_W = \sum_{i,j} (-1)^i \chi(d_{ij}), \ P_Z = \sum_{i,j} (-1)^i \chi(e_{ij}),
$$
which are independent of the choice of resolution. Moreover
\begin{align}
\begin{split} \label{WZ}
W &= \O_{U_\sigma} - I_W \\
&= \frac{1-P_W}{(1-\chi(v)) (1-\chi(w))}, \\
Z &= \O_{U_\sigma} - I_Z \\
&= \frac{1-P_Z}{(1-\chi(v)) (1-\chi(w))}.
\end{split}
\end{align}
Furthermore
\begin{align*}
R \Hom_{U_\sigma}(I_W,I_Z(a)) &= \sum_{i,j,k,l} (-1)^{i+k} \Hom(R(d_{ij}),R(a+e_{kl})) \\
&= \sum_{i,j,k,l} (-1)^{i+k} R(a+e_{kl} - d_{ij}) \\
&= \frac{\chi(a) \overline{P}_W P_Z}{(1-\chi(v))(1-\chi(w))}.
\end{align*}
Eliminating $P_W$, $P_Z$ using \eqref{WZ} gives the desired result.
\end{proof}
We implemented the calculation of \eqref{HilbxHilb} into a PARI/GP program (and some parts into Maple as well). This allows us to compute $A_1(s,y,q), \ldots, A_7(s,y,q)$ up to order $q^7$, where we calculated the coefficient of $q^i$ up to order $s^{29-3i}$ . We also calculated $A_1(s,1,q), \ldots, A_7(s,1,q)$ up to order $q^{30}$ and any order in $s$. The latter are listed up to order $q^4$ in Appendix \ref{app1}.

\section{Two more conjectures and consequences}

In Section \ref{toricsec}, we have given a toric procedure to calculate the universal functions $A_1(s,y,q), \ldots, A_7(s,y,q)$ and therefore also $\sfA_{\underline{\alpha}}(s,y,q)$ defined by \eqref{defsfA}. Consequently, we could now go ahead and provide checks of Conjecture \ref{conj}. 

Instead we first present two generalizations of Conjecture \ref{conj}. The first conjecture is a statement about intersection numbers on Hilbert schemes of points. It implies a formula for arbitrary blow-ups of surfaces $S$ with $b_1(S)=0$, $p_g(S)>0$, and Seiberg-Witten basic classes $0$ and $K_S \neq 0$. The second conjecture generalizes Conjecture \ref{conj} to \emph{arbitrary} surfaces $S$ with $b_1(S)=0$ and $p_g(S)>0$. It implies a blow-up formula, which is reminiscent of the blow-up formula of W.-P.~Li and Z.~Qin \cite{LQ1,LQ2}. It also implies a formula for surfaces with canonical divisor with irreducible reduced connected components. The latter refines a result from the physics literature due to Vafa-Witten \cite[Eqn.~(5.45)]{VW}.

\subsection{Numerical conjecture}

Suppose $S$ is a surface with $b_1(S) = 0$, $p_g(S)>0$, and Seiberg-Witten basic classes are $0$ and $K_S \neq 0$. Then Conjecture \ref{conj} applies to any choice of
$$
\underline{\beta} = (\beta_1, \beta_2,\beta_3,\beta_4) = (c_1^2, c_1 K_S, K_S^2, \chi(\O_S)),
$$
provided we choose a polarization $H$ for which there are no rank 2 strictly Gieseker $H$-semistable sheaves with Chern classes $c_1, c_2$. Moreover, as long as the assumptions of Corollary \ref{univcor} are satisfied, the coefficients of $\sfZ_{S,c_1}(x,y)$ are calculated by the universal functions $A_1, \ldots, A_7$. This raises the expectation that for \emph{any} choice of the 4-tuple $\underline{\beta} \in \Z^4$ the formula of Corollary \ref{univcor} is determined by the coefficients of the modular form of Conjecture \ref{conj}. This turns out to be false. Computer calculations show that we need to impose\footnote{The inequality $K_S^2 \geq \chi(\O_S) - 3$ is reminiscent of the (stronger) Noether's inequality $K_S^2 \geq 2(\chi(\O_S)-3)$ for minimal surfaces of general type.}
$$
\beta_3 \geq \beta_4 - 3.
$$
Indeed let $\underline{\beta} \in \Z^4$ with $\beta_1 \equiv \beta_2 \mod 2$ and $\beta_3 \geq \beta_4 - 3$. Let $n < \frac{1}{2}(\beta_1-\beta_2)+2\beta_4$.
We conjecture that  
\begin{align*}
\Coeff_{s^0x^{4n-\beta_1-3\beta_4}} \Big[&(x y^{-\frac{1}{2}})^{-\beta_1-3\beta_4} \sfA_{(0,0,\beta_1,0,\beta_2,\beta_3,\beta_4)}(s,y,(x y^{-\frac{1}{2}})^4) \\
&+ (-1)^{\beta_4} (x y^{-\frac{1}{2}})^{-\beta_1-3\beta_4} \sfA_{(\beta_3,\beta_2,\beta_1,\beta_3,\beta_2,\beta_3,\beta_4)}(s,y,(x y^{-\frac{1}{2}})^4) \Big]
\end{align*}
equals the coefficient of $x^{4n-\beta_1-3\beta_4}$ of 
$$
8 \Bigg( \frac{1}{2} \prod_{n=1}^{\infty} \frac{1}{(1-x^{2n})^{10}(1-x^{2n} y)(1-x^{2n} y^{-1})} \Bigg)^{\beta_4} \Bigg( \frac{2 \overline{\eta}(x^4)^2}{\theta_3(x,y^{\frac{1}{2}})} \Bigg)^{\beta_3}.
$$
In fact, we have a stronger conjecture, which arose by attempts to generalize Conjecture \ref{conj} to blow-ups.
\begin{conjecture} \label{numconj}
Let $\underline{\beta} \in \Z^4$ be such that $\beta_1 \equiv \beta_2 \mod 2$ and $\beta_3 \geq \beta_4 - 3$. Let $n < \frac{1}{2}(\beta_1-\beta_2)+2\beta_4$. Let $(\gamma_1,\gamma_2)\in \Z^2$. 
Then  
\begin{align*}\Coeff_{s^0x^{4n-\beta_1-3\beta_4}} \Big[&(x y^{-\frac{1}{2}})^{-\beta_1-3\beta_4} \sfA_{(\gamma_1,\gamma_2,\beta_1,\gamma_1,\beta_2,\beta_3,\beta_4)}(s,y,(x y^{-\frac{1}{2}})^4) \\
&+ (-1)^{\beta_4} (x y^{-\frac{1}{2}})^{-\beta_1-3\beta_4} \sfA_{(\beta_3-\gamma_1,\beta_2-\gamma_2,\beta_1,\beta_3-\gamma_1,\beta_2,\beta_3,\beta_4)}(s,y,(x y^{-\frac{1}{2}})^4) \Big]
\end{align*}
equals the coefficient of $x^{4n-\beta_1-3\beta_4}$ of 
\begin{align*}
\psi_{\gamma_1, \gamma_2, \beta_3, \beta_4}(x,y) := 8 (-1)^{\gamma_2}\Bigg( \frac{\phi(x,y)}{2}\Bigg)^{\beta_4} \Bigg( \frac{2 \overline{\eta}(x^4)^2}{\theta_3(x,y^{\frac{1}{2}})} \Bigg)^{\beta_3}\Bigg(\frac{\theta_3(x,y^{\frac{1}{2}})}{\theta_3(-x,y^{\frac{1}{2}})} \Bigg)^{\gamma_1},
\end{align*}
where 
$$
\phi(x,y) :=\prod_{n=1}^{\infty} \frac{1}{(1-x^{2n})^{10}(1-x^{2n} y)(1-x^{2n} y^{-1})}.
$$
\end{conjecture}

The evidence for this conjecture is presented in Section \ref{numconjevidence}. We now discuss its consequences. 

\begin{remark} Let  $\beta_1\in \Z$ be even, $\beta_4\in \Z_{\le 3}$, and $n < \frac{1}{2}\beta_1+2\beta_4$. Then we conjecture 
\begin{align} \label{LHSK3}
\Coeff_{s^0x^{4n-\beta_1-3\beta_4}} \Big[& (x y^{-\frac{1}{2}})^{-\beta_1-3\beta_4} \sfA_{(0,0,\beta_1,0,0,0,\beta_4)}(s,y,(x y^{-\frac{1}{2}})^4) \Big]
\end{align}
equals the coefficient of $x^{4n-\beta_1-3\beta_4}$ of 
\begin{align} \label{RHSK3}
4 \Bigg( \frac{\phi(x,y}{2} \Bigg)^{\beta_4}.
\end{align}
In the case $\beta_4$ is even this follows from Conjecture \ref{numconj} by taking $\gamma_1=\gamma_2=\beta_2=\beta_3 = 0$ because then the two summands on the left-hand side of the conjecture are equal. In the case $\beta_4$ is odd this says that \eqref{LHSK3} is $0$, because \eqref{RHSK3} only contains even powers of $x$.
\end{remark}
\begin{proposition} \label{numconjimpliesconj}
Conjecture \ref{numconj} and the strong form of Mochizuki's formula (Remark \ref{strongform}) imply Conjecture \ref{conj}.
\end{proposition}
\begin{proof}
We only use Conjecture \ref{numconj} for $\gamma_1=\gamma_2 = 0$. Suppose $S$ is a surface with $b_1(S) = 0$, $p_g(S)>0$, and Seiberg-Witten basic classes $0$ and $K_S \neq 0$. Choose a polarization $H$ and $c_1, c_2$ such that there are no rank 2 strictly Gieseker $H$-semistable sheaves with these Chern classes. By taking 
$$
(\beta_1, \beta_2,\beta_3,\beta_4) = (c_1^2, c_1K_S, K_S^2, \chi(\O_S))
$$
we automatically satisfy $\beta_1 \equiv \beta_2 \mod 2$. The fact that $\beta_3 \geq \beta_4 - 3$ can be seen as follows. If $S$ is not minimal then it is the blow-up of a $K3$ in one point and the inequality is trivial. If $S$ is minimal then it is minimal properly elliptic or minimal general type because $p_g(S)>0$. For minimal general type the inequality follows from Noether's inequality $K_S^2 \geq 2(\chi(\O_S) - 3)$. When $\pi : S \rightarrow B$ is minimal properly elliptic, we have $B \cong \PP^1$ because $b_1(S) = 0$. The canonical bundle $K_S$ satisfies $K_S^2 = 0$ and can be represented by an effective divisor containing $\pi^* D$, where $D \subset \PP^1$ is some effective divisor of degree 
$$
\chi(\O_S)-2 \geq 0. 
$$
Indeed $\chi(\O_S) \leq 3$, because otherwise at least $0, F, K_S$ are distinct Seiberg-Witten basic classes \cite{FM}. By the strong form of Mochizuki's formula (Corollary \ref{univcor} and Remark \ref{strongform}), Conjecture \ref{conj} follows for all
$$
\vd < c_1^2 - 2 c_1 K_S + 5 \chi(\O_S).
$$
If this inequality is not satisfied, then we replace $c_1$ by $c_1 + t H$ for some $t>0$. Since
$$
\sfZ_{S,c_1}(x,y) = \sfZ_{S,c_1 + 2 t H}(x,y)
$$
we can compute the coefficients of this generating function for all
$$
\vd < c_1^2 - 2 c_1 K_S + 5 \chi(\O_S) + 4 t^2 H^2 + 4 t H (c_1 - K_S).
$$
By choosing $t \gg 0$ the bound becomes arbitrarily high.
\end{proof}

\subsection{Fixed first Chern class} \label{fixedc1sec}

\begin{proposition} \label{fixedc1propnum}
Assume Conjecture \ref{numconj}. Let $\underline{\beta} \in \Z^4$ be such that $\beta_1 \equiv \beta_2 \mod 2$ and $\beta_3 \geq \beta_4 - 3$ and let $(\gamma_1,\gamma_2)\in \Z^2$. Then 
\begin{align*}
&\Coeff_{s^0} \Big[ (x y^{-\frac{1}{2}})^{-\beta_1-3\beta_4} \sfA_{(\gamma_1,\gamma_2,\beta_1,\gamma_1,\beta_2,\beta_3,\beta_4)}(s,y,(x y^{-\frac{1}{2}})^4) \\
&\qquad\quad \ \, + (-1)^{\beta_4} (x y^{-\frac{1}{2}})^{-\beta_1-3\beta_4} \sfA_{(\beta_3-\gamma_1,\beta_2-\gamma_2,\beta_1,\beta_3-\gamma_1,\beta_2,\beta_3,\beta_4)}(s,y,(x y^{-\frac{1}{2}})^4) \Big]\\
&=2(-1)^{\gamma_2}\sum_{k=0}^3 (i^k)^{\beta_1-\beta_4} \Bigg(\frac{\phi(i^k x,y)}{2}\Bigg)^{\beta_4}\Bigg(\frac{2\overline{\eta}(x^4)^2}{\theta_3(i^k x,y^{\frac{1}{2}})}\Bigg)^{\beta_3}\Bigg(\frac{\theta_3(i^k x,y^{\frac{1}{2}})}{\theta_3(-i^k x,y^{\frac{1}{2}})}\Bigg)^{\gamma_1} \\
&\quad +O(x^{\beta_1-2\beta_2+5\beta_4}),
\end{align*}
where $i = \sqrt{-1}$ and
$$
\phi(x,y) :=\prod_{n=1}^{\infty} \frac{1}{(1-x^{2n})^{10}(1-x^{2n} y)(1-x^{2n} y^{-1})}.
$$
\end{proposition}
\begin{proof}
Recall the formula for $\psi(x,y) := \psi_{\gamma_1,\gamma_2,\beta_3,\beta_4}(x,y)$ of Conjecture \ref{numconj}. We see that the term in the sum on the right-hand side corresponding to $k=0$ equals $\psi(x,y)/4$. Define coefficients $f_n(y)$ by
$$
\psi(x,y) = \sum_{n=0}^{\infty} f_n(y) \, x^n.
$$
Then the right-hand side of the formula of the proposition equals
\begin{align*}
\sum_{k=0}^{3} \frac{(i^k)^{\beta_1 - \beta_4}}{4} \psi(i^k x,y) &=\sum_{k=0}^{3} \sum_{n=0}^{\infty} \frac{(i^k)^{\beta_1 - \beta_4 + n}}{4} f_n(y) \, x^n  \\
&=\sum_{n=0}^{\infty} \Bigg( \frac{1}{4}\sum_{k=0}^{3} \big(i^{\beta_1 - \beta_4 + n}\big)^k\Bigg) f_n(y) \, x^n \\
&= \sum_{n \equiv -\beta_1 + \beta_4 \mod 4}  f_n(y) \, x^n.
\end{align*}
Therefore we conclude that the right-hand side of the formula of the proposition is obtained from $\psi(x,y)$ by extracting all terms $x^n$ for which $n \equiv -\beta_1 - 3 \beta_4 \mod 4$ and up to order  $O(x^{\beta_1-2\beta_2+5\beta_4})$. The result follows from Conjecture \ref{numconj}.
\end{proof}

The same type of proof applied to Conjecture \ref{conj} implies the following.
\begin{proposition} \label{fixedc1prop}
Assume Conjecture \ref{conj}. Let $S$ be a smooth projective surface with $b_1(S) = 0$ and $p_g(S)>0$. Suppose the Seiberg-Witten basic classes of $S$ are $0$ and $K_S \neq 0$. Let $H, c_1$ be chosen such that there are no rank 2 strictly Gieseker $H$-semistable sheaves with first Chern class $c_1$. Then
\begin{align*} 
\sfZ_{S,c_1}(x,y) = 2\sum_{k=0}^3 (i^k)^{c_1^2 - \chi(\O_S)} \Bigg( \frac{\phi(i^k x,y)}{2} \Bigg)^{\chi(\O_S)} \Bigg(\frac{2\overline{\eta}(x^4)^2}{\theta_3(i^k x,y^{\frac{1}{2}})} \Bigg)^{K_S^2}.
\end{align*}
In particular, Corollary \ref{cor} in the introduction follows.
\end{proposition}

Using Conjecture \ref{numconj} we can do better:
\begin{proposition} \label{nonmingt}
Assume Conjecture \ref{numconj} and the strong form of Mochizuki's formula (Remark \ref{strongform}). 
Let $S_0$ be a smooth projective surface with $b_1(S_0) = 0$ and $p_g(S_0)>0$. Suppose the Seiberg-Witten basic classes of $S_0$ are $0$ and $K_{S_0} \neq 0$. Suppose $S$ is obtained from an iterated blow-up (possibly at infinitely near points) of $S_0$ and let $E_1,\cdots, E_m$ denote the total transforms of the exceptional divisors. Suppose furthermore that $K_S^2\ge \chi(\O_S)-3$. Let $H, c_1$ be chosen such that there exist no rank 2 strictly Gieseker $H$-semistable sheaves on $S$ with first Chern class $c_1$. Then
\begin{align*}
&\sfZ_{S,c_1}(x,y) = \\
&2\sum_{k=0}^3 (i^k)^{c_1^2-\chi(\O_S)}\Bigg(\frac{\phi(i^k x,y)}{2}\Bigg)^{\chi(\O_S)}\Bigg(\frac{2\overline{\eta}(x^4)^2}{\theta_3(i^k x,y^{\frac{1}{2}})}\Bigg)^{K_{S}^2} \prod_{j=1}^m\Bigg(1+(-1)^{c_1E_j} \frac{\theta_3(-i^kx,y^{\frac{1}{2}})}{\theta_3(i^kx,y^{\frac{1}{2}})}\Bigg).
\end{align*}
\end{proposition}
\begin{proof} 
The surface $S$ is obtained from $S_0$ by an iterated blow-up in (possibly infinitely near) points $\pi : S \rightarrow S_0$. We denote by $E_1, \ldots, E_m$ the total transforms on the blow-ups. Write $M:=\{1,\ldots,m\}$, and for a subset $I\subset M$ write $E_I=\sum_{i\in I} E_i$. Then $K_S=K_{S_0}+E_M$. Moreover the Seiberg-Witten basic classes are the $E_I$ (with Seiberg-Witten invariant $1$) and the $K_{S_0}+E_I=K_S-E_{M-I}$ (with Seiberg-Witten invariant $(-1)^{\chi(\O_S)}$)
 for all $I\subset M$. E.g.~see \cite[Thm.~7.4.6]{Mor} for Seiberg-Witten invariants of blow-ups. We denote by $|I|$ the number of elements of $I$. Now note that 
 $$
 K_S(K_S-E_I)=(K_S-E_I)^2=K_S^2-E_I^2.
 $$
 By the strong form of Mochizuki's formula (Corollary \ref{univcor} and Remark \ref{strongform}) and Proposition \ref{fixedc1propnum}, we obtain the following equation modulo $x^{c_1^2-2c_1K_S+5\chi(\O_S)}$ 
 \begin{align*}
 &\sfZ_{S,c_1}(x,y) =\sum_{I\subset M}\Coeff_{s^0}\Big[\Big((x y^{-\frac{1}{2}})^{-c_1^2-3\chi(\O_S)}\sfA_{(E_I^2,E_Ic_1,c_1^2,E_IK_S,c_1K_S,K_S^2,\chi(\O_S))}(s,y,(x y^{-\frac{1}{2}})^4)\\&\quad+(-1)^{\chi(\O_S)}(x y^{-\frac{1}{2}})^{-c_1^2-3\chi(\O_S)}
 \sfA_{(K_{S}^2-E_I^2,c_1K_{S}-c_1E_I,c_1^2,K_S^2-E_I^2,c_1K_S,K_S^2,\chi(\O_S))}(s,y,(x y^{-\frac{1}{2}})^4)\Big]\\
 &=2\sum_{I\subset M}\sum_{k=1}^3 (i^k)^{c_1^2-\chi(\O_S)}\Bigg(\frac{\phi(i^k x,y)}{2}\Bigg)^{\chi(\O_S)}\Bigg(\frac{2\overline{\eta}(x^4)^2}{\theta_3(i^k x,y^{\frac{1}{2}})}\Bigg)^{K_S^2} (-1)^{c_1E_I} \Bigg(\frac{\theta_3(-i^k x,y^{\frac{1}{2}})}{\theta_3(i^kx,y^{\frac{1}{2}})}\Bigg)^{|I|},
 \end{align*}
 where we replaced $E_{M-I}$ by $E_I$ for all terms with Seiberg-Witten invariant $(-1)^{\chi(\O_S)}$. After interchanging the sums we get the formula of the proposition.
 \end{proof}

\subsection{Arbitrary surfaces with holomorphic 2-form}

We present the following conjecture about virtual $\chi_y$-genera of moduli spaces of rank 2 sheaves on \emph{arbitrary} smooth projective surfaces $S$ with $b_1(S)=0$ and $p_g(S)>0$. Although this conjecture is strictly stronger than Conjecture \ref{conj}, the latter is a little easier to state and was therefore the focus of the introduction. We provide some evidence for this conjecture in Sections \ref{K3sec}, \ref{ellipticevidence}, and \ref{blowupevidence}. 

\begin{conjecture}\label{generalsurfconj}
Let $S$ be a smooth projective surface with $b_1(S) = 0$ and $p_g(S)>0$. Let $H,c_1,c_2$ be chosen such that there are no rank 2 strictly Gieseker $H$-semistable sheaves with Chern classes $c_1,c_2$ and let $M:=M_S^H(2,c_1,c_2)$. Then $\overline{\chi}_{-y}^{\vir}(M)$ equals the coefficient of $x^{\vd(M)}$ of 
\begin{align*}
\psi_{S,c_1}(x,y) &:= 4 \Bigg(\frac{\phi(x,y)}{2}\Bigg)^{\chi(\O_S)}\Bigg(\frac{2 \overline{\eta}(x^4)^2}{\theta_3(x,y^{\frac{1}{2}})}\Bigg)^{K_{S}^2}  \sum_{a \in H^2(S,\Z)} \SW(a)(-1)^{c_1 a} \Bigg(\frac{\theta_3(x,y^{\frac{1}{2}})}{\theta_3(-x,y^{\frac{1}{2}})}\Bigg)^{a K_S},
\end{align*}
where
$$
\phi(x,y) :=\prod_{n=1}^{\infty} \frac{1}{(1-x^{2n})^{10}(1-x^{2n} y)(1-x^{2n} y^{-1})}.
$$
\end{conjecture}
Assuming this conjecture and when there are no strictly Gieseker $H$-semistable sheaves with first Chern class $c_1$, the same calculation as in Section \ref{fixedc1sec} gives
 \begin{equation} \label{Zpsi}
 \sfZ_{S,c_1}(x,y) = \frac{1}{2} \psi_{S,c_1}(x,y) + \frac{1}{2} i^{c_1^2 - \chi(\O_S)} \psi_{S,c_1}(ix,y),
\end{equation}
where $i = \sqrt{-1}$. Specializing to $y=1$ gives (part of) a formula from the physics literature due to Dijkgraaf-Park-Schroers, namely terms two and three of \cite[Eqn.~(6.1)]{DPS}.\footnote{Up to an overall factor $x^{-\chi(\O_S) + K_S^2/3}$ coming from our choice of normalization.} 
This involves a bit of rewriting using 
\begin{equation} \label{rewritetheta}
\theta_3(i^k x)=\theta_3(x^4)+i^k\theta_2(x^4),
\end{equation} 
$\SW(a) = (-1)^{\chi(\O_S)} \SW(K_S-a)$ \cite[Cor.~6.8.4]{Mor}, and $a^2=aK_S$ for Seiberg-Witten basic classes \cite[Prop.~6.3.1]{Moc}.

\begin{remark}
A straight-forward calculation shows that this conjecture implies both Proposition \ref{fixedc1prop} (without assuming Conjecture \ref{conj}) and Proposition \ref{nonmingt} (without assuming Conjecture \ref{numconj} and without assuming $K_S^2 \geq \chi(\O_S)-3$). In fact, this conjecture implies Conjecture \ref{conj}.
\end{remark}

The first application of Conjecture \ref{generalsurfconj} is the following blow-up formula.
\begin{proposition} \label{blowupcor}
Assume Conjecture \ref{generalsurfconj} holds. Let $\pi : \widetilde{S} \rightarrow S$ be the blow-up in a point of a smooth projective surface $S$ with $b_1(S) = 0$, $p_g(S)>0$. Suppose $H, c_1$ are chosen such that there are no rank 2 strictly Gieseker $H$-semistable sheaves with first Chern class $c_1$. Let $\widetilde{c}_1 = \pi^* c_1 - \epsilon E$ with $\epsilon=0,1$ and suppose $\widetilde{H}$ is a polarization on $\widetilde{S}$ such that there are no rank 2 strictly Gieseker $\widetilde{H}$-semistable sheaves on $\widetilde{S}$ with first Chern class $\widetilde{c}_1$. Then
\begin{align*}
\sfZ_{\widetilde{S}, \widetilde{c}_1}(x,y) = \begin{cases} \frac{\theta_3(x^4,y)}{\overline{\eta}(x^4)^2} \, \sfZ_{S, c_1}(x,y), & \epsilon=0  \\ \frac{\theta_2(x^4,y)}{\overline{\eta}(x^4)^2} \, \sfZ_{S, c_1}(x,y), & \epsilon=1. \end{cases}
\end{align*}
\end{proposition}
\begin{proof}
The Seiberg-Witten basic classes of $\widetilde{S}$ are $\pi^*a$, $\pi^*a+E$, where $a$ is a Seiberg-Witten basic class of $S$ and the corresponding Seiberg-Witten invariants are \cite[Thm.~7.4.6]{Mor}
\begin{align*}
\SW(\pi^*a) = \SW(\pi^*a+E) = \SW(a). 
\end{align*}
Using $\chi(\O_{\widetilde{S}}) = \chi(\O_S)$, $K_{\widetilde{S}} = \pi^* K_{S}+E$, $K_{\widetilde{S}}^2 = K_S^2-1$, Conjecture \ref{generalsurfconj} implies
$$
\psi_{\widetilde{S},\widetilde{c}_1}(x,y) = \frac{1}{2} \Bigg( \frac{\theta_3(x,y^{\frac{1}{2}})}{\overline{\eta}(x^4)^2} + (-1)^{\epsilon} \frac{\theta_3(-x,y^{\frac{1}{2}})}{\overline{\eta}(x^4)^2} \Bigg) \psi_{S,c_1}(x,y).
$$
Specializing to $\epsilon=0,1$ and using \eqref{rewritetheta} the result follows.
\end{proof}

For $y=1$ this blow-up formula appears in a physics context in \cite[Sect.~4.3]{VW}. 

\begin{remark}
Let $\pi : \widetilde{S} \rightarrow S$ be the blow-up in a point of a simply connected smooth projective surface $S$. Let $H, c_1$ be chosen on $S$ such that $c_1 H$ is odd. Let $\widetilde{H} = r \pi^* H - E$ for $r \gg 0$ and $\widetilde{c}_1 = \pi^* c_1 - \epsilon E$ for $\epsilon=0,1$ such that $\widetilde{c}_1 \widetilde{H}$ is odd as well. One can show that the moduli spaces $M_{\widetilde{S}}^{\widetilde{H}}(2,\widetilde{c}_1,c_2)$ do not depend on the choice of such $\widetilde{H}$ \cite{LQ1}. In this setting Li-Qin \cite{LQ1, LQ2} derived a blow-up formula for the virtual Hodge polynomials of these moduli spaces\footnote{Here the adjective ``virtual'' does not refer to virtual cycles. The definition of virtual Hodge polynomials involves Deligne's weight filtration. They coincide with ordinary Hodge polynomials for smooth projective varieties.}  
$$
h(M_S^H(2,c_1,c_2), x_1, x_2), \, h(M_{\widetilde{S}}^{\widetilde{H}}(2,\widetilde{c}_1,c_2), x_1, x_2).
$$
Normalize the virtual Hodge polynomials as follows
\begin{align*}
\overline{h}(M_S^H(2,c_1,c_2), x_1, x_2) &= (x_1 x_2)^{ - \vd(M_S^H(2,c_1,c_2))/2} h(M_S^H(2,c_1,c_2), x_1, x_2), \\
\overline{h}(M_{\widetilde{S}}^{\widetilde{H}}(2,\widetilde{c}_1,c_2), x_1, x_2) &= (x_1 x_2)^{ - \vd(M_{\widetilde{S}}^{\widetilde{H}}(2,\widetilde{c}_1,c_2)/2}  h(M_{\widetilde{S}}^{\widetilde{H}}(2,\widetilde{c}_1,c_2), x_1, x_2).
\end{align*}
Then Li-Qin's formula reads (see also \cite[Rem.~3.2]{Got2})
$$
\sum_{c_2} \overline{h}(M_{\widetilde{S}}^{\widetilde{H}}(2,\widetilde{c}_1,c_2), x_1, x_2) x^{4c_2 - \widetilde{c}_{1}^2} = \begin{cases} \frac{\theta_3(x^4,x_1x_2)}{\overline{\eta}(x^4)^2} \sum_{c_2} \overline{h}(M_S^H(2,c_1,c_2), x_1, x_2) x^{4c_2 - c_{1}^2}, & \epsilon=0  \\ \frac{\theta_2(x^4,x_1x_2)}{\overline{\eta}(x^4)^2} \sum_{c_2} \overline{h}(M_S^H(2,c_1,c_2), x_1, x_2) x^{4c_2 - c_{1}^2}, & \epsilon=1. \end{cases}
$$
When specializing to $x_1 = y$ and $x_2=1$, this gives the ratios of Proposition \ref{blowupcor}. Hence the blow-up formula for virtual $\chi_y$-genera (virtual in the sense of virtual classes) and the blow-up formula for $\chi_y$-genera (defined via virtual Hodge polynomials) coincide. In particular, the blow-up formula for virtual Euler characteristics and classical Euler characteristics involve the same ratio as well.
\end{remark}

The second application of Conjecture \ref{generalsurfconj} is to surfaces with canonical divisor with irreducible reduced connected components.
\begin{proposition} \label{disconn}
Assume Conjecture \ref{generalsurfconj} holds. Let $S$ be a smooth projective surface with $b_1(S) = 0$, $p_g(S)>0$, and suppose $C_1+ \cdots +C_m \in |K_S|$, where $C_1, \ldots, C_m$ are mutually disjoint irreducible reduced curves. Suppose $H,c_1$ are chosen such that there are no rank 2 strictly Gieseker $H$-semistable sheaves with first Chern class $c_1$. Then
\begin{align*} 
&\sfZ_{S,c_1}(x,y) = 2 \Bigg(\frac{\phi(x,y)}{2}\Bigg)^{\chi(\O_S)} \prod_{j=1}^{m} \Bigg\{ \Bigg( \frac{2\overline{\eta}(x^4)^2}{\theta_3(x,y^{\frac{1}{2}})} \Bigg)^{C_j^2} + (-1)^{c_1 C_j+h^0(N_{C_j/S})} \Bigg( \frac{2\overline{\eta}(x^4)^2}{\theta_3(-x,y^{\frac{1}{2}})} \Bigg)^{C_j^2} \Bigg\} \\
&+ 2(-i)^{c_1^2 - \chi(\O_S)} \Bigg(\frac{\phi(-ix,y)}{2}\Bigg)^{\chi(\O_S)}  \prod_{j=1}^{m} \Bigg\{ \Bigg( \frac{2\overline{\eta}(x^4)^2}{\theta_3(-i x,y^{\frac{1}{2}})} \Bigg)^{C_j^2} + (-1)^{c_1 C_j+h^0(N_{C_j/S)}} \Bigg( \frac{2\overline{\eta}(x^4)^2}{\theta_3(i x,y^{\frac{1}{2}})} \Bigg)^{C_j^2} \Bigg\}
\end{align*}
where $i = \sqrt{-1}$ and $N_{C_j/S}$ denotes the normal bundle of $C_j \subset S$.
\end{proposition}

Specializing to $y=1$ and using $\theta_3(i^k x)=\theta_3(x^4)+i^k\theta_2(x^4)$ gives a more explicit version of \cite[Eqn.~(5.45)]{VW}.\footnote{Up to an overall factor $x^{-\chi(\O_S) + K_S^2/3}$ coming from our choice of normalization.}
Before we prove this proposition, we need three lemmas about disconnected curves and their Seiberg-Witten invariants.
\begin{lemma} \label{lemm1}
Let $C, D$ be irreducible reduced mutually disjoint curves on a smooth projective surface $S$ with $b_1(S)=0$. Then precisely one of the following is true:
\begin{itemize}
\item $C$ or $D$ is rigid, i.e.~$|C|$ or $|D|$ is 0-dimensional.
\item $|C| = |D| \cong \PP^1$ is a pencil.
\end{itemize}
\end{lemma}
\begin{proof}
Suppose neither of $C,D$ is rigid. Then their linear systems sweep out $S$. Therefore $|C|$ contains an element $F$ which intersects $D$. Note that $F$ is connected because $C$ is irreducible reduced. The intersection cannot be only in dimension 0, because $CD=0$. Therefore $F = D + \sum_{i \in I} n_i F_i$, where $I$ is a finite index set, $n_i>0$, and $D, \{F_i\}_{i \in I}$ are all mutually distinct prime divisors. Suppose $|I|>0$. Then
$$
0 = CD = FD = D^2 + \sum_{i \in I} n_i F_i D > D^2.
$$
Hence $D^2<0$, so $H^0(N_{D/S}) = 0$ contradicting the assumption that $D$ is not rigid. Therefore $I = \varnothing$ and $D \in |C|$. Furthermore, $|C|$ is base-point free and $C^{\prime} C^{\prime \prime} = C D = 0$ for all $C^{\prime}, C^{\prime \prime} \in |L|$ so $|L| \cong \PP^1$. 
\end{proof}
  
Suppose $C_1, \ldots, C_m$ are irreducible reduced mutually disconnected curves on a smooth projective surface $S$ with $b_1(S) = 0$ and let $M:=\{1, \ldots, m\}$. Then for any $I = \{i_1, \ldots, i_k\} \subset M$, we define
$$
C_I := \sum_{i \in I} C_i.
$$
For $I, J \subset M$ we write $I \sim J$ whenever $C_I \sim_{\mathrm{lin}} C_J$. This defines an equivalence relation. We denote the equivalence class corresponding to $I$ by $[I]$ and denote its number of elements by $|[I]|$.
 \begin{lemma} \label{lemm2}
For any $I \subset M$, we have $|[I]| = \binom{\dim |C_M|}{ \dim |C_I|}$.
 \end{lemma}
 \begin{proof}
 Since $b_1(S) = 0$, any effective divisor  $D \subset S$ satisfies $\dim |D| = h^0(N_{D/S})$. In particular we have
 \begin{equation} \label{dimform}
 \dim|C_I| = h^0(N_{C_I/S}) = \sum_{i \in I} h^0(N_{C_i/S}) = \sum_{i \in I} \dim |C_i|.
 \end{equation}
 Suppose, possibly after relabeling, that $C_1, \ldots, C_a$ are the rigid curves (i.e.~their linear systems are 0-dimensional). Then $C_{a+1}, \ldots, C_{m}$ are all linearly equivalent (Lemma \ref{lemm1}). Moreover, if $m>a+1$, then their linear systems are pencils (Lemma \ref{lemm1}). There are three cases: \\
 
 \noindent \textbf{Case 1:} $m=a$. Then all curves $C_i$ are rigid, so $I \sim J$ if and only if $I=J$ and the statement follows from \eqref{dimform}. \\
 
 \noindent \textbf{Case 2:} $m=a+1$. Then only $C_m$ is not rigid and it is again easy to see that $I \sim J$ if and only if $I=J$. By \eqref{dimform}, we have that $\dim|C_I| = \dim |C_m|$ if $m \in I$ and zero otherwise. The statement follows. \\
 
 \noindent \textbf{Case 3:} $m>a+1$. Then $C_1, \ldots, C_a$ are rigid and $|C_{a+1}| = \cdots = |C_m| \cong \PP^1$. Let $A:=\{1, \ldots, a\}$ and $B:=\{a+1, \ldots, m\}$. Then $I \sim J$ if and only if $A \cap I = A \cap J$ and $|B \cap I| = |B \cap J|$. Therefore $|[I]| = \binom{|B|}{|B \cap I|}$. The result follows from equation \eqref{dimform} as follows
 \begin{equation*}
 \dim|C_I| = \sum_{i \in A \cap I} \dim|C_i| + \sum_{i \in B \cap I} \dim|C_i| = |B \cap I|. \qedhere
 \end{equation*}
 \end{proof}

\begin{lemma} \label{lemm3}
Let $S$ be a smooth projective surface with $b_1(S) = 0$, $p_g(S)>0$, and suppose $C_1+ \cdots +C_m \in |K_S|$, where $C_1, \ldots, C_m$ are mutually disjoint irreducible reduced curves. Then the Seiberg-Witten basic classes of $S$ are $\{C_I\}_{I \subset M}$ and
$$
\SW(C_I) = |[I]| \prod_{i \in I} (-1)^{h^0(N_{C_i/S})}.
$$
\end{lemma}
\begin{proof}
The proof combines Lemma \ref{lemm2} and the proof of \cite[Prop.~6.3.1]{Moc}. We first note that all Seiberg-Witten basic classes must be of the form $\{C_I\}_{I \subset M}$ (this can be seen most easily from the cosection localization of Chang-Kiem \cite[Lem.~3.2]{CK}). Let $\varnothing \neq I \subset M$. Then Mochizuki shows that 
$$
\SW(C_I) = c_{\mathrm{top}}(\Ob) \cap |C_I|,
$$
where $\Ob$ is a rank $h^1(N_{C_I/S})$ vector bundle with total Chern class
$$
(1+h)^{h^1(N_{C_I/S}) - p_g(S)}
$$
where $h$ denotes the hyperplane class on $|C_I|$. Hence
$$
\SW(C_I) = \binom{h^1(N_{C_I/S})-p_g(S)}{h^1(N_{C_I/S})} = (-1)^{h^1(N_{C_I/S})} \binom{p_g(S)-1}{h^1(N_{C_I/S})}.
$$
By Serre duality and adjunction $K_{C_i} = (C_M+C_i)|_{C_i} = 2C_i|_{C_i}$
$$
h^1(N_{C_I/S}) = \sum_{i \in I} h^1(\O_{C_i}(C_i)) = \sum_{i \in I} h^0(\O_{C_i}(C_i)) = h^0(N_{C_I/S}).  
$$
Therefore Lemma \ref{lemm2} implies
\begin{equation*}
\binom{p_g(S)-1}{h^1(N_{C_I/S})} = \binom{\dim|C_M|}{\dim|C_I|} = |[I]|. \qedhere
\end{equation*}
\end{proof}

\begin{proof}[Proof of Proposition \ref{disconn}]
Combining Conjecture \ref{generalsurfconj} and Lemma \ref{lemm3} gives 
\begin{align*}
\psi_{S,c_1}(x,y) &= 4 \Bigg(\frac{\phi(x,y)}{2}\Bigg)^{\chi(\O_S)}\Bigg(\frac{2 \overline{\eta}(x^4)^2}{\theta_3(x,y^{\frac{1}{2}})}\Bigg)^{K_{S}^2}  \sum_{[I]} |[I]| (-1)^{h^0(N_{C_I/S})+c_1C_I} \Bigg(\frac{\theta_3(x,y^{\frac{1}{2}})}{\theta_3(-x,y^{\frac{1}{2}})}\Bigg)^{C_I^2} \\
&= 4 \Bigg(\frac{\phi(x,y)}{2}\Bigg)^{\chi(\O_S)}\Bigg(\frac{2 \overline{\eta}(x^4)^2}{\theta_3(x,y^{\frac{1}{2}})}\Bigg)^{K_{S}^2}  \sum_{I} (-1)^{h^0(N_{C_I/S})+c_1C_I} \Bigg(\frac{\theta_3(x,y^{\frac{1}{2}})}{\theta_3(-x,y^{\frac{1}{2}})}\Bigg)^{C_I^2} \\
&= 4 \Bigg(\frac{\phi(x,y)}{2}\Bigg)^{\chi(\O_S)}\Bigg(\frac{2 \overline{\eta}(x^4)^2}{\theta_3(x,y^{\frac{1}{2}})}\Bigg)^{K_{S}^2}  \prod_{i=1}^{m}\Bigg(1+(-1)^{h^0(N_{C_i/S})+C_i c_1} \Bigg(\frac{\theta_3(x,y^{\frac{1}{2}})}{\theta_3(-x,y^{\frac{1}{2}})}\Bigg)^{C_i^2} \Bigg).
\end{align*}
The formula for $\sfZ_{S,c_1}(x,y)$ follows from \eqref{Zpsi} after re-organizing the terms.
\end{proof}

\section{Verification of the conjectures in examples} \label{evidencesec}

In this section we check Conjectures \ref{conj}, \ref{numconj}, and \ref{generalsurfconj} in many cases. We recall that we calculated  $A_1(s,y,q), \ldots, A_7(s,y,q)$ up to order $q^7$, where we calculated the coefficient of $q^i$ up to order $s^{29-3i}$.  We computed $A_1(s,1,q), \ldots, A_7(s,1,q)$ up to order $q^{30}$ and any order in $s$ (see Section \ref{toricsec}). The latter are listed up to order $q^4$ in Appendix \ref{app1}.
%

\subsection{K3 surfaces} \label{K3sec}

Let $S$ be a $K3$ surface. The canonical class is trivial and $b_1(S) = 0$ so we are in the setting of Conjecture \ref{generalsurfconj}, which states 
\begin{equation}\label{K3gen}
\sfZ_{S,c_1}(x,y)=\frac{1}{2}\big(\phi(x,y)^2-i^{c_1^2} \phi(ix,y)^2\big).
\end{equation}
This can be restated as saying that $\overline\chi_{-y}^\vir(M^H_S(2,c_1,c_2))$ is the coefficient of $4c_2-c_1^2-6$ of 
$$
\phi(x,y)^2=
\prod_{n=1}^{\infty} \frac{1}{(1-x^{2n})^{20}(1-x^{2n} y)^2(1-x^{2n} y^{-1})^2},
$$
when there are no rank 2 strictly $H$-semistable sheaves with Chern classes $c_1,c_2$.
This looks very similar to Conjecture \ref{conj}, which does not apply, and which would be off by a factor 2. 

In the absence of strictly Gieseker $H$-semistable sheaves, the moduli space $M_{S}^{H}(2,c_1,c_2)$ is smooth of expected dimension.
If in addition $c_1$ is primitive, big, and nef, then in \cite{GH} it is shown that $$\chi_{-y}(M^S_H(2,c_1,c_2))=\chi_{-y}(S^{[2c_2 - c_1^2/2 - 3]}).$$ 
More generally,  in \cite[Cor.~4.8]{Huy}, when $c_1$ is primitive, and in \cite[Thm.~0.1]{Yos4} in general, it is shown that in the absence of strictly Gieseker $H$-semistable sheaves $M_{S}^{H}(r,c_1,c_2)$ is deformation equivalent to the Hilbert scheme $S^{[\vd/2]}$. In particular they have the same Hodge numbers. In turn, $\chi_y$-genera of Hilbert schemes of points on $K3$ surfaces were computed by the first named author and W.~Soergel in \cite{GS}. Combining this formula with the result above implies that  \eqref{K3gen} holds for $K3$ surfaces.
 
The only Seiberg-Witten basic class of $S$ is $0$ with corresponding Seiberg-Witten invariant 1. Suppose $c_1$ is a choice of first Chern class such that $c_1^2 = c$, then $c \in 2 \Z$. 
Then we see that \eqref{K3gen} also follows from Conjecture \ref{numconj}. We put $(\gamma_1,\gamma_2)=(0,0)$ and $(\beta_1,\beta_2,\beta_3,\beta_4)=(c,0,0,2)$.
Then the two summands in $\Coeff_{s^0x^n}[\cdot]$ in Conjecture \ref{numconj} are equal, and each equals the contribution of the Seiberg-Witten basic class $0=K_S$.
Therefore, if $c_2$ satisfies the assumptions of Conjecture \ref{numconj}, it gives that 
$\overline\chi_{-y}^\vir(M^H_S(2,c_1,c_2))$ is the coefficient of $x^{4c_2-c_1^2-6}$ of 
$$
\frac{1}{2}\psi_{0,0,0,2}(x,y)=\phi(x,y)^2.
$$
The same argument as in the proof of Proposition \ref{numconjimpliesconj}  shows  that \eqref{K3gen} follows from Conjecture \ref{numconj} and the strong form of Mochizuki's formula.

All the same, we want to calculate terms of \eqref{K3gen} directly by applying Corollary \ref{univcor} and our explicit knowledge of the universal functions $A_i$ as described in Section \ref{toricsec}. We use the conjectured strong form of Mochizuki's formula (Remark \ref{strongform}), so our calculations can also be viewed as an additional test of Remark \ref{strongform}.
The easiest way to satisfy all assumptions of Corollary \ref{univcor} is by choosing $c_1$ and $H$ such that $c_1 H > 0$ is odd (though this is unnecessarily strong). According to Theorem \ref{mocthm} (i), we can calculate modulo
$$
\mod x^{c + 10}.
$$
In addition we have a bound given by the accuracy to which we calculated the universal functions $A_i(s,y,q)$ and $A_i(s,1,q)$. For $A_i(s,1,q)$ this bound is $q^{30}$, which for $\sfZ_{S,c_1}(x,1)$ means we can calculate modulo
$$
\mod x^{-c+118}.
$$
For $y=1$, i.e.~the case of virtual Euler characteristics, we tested Conjecture \ref{generalsurfconj} for all even $c \in \{-6, \ldots, 116\}$ up to the above accuracies. In conclusion, for $K3$ surfaces and $y=1$, we verified Conjecture \ref{generalsurfconj} by direct calculation for:
\begin{itemize}
\item $c_1$ such that $c_1^2 = c \in \{-6, \ldots, 116\}$ is even, modulo $x^{c+10}$ (bound from Corollary \ref{univcor} (i)) and modulo $x^{-c+118}$ (bound from $A_i$). Here ``and'' means that both bounds apply, in other words ``modulo $x^{\min(c+10,-c+118)}$.
\end{itemize}

\noindent For virtual $\chi_y$-genus, we tested Conjecture \ref{generalsurfconj} for: 
\begin{itemize}
\item $c_1$ such that $c_1^2 = c \in \{ -6, \ldots, 32\}$ is even, modulo $x^{c+10}$ (bound from Corollary \ref{univcor} (i)) and modulo $x^{-c+26}$ (bound from $A_i$).
\end{itemize}

\subsection{Elliptic surfaces} \label{ellipticevidence}

Let $S \rightarrow \PP^1$ be a non-trivial elliptic surface with section, $12n>0$ rational nodal fibres, and no other singular fibres, i.e.~an elliptic surface of type $E(n)$. The canonical class is given by $K_S = (n-2) F$, where $F$ denotes the class of the fibre. Note that $\chi(\O_S) = n$. Moreover, choose a section $B \subset S$, then its class satisfies $B^2 = -n$. 

We assume $n \geq 2$, then $E(n)$ has a smooth  canonical divisor which has $m=n-2$ connected components $F_j$; each a smooth elliptic fibre of $S$. The surface $E(3)$ can be used to test Conjecture \ref{conj} and the surfaces $E(n)$ with $n \geq 2$ can be used to test Conjecture \ref{generalsurfconj}. Note that $E(2)$ is $K3$, which was discussed in Section \ref{K3sec}, so we take $n \geq 3$.

Let $c_1 = \epsilon B + d F$, for $\epsilon, d \in \Z$ and $B$ the class of the section. 
Since $F_j^2=0$ and  $c_1 F_j=\epsilon$, Conjecture \ref{generalsurfconj} gives 
\begin{equation} 
\label{ellgen}
\sfZ_{S,c_1}(x,y)=\begin{cases}\frac{1}{2}\big(\phi(x,y)^n+(-i)^{c_1^2 - n}\phi(-ix,y)^n\big), & \epsilon \hbox{ odd} \\0, & \epsilon \hbox{ even.}\end{cases}
\end{equation}
In the case $n=3$, this can be restated as saying that 
$\overline\chi_{-y}^\vir(M^H_S(2,c_1,c_2))$ is the coefficient of $4c_2-c_1^2-9$ of $
\phi(x,y)^3$, which is the statement of 
Conjecture \ref{conj}.

In \cite{Yos3}, Yoshioka fixes  $c_1, c_2$ such that $c_1 F$ is odd and an ample divisor $H = B + \beta F$ with $\beta \gg c_2$. By choosing $\beta$ of the right parity $c_1H$ is odd, so there are no rank 2 strictly $\mu$-semistable sheaves with Chern classes $c_1,c_2$. Then $M:=M_{S}^{H}(2,c_1,c_2)$ is smooth irreducible of expected dimension and independent of $H$.  
In this setting, the Betti numbers and Euler characteristics of $M:=M_{S}^{H}(2,c_1,c_2)$ were calculated in \cite{Yos3} and are indeed given by the specialization $y=1$ of (\ref{ellgen}). 

Next we want to calculate terms of \eqref{ellgen} directly by applying Corollary \ref{univcor} and our explicit knowledge of the universal functions $A_i$ as described in Section \ref{toricsec}. There are numerous choices for the polarization $H$ for which the conditions of Corollary \ref{univcor} are satisfied. Specifically, let $H = \alpha B + \beta F$ with $\alpha>0$ and $\beta > \alpha n$ be a polarization. When 
\begin{align*}
c_1 H &= (\beta - \alpha n) \epsilon + \alpha d > 2\alpha (n-2) = 2 K_S H, \\
c_1 H &= (\beta - \alpha n) \epsilon + \alpha d  \qquad \mathrm{odd},
\end{align*}
all conditions are satisfied. For fixed $\epsilon > 0$ and $d \in \Z$ not both even, there are many solutions $\alpha>0$ and $\beta > \alpha n$ to these equations. By \cite{FM}, the Seiberg-Witten basic classes are $0, F \ldots, (n-2)F$ and
$$
\SW(pF) = (-1)^p \binom{n-2}{p}.
$$
For $y=1$, we verified Conjecture \ref{generalsurfconj} in the following cases: 
\begin{itemize}
\item $E(n)$ such that $n = 3, \ldots, 8$, $c_1=\epsilon B + d F$ such that $\epsilon = 1,2$ and $d = 0,\ldots, 8$ not both even, modulo $x^{- \epsilon^2 n+2\epsilon d - 2 \epsilon n + 4 \epsilon + 5n}$ (bound from Corollary \ref{univcor} (i)) and modulo $x^{n \epsilon^2 - 2 \epsilon d - 3n + 4\min\{0, \epsilon (n-2) \} + 124}$ (bound from $A_i$).
\end{itemize}

\noindent For virtual $\chi_y$-genus, we tested Conjecture \ref{generalsurfconj} for:
\begin{itemize}
\item $E(n)$ such that $n = 3,\ldots,6$, $c_1=\epsilon B + d F$ such that $\epsilon = 1,2$ and $d = 4 \ldots,8$ not both even, modulo $x^{- \epsilon^2 n + 2\epsilon d - 2 \epsilon n + 4 \epsilon + 5n}$ (bound from Corollary \ref{univcor} (i)) and modulo $x^{n \epsilon^2 - 2 \epsilon d - 3n + 4\min\{0, \epsilon (n-2) \} + 32}$ (bound from $A_i$).
\end{itemize}

\subsection{Blow-ups} \label{blowupevidence}

Let $\pi : \widetilde{S} \rightarrow S$ be the blow-up of a $K3$ surface $S$ in a point and denote the exceptional divisor by $E$. Then $K_{\widetilde{S}} = E$ is a smooth canonical divisor. We want to gather evidence for Conjecture \ref{conj} by applying Corollary \ref{univcor} and our explicit knowledge of the universal functions $A_i$ as described in Section \ref{toricsec}. Let $\widetilde{c}_1 = \pi^* c_1 + \epsilon E$ and $c_1^2=c \in 2 \Z$. The easiest way to satisfy the conditions of Corollary \ref{univcor} is to take a polarization $H$ on $S$ such that $c_1 H > 0$ is odd and taking $\widetilde{H} = r H - E$ with $r \gg 0$ and $r + \epsilon$ odd, but more general choices are possible. The blow-up formula for Seiberg-Witten invariants implies that $\widetilde{S}$ has Seiberg-Witten basic classes $0,E$ and $\SW(0)=\SW(E) = 1$ \cite[Thm.~7.4.6]{Mor}. For $y=1$, we verified Conjecture \ref{conj} in the following cases: 
\begin{itemize}
\item  $\widetilde{c}_1 = \pi^* c_1 + \epsilon E$ with $c_1^2=c$, $\epsilon = -3,\ldots,3$, $c = -10,\ldots,80$ even, modulo $x^{c - \epsilon^2 + 3 \epsilon + 10}$ (bound from Corollary \ref{univcor} (i)) and modulo $x^{-c + \epsilon^2 + 4\min\{0,-(\epsilon-1)\}+118}$ (bound from $A_i$).
\end{itemize}

\noindent For virtual $\chi_y$-genus, we tested Conjecture \ref{conj} for:
\begin{itemize}
\item  $\widetilde{c}_1 = \pi^* c_1 + \epsilon E$ with $c_1^2=c$, $\epsilon = 0,\ldots,3$, $c =  -10,\ldots,26$ even, modulo $x^{c - \epsilon^2 + 3 \epsilon+ 10}$ (bound from Corollary \ref{univcor} (i)) and modulo $x^{-c + \epsilon^2 + 4\min\{0,-(\epsilon-1)\}+26}$ (bound from $A_i$).
\end{itemize}

Next we consider the blow-up of a $K3$ surface $S$ in two distinct points
$$
\pi : \widetilde{S} \rightarrow S
$$
and we denote the exceptional divisors by $E_1, E_2$. We gather evidence for Conjecture \ref{generalsurfconj} by applying Corollary \ref{univcor} and our explicit knowledge of the universal functions $A_i$ as described in Section \ref{toricsec}. By the blow-up formula for Seiberg-Witten invariants, the Seiberg-Witten basic classes are
$$
0, \, E_1, \, E_2, \, E_1+E_2
$$
and the invariant is 1 in each case. We consider classes
$$
\widetilde{c}_1 = \pi^* c_1 + \epsilon_1 E_1 + \epsilon_2 E_2.
$$
As before the easiest way to satisfy the conditions of Corollary \ref{univcor} is to take a polarization $H$ on $S$ such that $c_1 H > 0$ is odd and taking $\widetilde{H} = r H - E_1 - E_2$ with $r \gg 0$ and parity such that $r + \epsilon_1 + \epsilon_2$ is odd, though other choices are possible. For $y=1$, we verified Conjecture \ref{generalsurfconj} in the following cases: 
\begin{itemize}
\item $\widetilde{c}_1 = \pi^* c_1 + \epsilon_1 E_1 + \epsilon_2 E_2$ with $c_1^2=c$, $\epsilon_1, \epsilon_2 = -2,\ldots,2$, $c =-10,\ldots ,126$ even, modulo $x^{c - \epsilon_1^2 - \epsilon_2^2 + 2\epsilon_1 + 2\epsilon_2 + 10}$ (bound coming from Corollary \ref{univcor} (i)) and modulo $x^{-c+ \epsilon_1^2+\epsilon_2^2 +4\min\{0,-(\epsilon_1-1),-(\epsilon_2-1),-(\epsilon_1+\epsilon_2-2)\} + 118}$ (bound from $A_i$).
\end{itemize}

\noindent For virtual $\chi_y$-genus, we tested Conjecture \ref{generalsurfconj} for:
\begin{itemize}
\item $\widetilde{c}_1 = \pi^* c_1 + \epsilon_1 E_1 + \epsilon_2 E_2$ with $c_1^2=c$, $\epsilon_1, \epsilon_2 = 0,1,2$, $c =-10,\ldots,30$ even, modulo $x^{c - \epsilon_1^2 - \epsilon_2^2 + 2\epsilon_1 + 2\epsilon_2 + 10}$ (bound from Corollary \ref{univcor} (i)) and in addition modulo $x^{-c+ \epsilon_1^2+\epsilon_2^2 +4 \min\{0,-(\epsilon_1-1),-(\epsilon_2-1),-(\epsilon_1+\epsilon_2-2)\} + 26}$ (bound from $A_i$).
\end{itemize}

\subsection{Double covers}

Next, we consider double covers
$$
\pi : S_d \rightarrow \PP^2
$$
branched over a smooth curve $C$ of degree $2d$. Then $K_{S_d} = (d-3) L$, where $L$ is the pull-back of the class of a line on $\PP^2$. These surfaces satisfy $b_1(S_d) = 0$. It is easy to calculate
\begin{align*}
K_{S_d}^2 &= 2(d-3)^2, \\ 
\chi(\O_{S_d}) &=\frac{1}{2}d(d-3)+2.
\end{align*}
Since $\alpha L$ is base-point free for all $\alpha > 0$, the canonical linear system $|K_{S_d}|$ contains smooth connected canonical divisors when $d \geq 4$, which we assume from now on. Let $c_1 = \epsilon L$. As before, we want to gather evidence for Conjecture \ref{conj} by applying Corollary \ref{univcor} and our explicit knowledge of the universal functions $A_i$ as described in Section \ref{toricsec}. The Seiberg-Witten basic classes are $0, K_S \neq 0$ with Seiberg-Witten invariants 
$$
1, \, (-1)^{\chi(\O_{S_d})}.
$$ 
We first take $H=L$ as the polarization on $S$. 
Then conditions (ii), (iii) of Corollary \ref{univcor} require
\begin{align*}
c_1 H = 2  \epsilon > 4(d-3)  = 2K_S H,
\end{align*}
i.e.~$\epsilon > 2(d-3)$. 
In the case $\epsilon=2k$ is even, we can ensure the absence of rank 2 strictly Gieseker $H$-semistable sheaves with Chern classes $c_1, c_2$ by choosing $c_2$ such that
$$
\frac{1}{2} c_1(c_1 - K_S) - c_2 = \epsilon(\epsilon - (d-3)) - c_2
$$
is odd \cite[Rem.~4.6.8]{HL}. 

Now assume that $\epsilon=2k+1$ is odd.
If $L$ generates the Picard group of $S_d$, then there are no rank 2 strictly $\mu$-semistable sheaves with Chern classes $\epsilon L$ and $c_2$. 
 In general the Picard group of  $S_d$ can have more generators, but $L$ is still ample and primitive. In this case we take the polarization $H$ general and sufficiently close to $L$
 (i.e. of the form $nL+B$ for $n$ sufficiently large), so that conditions (ii) and (iii) of Corollary \ref{univcor}  still hold when 
 $\epsilon>2(d-3)$, and so that there are no rank 2 strictly $\mu$-semistable sheaves with Chern classes $\epsilon L$ and $c_2$.
 
 Recall that conjecturally conditions (ii), (iii) of Corollary \ref{univcor} are not necessary (see Remark \ref{strongform}). Therefore we will also test Conjecture \ref{conj} for $\epsilon, d$ which do not satisfy $\epsilon > 2(d-3)$. 

Note that $\mu$-stability is invariant under tensorizing by a line bundle and that Gieseker stability  with respect to $H$ is invariant under tensorizing by a power of $H$.
Therefore we know that $M^H_{S_d}(2,L,c_2)$ together with its virtual structure is isomorphic to 
$M^H_{S_d}(2,(2k+1)L,c_2+2(k^2+k))$, and when $c_2$ is odd, the same holds for $M^L_{S_d}(2,0,c_2)$ and $M^L_{S_d}(2,2kL,c_2+2k^2)$.

For $y=1$, we verified Conjecture \ref{conj} in the following cases: 
\begin{itemize}
\item  $S_d$ such that $d=4,\ldots,7$, $c_1 = \epsilon L$ with $\epsilon = -3, \ldots, 6$, modulo $x^{\frac{5}{2}d(d-3) - 4\epsilon(d-3) + 2\epsilon^2 +10}$ (bound from Corollary \ref{univcor} (i)) and modulo $x^{-\frac{3}{2}d(d-3)-2 \epsilon^2 + 8\min\{0,(d-3)(\epsilon-(d-3))\} +118}$ (bound from $A_i$).
\end{itemize}

\noindent For virtual $\chi_y$-genus, we tested Conjecture \ref{conj} for:
\begin{itemize}
\item  $S_d$ such that $d=4,\ldots,7$, $c_1 = \epsilon L$ with $\epsilon =  -3, \ldots, 6$, modulo $x^{\frac{5}{2}d(d-3) - 4\epsilon(d-3) + 2\epsilon^2 +10}$ (bound from Corollary \ref{univcor} (i)) and modulo $x^{-\frac{3}{2}d(d-3)-2 \epsilon^2 + 8\min\{0,(d-3)(\epsilon-(d-3))\} +26}$ (bound from $A_i$).
\end{itemize}

\noindent In order to give an idea of the complexity of the numbers involved, we give some examples. First we compute the virtual Euler numbers.
We only show cases where $\epsilon > 2(d-3)$, so that we do not need to use the strong form of Mochizuki's formula, and thus we have unconditionally proved these formulas.

The formula of Corollary \ref{cor} is proved for $(d=4,\epsilon=5)$, modulo $x^{50}$, for $(d=4,\epsilon =6)$ modulo $x^{40}$,
for$(d=5,\epsilon=5)$ modulo $x^{45}$ for $(d=5,\epsilon=6)$ modulo $x^{31}$. When $\epsilon$ is even, let $\sfZ^\odd_{S_d,H,0}(x,1)$ 
be the part of the generating function corresponding 
to $c_2$ odd, so the moduli spaces only consist of Gieseker stable sheaves. Using the invariance under tensoring by $H$ or $L$ respectively we have 
\begin{align*}\sfZ_{S_4,H,L}(x,1)=\,&120x^2 + 80800x^6 + 18764544x^{10} + 2352907648x^{14} + 192977504816x^{18}\\
&+ 11510039013632x^{22} + 533695300891136x^{26} +O(x^{30}),\\
\sfZ^\odd_{S_4,L,0}(x,1)=\,&2 + 1354148x^8 +  22293779698x^{16}  + 80622742217604x^{24}  \\
&+ 115687108304998636x^{32} + O(x^{40}),\\
\sfZ_{S_5,H,L}(x,1)=\,&-256x - 2622464x^5 - 4399076864x^9 - 3005594355712x^{13}\\
&- 1137273257362688x^{17} - 278765441520823296x^{21} +O(x^{25}),\\
\sfZ^\odd_{S_5,L,0}(x,1)= \,&- 123928576x^7  - 62207974965248x^{15}- 3825621677917863936x^{23} + O(x^{31}).
\end{align*}
Now we give some examples of the virtual $\chi_{y}$-genus.  In this case we need to use the strong form of Mochizuki's formula.
For $d=4$, $c_1=0,L$ we get the virtual refinements
\begin{align*}
&\sfZ_{S_4,H,L}(xy^{\frac{1}{2}},y)=(14y^2 + 92y + 14)x^2 + (154y^6 + 2540y^5 + 16398y^4 + 42616y^3 + \ldots)x^6 \\
&\quad+ (756y^{10} + 17360y^9 + 185020y^8 + 1145700y^7 + 4174352y^6 + 7718168y^5 +\ldots)x^{10} + O(x^{14}),\\
&\sfZ^\odd_{S_4,L,0}(xy^{\frac{1}{2}},y)=2 + (364y^8 + 7300y^7 + 64090y^6 + 293556y^5 + 623528y^4 + \ldots)x^8 + O(x^{16}),
\end{align*}
where the missing terms are determined by the symmetry of the virtual $\chi_y$-genus.

Next let $\FF_a = \PP(\O_{\PP^1} \oplus \O_{\PP^1}(a))$ denote the $a$th Hirzebruch surface, where $a \geq 0$. Suppose $B$ is the section corresponding to the surjection $\O_{\PP^1} \oplus \O_{\PP^1}(a) \twoheadrightarrow \O_{\PP^1}$ onto the first factor. Then $B^2 = -a$. We denote the class of the fibre of $\FF_a \rightarrow \PP^1$ by $F$. Let $d_1>0$ and $d_2>ad_1$. A rich source of examples are double covers 
$$
\pi : S_{a,d_1,d_2} \rightarrow \FF_a
$$
branched over a smooth connected curve in $|\O(2d_1 B + 2d_2 F)|$. Let $\widetilde{B} := \pi^* B$ and $\widetilde{F} = \pi^* F$. We choose $d_1 > 2$ and $d_2 > a(d_1-1)+2$, so the canonical divisor
$$
K_{S_{a,d_1,d_2}}=(d_1-2)\widetilde{B}+ (d_2-(a + 2))\widetilde{F} 
$$
is the pull-back of a very ample divisor and its linear system contains smooth connected curves. These surfaces satisfy $b_1(S_{a,d_1,d_2}) = 0$. We have 
\begin{align*}
K_{S_{a,d_1,d_2}}^2 &=4(d_1-2)\Big(d_2-\frac{d_1a}{2}-2\Big), \\
\chi(\O_{S_{a,d_1,d_2}}) &=(d_1-1)\Big(d_2-\frac{d_1a}{2}-1\Big)+1.
\end{align*}
We first note that it is enough to determine $\sfZ_{S_{a,d_1,d_2},\epsilon_1\widetilde{B}+\epsilon_2\widetilde{F}}(x,y)$ for $a=0,1$.

\begin{proposition} Assume the strong form of Mochizuki's formula holds (Remark \ref{strongform}). Let $a \geq 0$, $d_1 > 2$,  $d_2 > (a+1)(d_1-1) + 1$, and $d_2 > (a+1)d_1$. Let $H$ be a polarization on $S_{a,d_1,d_2}$ such that there exist no rank 2 strictly Gieseker $H$-semistable sheaves with Chern classes $c_1 := \epsilon_1 \widetilde{B} + \epsilon_2 \widetilde{F}$ and $c_2$. Let $H'$ be a polarization on $S_{a+2,d_1,d_1+d_2}$ such that there exist no rank 2 strictly Gieseker $H'$-semistable sheaves with Chern classes $c_1' := \epsilon_1 \widetilde{B} + (\epsilon_1+\epsilon_2) \widetilde{F}$ and $c_2$. Then
$$
\chi_{-y}^{\vir}(M_{S_{a,d_1,d_2}}^{H}(2,\epsilon_1 \widetilde{B} + \epsilon_2 \widetilde{F},c_2)) = \chi_{-y}^{\vir}(M_{S_{a+2,d_1,d_1+d_2}}^{H'}(2,\epsilon_1 \widetilde{B} + (\epsilon_1+\epsilon_2) \widetilde{F},c_2)),
$$
for all
\begin{align*}
c_2 &< 2\chi(\O_{S_{a,d_1,d_2}}) + \frac{1}{2} c_1(c_1-K_{S_{a,d_1,d_2}}) \\
&= 2\chi(\O_{S_{a+2,d_1,d_1+d_2}}) + \frac{1}{2} c_1^{\prime}(c_1^\prime -  K_{S_{a+2,d_1,d_1+d_2}}).
\end{align*}
\end{proposition}
\begin{proof}
Write $S:=S_{a,d_1,d_2}$ and $S':=S_{a+2,d_1,d_1+d_2}$. Fix $c_1:=\epsilon_1\widetilde{B}+\epsilon_2\widetilde{F}$ on $S$ and $c_1':=\epsilon_1\widetilde{B}+(\epsilon_1+\epsilon_2)\widetilde{F}$ on $S'$. The map 
$$
\phi : H^2(S,\Z)\to H^2(S',\Z), \ \widetilde{B}\mapsto \widetilde{B}+\widetilde{F}, \ \widetilde{F}\mapsto \widetilde{F}
$$
 is an isomorphism between the subgroups generated by $\widetilde{B}, \widetilde{F}$ on $S$, $S'$ and it preserves the intersection forms on these subgroups.
It sends $K_{S}$ to $K_{S'}$, and thus it sends
the  Seiberg-Witten basic classes $0$, $K_ {S}$ of $S$ to the corresponding Seiberg-Witten basic classes  of $S'$. We also see that $c_1$ is send to $c_1'$. Moreover $\chi(\O_{S})=\chi(\O_{S'})$. 
We apply the strong form of Mochizuki's formula to $S$ with $H,c_1,c_2$ and $S'$ with $H',c_1',c_2$. Since
$$
(a_1^2, a_1 c_1, c_1^2, a_1 K_S, c_1 K_S, K_S^2,\chi(\O_S)) = (a_1'^2, a_1' c_1', c_1^{ \prime 2}, a_1' K_{S'}, c_1' K_{S'}, K_{S'}^2,\chi(\O_{S'}))
$$
for $(a_1,a_1') = (0,0)$ and $(a_1,a_1') = (K_S,K_{S'})$, the assertion follows from Corollary \ref{univcor} and Remark \ref{strongform}.
\end{proof}

For $y=1$, we checked Conjecture \ref{conj} in the following cases:
\begin{itemize}
\item $S_{a,d_1,d_2}$ such that  $a=0,\ldots,3$, $d_1=3,\ldots,6$, $d_2 = a(d_1-1)+3,\ldots,a(d_1-1)+10$, $c_1 = \epsilon_1 \widetilde{B} + \epsilon_2 \widetilde{F}$ with $\epsilon_1=-2,\ldots,2$, $\epsilon_2 = -2,\ldots,10$, modulo $x^{M}$ (bound from Corollary \ref{univcor} (i)) and modulo $x^{N+124}$ (bound from $A_i$),
\end{itemize}

\noindent where
\begin{align*}
&M := -\frac{5}{2} a d_1^2+\frac{5}{2} a d_1+ 5 d_1 d_2-5d_1-5d_2 +10+4 a \epsilon_1 d_1 - 4 a \epsilon_1-4 \epsilon_1 d_2+8 \epsilon_1-4 \epsilon_2 d_1\\
&\qquad\ \, +8 \epsilon_2-2 a \epsilon_1^2+4 \epsilon_1 \epsilon_2, \\
&N := 2a \epsilon_1^2 - 4 \epsilon_1 \epsilon_2 -3 (d_1-1)(d_2-\frac{d_1 a}{2} - 1) - 3 \\
&\qquad \ \, + 8\min\Big\{0,a(d_1-2)(\epsilon_1 - (d_1-2)) + (d_1-2)(\epsilon_2-(d_2-(a+2))) \\
&\qquad\qquad\qquad\qquad\qquad\qquad\qquad\qquad\quad\ \ \ \, + (d_2 - (a+2))(\epsilon_1 - (d_1-2)) \Big\}.
\end{align*}

\noindent For virtual $\chi_{y}$-genus, we checked Conjecture \ref{conj} in the following cases:
\begin{itemize}
\item $S_{a,d_1,d_2}$ such that  $a=0,\ldots,3$, $d_1=3,4,5$, $d_2 = a(d_1-1)+3,\ldots, a(d_1-1)+5$, $c_1 = \epsilon_1 \widetilde{B} + \epsilon_2 \widetilde{F}$ with $\epsilon_1=-2,\ldots,2$, $\epsilon_2 = -2,\ldots, 6$, modulo $x^{M}$ (bound from Corollary \ref{univcor} (i)) and modulo $x^{N+32}$ (bound from $A_i$).
\end{itemize}

\noindent Some examples of these tables are the following: 
\begin{align*} \sfZ^\odd_{S_{0,3,3},0}(x,1)&=- 105472x^5  - 17699188736x^{13}  - 311478544324608x^{21} +O(x^{29}),\\
\sfZ_{S_{0,3,4},  \widetilde{B}}(x,1)&= -36864 x^{3}-123928576 x^{7}-125523673088 x^{11} +O( x^{15} ),\\
\sfZ_{S_{1,3,5},  \widetilde{F}}(x,1) &= 1248 x^{2}+3740160 x^{6}+3011270208 x^{10}+1143966397440 x^{14}  +O(x^{18}) ,\\
\sfZ_{S_{1,3,5},  \widetilde{F}+\widetilde B}(x,1) &=8 + 85920x^4 + 119219712x^8 + 63288183168x^{12}+O(x^{16}).
\end{align*}
Here $c_1,c_2$ always have the property that there are no rank 2 strictly Gieseker $H$-semistable sheaves with these Chern classes. The virtual refinement of $\sfZ_{S_{1,3,5},  \widetilde{F}}(x,1)$ is given by
\begin{align*}
\sfZ_{S_{1,3,5}, \widetilde{F}}(xy^{\frac{1}{2}},y)=\,&(216y^2 + 816y + 216)x^2 + (13720y^6 + 184128y^5\\
&+ 876264y^4 + 1591936y^3 + 876264y^2 + 184128y + 13720)x^6+O(x^{10}).
\end{align*}

\subsection{Divisors in products of projective spaces}

Let $S_d$ be a smooth surface of degree $d$ in $\PP^3$, then $$K_{S_d} = (d-4)L,$$ where $L$ is the hyperplane section. Moreover $S_d$ is simply connected by the Lefschetz hyperplane theorem. We take $d \geq 5$, so $S_d$ is of general type. One readily calculates
\begin{align*}
K_{S_d}^2 &= d(d-4)^2, \\ 
\chi(\O_{S_d}) &= \frac{1}{6} d (d^2+11) - d^2.
\end{align*}
The hyperplane section $L$ on $S_d$ is very ample so $|K_S|$ contains smooth connected canonical divisors. We test Conjecture \ref{conj} using Corollary \ref{univcor}.

We take as polarization $H = L$ and $c_1 = \epsilon L$.
If $S_d$ is very general, i.e. in the complement of countably many closed subvarieties in the projective space of hypersurfaces of degree $d$, then $\Pic(S_d)=\Z L$ by the Noether-Lefschetz Theorem.
For
\begin{align*}
c_1 H &= d \epsilon> 2 d (d-4) = 2 K_{S_d} H, \\
c_1 H &= d \epsilon  \quad \hbox{odd, or } S_d \hbox{ very general and } \epsilon \ \mathrm{ odd}
\end{align*}
there are no rank 2 strictly $\mu$-semistable sheaves with first Chern class $c_1$ and conditions (ii), (iii) of Corollary \ref{univcor} are satisfied. 

We also perform computations when these conditions are not satisfied (recall Remark \ref{strongform}). For $y=1$, we verified Conjecture \ref{conj} in the following cases: 
\begin{itemize}
\item $S_d$ such that $d=4,5,6$, $c_1 = \epsilon E$ with $\epsilon = -1,\ldots,4$, modulo $x^{\frac{5}{6}d(d^2+11)-5d^2+\epsilon d (\epsilon-2d+8)}$ (bound from Corollary \ref{univcor} (i)) and modulo $x^{-\epsilon^2 d - \frac{1}{2}d(d^2+11)+3d^2 + 4 \min\{0,\epsilon d(d-\epsilon-4)\} + 124}$ (bound from $A_i$).
\end{itemize}

\noindent For virtual $\chi_y$-genus, we tested Conjecture \ref{conj} for:
\begin{itemize}
\item $S_d$ such that $d=4,5,6$, $c_1 = \epsilon E$ with $\epsilon = 1,2$, modulo $x^{\frac{5}{6}d(d^2+11)-5d^2+\epsilon d (\epsilon-2d+8)}$ (bound from Corollary \ref{univcor} (i)) and modulo $x^{-\epsilon^2 d - \frac{1}{2}d(d^2+11)+3d^2 + 4 \min\{0,\epsilon d(d-\epsilon-4)\} + 32}$ (bound from $A_i$).
\end{itemize}

We list a few examples. First we deal with the specialization to the virtual Euler number and in this case we only consider the case that
$\epsilon>2(d-4)$, so that we do not use the strong from of Mochizuki's formula and the formulas are proved unconditionally.
For $d=5$ and $\epsilon=3,4$ 
we obtain
\begin{align*}\sfZ_{S_5,L,L}(x,1)=\,&8 + 52720x^4 + 48754480x^8 + 17856390560x^{12}  + 3626761297400x^{16} \\
&+ 482220775619120x^{20}+ 46283858505022160x^{24} +O(x^{28}),\\
\sfZ^\odd_{S_5,L,0}(x,1)=\,&- 316096x^5  - 70399202880x^{13}  - 1550539821466560x^{21} + O(x^{29}).
\end{align*}
For $d=5$, $\epsilon=1$ and using the strong form of Mochizuki's formula, we obtain the refinement
\begin{align*}
&\sfZ_{S_5,L,L}(xy^{\frac{1}{2}},y)=8 + (1280y^4 + 11440y^3+ 27280y^2+ 11440y+1280)x^4  \\
&\quad +(25520y^8+ 448640y^7+ 3228960y^6 + 11405320y^5+18537600y^4+\ldots)x^8 + O(x^{12}).
\end{align*}

Next we move on to smooth surfaces $S_{d_1,d_2}$ of bidegree $(d_1,d_2)$ in $\PP^2 \times \PP^1$. Again $S_{d_1,d_2}$ is simply connected and one readily calculates
\begin{align*}
K_{S_{d_1,d_2}}^2 &= (d_1-3)^2 d_2 + 2d_1(d_1-3)(d_2-2), \\
\chi(\O_{S_{d_1,d_2}}) &= \frac{1}{2}d_1(d_1d_2-d_1-3d_2+3)+d_2.
\end{align*}
We take $d_1 \geq 4$ and $d_2 \geq 3$ so $K_{S_{d_1,d_2}}$ is very ample and its linear system contains a smooth connected canonical divisor. 

Let $L_i$ be the restriction of $\pi_i^{*} L$ to $S$, where $\pi_1 : \PP^2 \times \PP^1 \rightarrow \PP^2$ and $\pi_2 : \PP^2 \times \PP^1 \rightarrow \PP^1$ denote projections and $L$ is the hyperplane class on each factor. For $y=1$, we tested Conjecture \ref{conj} using Corollary \ref{univcor} and the strong version of Mochizuki's formula (Remark \ref{strongform}) in the following cases:
\begin{itemize}
\item for $S_{d_1,d_2}$ such that $(d_1,d_2)=(4,3),(5,3),(6,3),(4,4),(5,4),(4,5)$, $c_1 = \epsilon_1 L_1 + \epsilon_2 L_2$ with $\epsilon_1=-3,\ldots,7, \epsilon_2 = -2,\ldots,8$, modulo $x^{M}$ (bound from Corollary \ref{univcor} (i)) and modulo $x^{N+124}$ (bound from $A_i$),
\end{itemize}

\noindent where
\begin{align*}
M :=\, &\frac{5}{2}d_1(d_1d_2-3d_2-d_1+3)+5d_2 + \epsilon_1^2 d_2 + 2 \epsilon_1 \epsilon_2 d_1 - 2 \epsilon_1 (d_1-3) d_2 - 2 \epsilon_2 (d_1-3) d_1 \\
&- 2\epsilon_1 (d_2-2) d_1, \\
N :=\, &- \epsilon_1^2 d_2 - 2 \epsilon_1 \epsilon_2 d_1- \frac{3}{2}d_1(d_1d_2-3d_2-d_1+3)-3d_2 \\
&+ 4\min\Big\{0,(d_1-3)(\epsilon_1-(d_1-3))d_2 + (d_1-3)(\epsilon_2 - (d_2-2))d_1 \\
&\qquad\qquad\qquad\qquad\qquad\qquad\qquad\qquad \ \!\! + (d_2-2)(\epsilon_1 - (d_1-3))d_1\Big\}.
\end{align*}

\noindent For general $y$, we checked Conjecture \ref{conj} in the following cases:
\begin{itemize}
\item for $S_{d_1,d_2}$ such that $(d_1,d_2)=(4,3), (4,4)$, $c_1 = \epsilon_1 L_1 + \epsilon_2 L_2$ with $\epsilon_1=-1,\ldots,3, \epsilon_2 =  -2,\ldots,5$, modulo $x^{M}$ (bound from Corollary \ref{univcor} (i)) and modulo $x^{N+32}$ (bound from $A_i$).
\end{itemize}

\noindent Examples contained in these tables with $(d_1,d_2) = (4,3)$ are 
\begin{align*}
\sfZ_{S_{4,3},3 L_1 }(x,1)&=128 + 5350656x^4 + 18196176128x^8 + 20761214894592x^{12} +O(x^{16}),\\
\sfZ_{S_{4,3},3 L_2 }^\odd(x,1) &= - 2704756736 x^{7} +O(x^{15}),
\end{align*}
where we choose a polarization $H$ such that $L_1 H$ is odd in the first case. A virtual refinement of the second formula is given by
\begin{align*}
\sfZ_{S_{4,3},3 L_2 }^\odd(xy^{\frac{1}{2}},y) =\, &- \big(6323328y^{7} +81371136y^{6} +394518784y^{5} + 870165120(y^4+y^3) \\
&+394518784y^{2}+81371136y+6323328\big)x^7 + O(x^{15}).
\end{align*}

Finally we consider smooth surfaces $S_{d_1,d_2,d_3}$ of tridegree $(d_1,d_2,d_3)$ in $\PP^1 \times \PP^1 \times \PP^1$. Then $S_{d_1,d_2,d_3}$ is simply connected and one can compute 
\begin{align*}
K_{S_{d_1,d_2,d_3}}^2 &= 2 \prod_{{\scriptsize{\begin{array}{c} (i,j,k) \\ i,j,k \, \mathrm{distinct} \end{array}}}} d_i (d_j - 2)(d_k-2) , \\ 
\chi(\O_{S_{d_1,d_2,d_3}}) &= \frac{1}{2} d_1d_2d_3-\frac{1}{6}d_1d_2-\frac{1}{6}d_1d_3-\frac{1}{6} d_2 d_3-\frac{1}{3}d_1-\frac{1}{3}d_2-\frac{1}{3}d_3+2.
\end{align*}
Taking $d_1,d_2,d_3 \geq 3$, the canonical linear system is very ample and contains smooth connected curves. Denote by $L_i$ the restriction of $\pi_i^{*} L$ to $S$, where $\pi_i : \PP^1 \times \PP^1 \times \PP^1 \rightarrow \PP^1$ are the projections and $L$ is the class of a point on $\PP^1$. For $y=1$, we tested Conjecture \ref{conj} using Corollary \ref{univcor} and the strong version of Mochizuki's formula (Remark \ref{strongform}) in the following cases:
\begin{itemize}
\item for $S_{d_1,d_2,d_3}$ such that $(d_1,d_2,d_3)=(3,3,3), (3,3,4), (3,3,5),(3,4,4)$, $c_1 = \epsilon_1 L_1 + \epsilon_2 L_2 + \epsilon_3 L_3$ with $\epsilon_1, \epsilon_2, \epsilon_3 = -2,\ldots,5$, modulo $x^{M}$ (bound from Corollary \ref{univcor} (i)) and modulo $x^{N+124}$ (bound from $A_i$),
\end{itemize}

\noindent where
\begin{align*}
M := \, &\frac{5}{2} d_1d_2d_3-\frac{5}{6}d_1d_2-\frac{5}{6}d_1d_3-\frac{5}{6} d_2 d_3-\frac{5}{3}d_1-\frac{5}{3}d_2-\frac{5}{3}d_3+10 \\
&- 2 \prod_{{\scriptsize{\begin{array}{c} (i,j,k) \\ i,j,k \, \mathrm{distinct} \end{array}}}} (d_i-2) d_j e_k + 2 \prod_{{\scriptsize{\begin{array}{c} (i,j,k) \\ i,j,k \, \mathrm{distinct} \end{array}}}} e_i e_j d_k, \\
N := \, &-\frac{3}{2} d_1d_2d_3+\frac{1}{2}d_1d_2+\frac{1}{2}d_1d_3+\frac{1}{2} d_2 d_3+d_1+d_2+d_3-6 \\
&-\prod_{{\scriptsize{\begin{array}{c} (i,j,k) \\ i,j,k \, \mathrm{distinct} \end{array}}}} e_i e_j d_k +4\min\Big\{0,\prod (d_i-2)(e_j - (d_j-2)) d_k \Big\}.
\end{align*}

\noindent For arbitrary $y$, we checked Conjecture \ref{conj} in the following cases:
\begin{itemize}
\item for $S_{d_1,d_2,d_3}$ such that $(d_1,d_2,d_3)=(3,3,3),(3,3,4)$, $c_1 = \epsilon_1 L_1 + \epsilon_2 L_2 + \epsilon_3 L_3$ with $\epsilon_1, \epsilon_2, \epsilon_3 =-1,\ldots,3$, modulo $x^{M}$ (bound from Corollary \ref{univcor} (i)) and modulo $x^{N+32}$ (bound from $A_i$).
\end{itemize}

\noindent An example covered by these tables with $(d_1,d_2,d_3) = (3,3,3)$ is the following 
$$
\sfZ_{S, -L_1+2L_2+2L_3}(x,1) =-147456 x -8358985728 x^5 +O( x^{9}),
$$
where we choose a polarization $H$ such that $L_1 H$ is odd. Its virtual refinement is
$$
\sfZ_{S,3L_1}(xy^{\frac{1}{2}},y)=\sfZ_{S, -L_1+2L_2+2L_3}(xy^{\frac{1}{2}},y)= (-73728y- 73728)x+O(x^5).
$$

\subsection{Complete intersections in projective spaces}

For $d_1,d_2\in \Z_{\ge 2}$, with $d_1+d_2\ge 6$, let $S_{d_1,d_2}$ be a smooth complete intersection of bidegree $(d_1,d_2)$ in $\PP^4$. 
Then 
$S_{d_1,d_2}$ is simply connected. Let $L$ be the restriction of the hyperplane class on $\PP^4$ to $S_{d_1,d_2}$.
If $S_{d_1,d_2}$ is very general, then by the Noether-Lefschetz theorem for complete intersections (see e.g.~\cite{Kim})
the Picard group of $S_{d_1,d_2}$ is generated by $L$.
Putting $d:=d_1d_2$, $D:=d_1+d_2$, we have $K_{S_{d_1,d_2}}=(D-5)L$ is very ample and
\begin{align*}
K_{S_{d_1,d_2}}^2&=d(D-5)^2, \\ 
\chi(\O_{S_{d_1,d_2}})&=\frac{d}{12}(2D^2-15D-d+35).
\end{align*}
If the Picard group of $S_{d_1,d_2}$ is $\Z L$ and $H=L$, then there are no rank $2$ strictly $\mu$-semistable sheaves on $S_{d_1,d_2}$ with $c_1=\epsilon L$ and $\epsilon$ odd.

\noindent For $y=1$, we checked Conjecture \ref{conj} in the following cases:
\begin{itemize}
\item for $S_{d_1,d_2}$ such that $(d_1,d_2)=(2,4),(2,5),(3,3),(3,4)$, $c_1 = \epsilon  L$ with $\epsilon=-1,\ldots,3$, modulo $x^{M}$ (bound from Corollary \ref{univcor} (i)) and modulo $x^{N+124}$ (bound from $A_i$).
\end{itemize}

\noindent Here 
\begin{align*}
M&:=\epsilon^2d-2\epsilon d(D-5)+\frac{5d}{12}(2D^2-15D-d+35),\\
N&:=-\epsilon^2d-\frac{3d}{12}(2D^2-15D-d+35)+4\min\Big\{0,d(D-5)(\epsilon-(D-5))\Big\}.
\end{align*}
For $y$ general, we checked Conjecture \ref{conj} in the following cases:
\begin{itemize}
\item for $S_{d_1,d_2}$ such that $(d_1,d_2)=(2,4), (3,3)$, $c_1 = \epsilon  L$ with $\epsilon=1$, modulo $x^{M}$ (bound from Corollary \ref{univcor} (i)) and modulo $x^{N+32}$ (bound from $A_i$).
\end{itemize}

\noindent For $\epsilon=3$ we determined $\sfZ_{S_{3,3},L,L}(x,1)$, and $\sfZ_{S_{2,4},L,L}(x,1)$ modulo $x^{25}$ and $x^{34}$, where the second is only written modulo $x^{26}$.
\begin{align*}
\sfZ_{S_{3,3},L,L}(x,1)=\,&-1152x - 11784960x^5 - 18762235136x^9 - 11903890079232x^{13} \\
&- 4135344957021312x^{17} - 924519456314916096x^{21} + O(x^{25}),\\
\sfZ_{S_{2,4},L,L}(x,1)=\,&6912x^2 + 30124032x^6 + 31867565056x^{10} + 15237098061824x^{14} \\ 
&+ 4243875728564736x^{18} + 789670403161694208x^{22} +O(x^{26}).
\end{align*}
For $\epsilon=1$ and using the strong form of Mochizuki's formula (Remark \ref{strongform}), we get
\begin{align*}
\chi_{-y}^\vir(M^L_{S_{3,3}}(2,L,7))&=-576y - 576, \\
\chi_{-y}^\vir(M^L_{S_{2,4}}(2,L,7))&=1344y^2 + 4224y + 1344.
\end{align*}

For $d_1,d_2,d_3\in \Z_{\ge 2}$, with $d_1+d_2+d_3\ge 7$, let $S_{d_1,d_2,d_3}$ be a smooth complete intersection of tridegree $(d_1,d_2,d_3)$ in $\PP^5$. Then 
$S_{d_1,d_2,d_3}$ is simply connected. When $S_{d_1,d_2,d_3}$ is very general, the Picard group of $S_{d_1,d_2,d_3}$ is generated by the restriction $L$ of a hyperplane class on $\PP^5$.
Putting $d:=d_1d_2d_3$, $D:=d_1+d_2+d_3$, we have $K_{S_{d_1,d_2,d_3}}=(D-6)L$ is very ample and 
\begin{align*}
K_{S_{d_1,d_2,d_3}}^2&=d(D-6)^2,\\ 
\chi(\O_{S_{d_1,d_2,d_3}})&=\frac{d}{12}(2D^2-18D+51-d_1d_2-d_1d_3-d_2d_3).
\end{align*}
\noindent For $y=1$, we checked Conjecture \ref{conj} in the following cases:
\begin{itemize}
\item for $S_{d_1,d_2,d_3}$ with  $(d_1,d_2,d_3)=(2,2,3)$, $c_1 = \epsilon  L$ such that $\epsilon=0,1,2$, modulo $x^{12\epsilon^2-24\epsilon+35}$ (bound from Corollary \ref{univcor} (i)) and modulo $ x^{48\min\{0,\epsilon-1\}-12\epsilon^2+103}$ (bound from $A_i$).
\end{itemize}

\noindent As an example we get
\begin{align*}
\sfZ_{S_{2,2,3},L,L}(x,1)=&-1261568x^3 - 7379091456x^7 - 11717181702144x^{11}\\
&- 8585117244063744x^{15} - 3662336916158939136x^{19} + O(x^{23}).
\end{align*}

\subsection{Verifications of Conjecture \ref{numconj}} \label{numconjevidence}

Let $\underline{\beta} \in \Z^4$ and $(\gamma_1,\gamma_2)\in \Z^2$ such that $\beta_1 \equiv \beta_2 \mod 2$ and $\beta_3 \geq \beta_4 - 3$. 
Put 
\begin{align*}
M_1&:=\min\Big\{2\beta_1-2\beta_2+8\beta_4, 124-4\gamma_1+4\gamma_2,124+4\beta_2-4\beta_3+4\gamma_1-4\gamma_2\Big\}-\beta_1-3\beta_4,\\
M_2&:= \min\Big\{2\beta_1-2\beta_2+8\beta_4, 32-4\gamma_1+4\gamma_2,32+4\beta_2-4\beta_3+4\gamma_1-4\gamma_2\Big\}-\beta_1-3\beta_4.
\end{align*}
For $y=1$, we verified Conjecture \ref{numconj} up to order $x^{M_1-1}$ in the following cases:
\begin{itemize}
\item $\gamma_1=\gamma_2=0$ and any $|\beta_1|, |\beta_2|, |\beta_3|, |\beta_4| \le 16$,
\item any $|\gamma_1|, |\gamma_2|,  |\beta_1|, |\beta_2|, |\beta_3|, |\beta_4|  \le 4$.
\end{itemize}
For arbitrary $y$, we verified Conjecture \ref{numconj} up to order $x^{M_2-1}$ in the following cases:
\begin{itemize}
\item $\gamma_1=\gamma_2=0$ and any $|\beta_1|, |\beta_2|, |\beta_3|, |\beta_4| \le 3$,
\item any $|\gamma_1|, |\gamma_2|,  |\beta_1|, |\beta_2|, |\beta_3|, |\beta_4|  \le 2$.
\end{itemize}
We have examples showing that the condition $\beta_3 \geq \beta_4 - 3$ is necessary.

\begin{remark}
Any choice of $\underline{\beta} \in \Z^4$ with $\beta_1 \equiv \beta_2 \mod 2$ satisfying 
$$
\frac{1}{12} \beta_3 \leq \beta_4 < \frac{1}{9} \beta_3
$$ 
does \emph{not} arise geometrically and automatically implies $\beta_3 \geq \beta_4-3$. These inequalities correspond to ``$e(S) \geq 0$ but violation of the Bogomolov-Miyaoka-Yau inequality'', so in this regime $\underline{\beta}$ cannot be realized by a smooth projective algebraic surface. With these restrictions, we verified Conjecture \ref{numconj} for $y=1$ up to order $x^{M_1-1}$ for all
\begin{align*}
\gamma_1=\gamma_2=0 \hbox{ and } |\beta_1|, |\beta_2|, |\beta_3|, |\beta_4| &\le 25,
\end{align*}
and for arbitrary $y$ up to order $x^{M_2-1}$ for all
\begin{align*}
\gamma_1=\gamma_2=0 \hbox{ and } |\beta_1|, |\beta_2|, |\beta_3|, |\beta_4|  &\le 27.
\end{align*}
\end{remark}

\appendix
\section{$A_1, \ldots, A_7$} \label{app1}

We calculated $A_1(s,1,q), \ldots, A_7(s,1,q)$ up to order $q^{30}$ and list them up to order $q^4$. 
\begingroup
\allowdisplaybreaks
{\scriptsize{\begin{align*}
A_1 =&\, 1-\frac{5}{8 s^3}\,{q}+ \left( \frac{3}{2}-\frac{3}{s^2}+{\frac {27}{8\,{s}
^{4}}}-{\frac {45}{128\,{s}^{6}}} \right) {q}^{2}+ \left( -\frac{6}{s}
+{\frac {375}{16\,{s}^{3}}}-{\frac {153}{8\,{s}^{5}}}+{\frac {273}{64
\,{s}^{7}}}-{\frac {337}{1024\,{s}^{9}}} \right) {q}^{3} \\
&+ \left( -{
\frac {181}{8}}+{\frac {193}{2\,{s}^{2}}}-{\frac {2649}{16\,{s}^{4}}}+
{\frac {30741}{256\,{s}^{6}}}-{\frac {4977}{128\,{s}^{8}}}+{\frac {
6213}{1024\,{s}^{10}}}-{\frac {12097}{32768\,{s}^{12}}} \right) {q}^{4}+O(q^5), \\
A_2 =&\, 1+\frac{1}{8}\,{\frac {q}{{s}^{3}}}+ \left( -\frac{3}{2 s^2}+\frac{3}{8 s^4}+{
\frac {1}{128\,{s}^{6}}}+\frac{3}{2} \right) {q}^{2}+ \left( {\frac {181}{16\,
{s}^{3}}}-{\frac {99}{16\,{s}^{5}}}+{\frac {75}{64\,{s}^{7}}}-{\frac {
69}{1024\,{s}^{9}}}-\frac{6}{s} \right) {q}^{3} \\
&+ \left( {\frac {14423
}{256\,{s}^{6}}}-{\frac {5021}{32768\,{s}^{12}}}+{\frac {2763}{1024\,{
s}^{10}}}-{\frac {699}{8\,{s}^{4}}}-{\frac {293}{8}}+{\frac {283}{4\,{
s}^{2}}}-{\frac {4593}{256\,{s}^{8}}} \right) {q}^{4} + O(q^5), \\
A_3 =& \, 1+\frac{3}{4}\,{\frac {q}{{s}^{3}}}+ \left( 3-\frac{9}{4 s^4}+\frac{3}{4 s^6}
 \right) {q}^{2}+ \left( 8\,s+{\frac {27}{32\,{s}^{9}}}+\frac{1}{s^3}-{
\frac {83}{16\,{s}^{7}}}-\frac{6}{s}+\frac{15}{2 s^5} \right) {q}^{3} \\
&+\left( -57+42\,{s}^{2}-{\frac {411}{16\,{s}^{6}}}+{\frac {8301}{8192
\,{s}^{12}}}-{\frac {2421}{256\,{s}^{10}}}-{\frac {243}{16\,{s}^{4}}}+
\frac{51}{s^2}+{\frac {3603}{128\,{s}^{8}}} \right) {q}^{4} +O(q^5), \\
A_4 =&\, 1+ \left( -1+\frac{3}{2 s^2}+\frac{1}{4 s^3} \right) q+ \left( \frac{3}{2}+\frac{6}{s}
+\frac{3}{4s^2}-{\frac {25}{4\,{s}^{3}}}-\frac{3}{4 s^4}+\frac{5}{4 s^5}+{\frac {39}{128\,{s}^{6}}} \right) {q}^{2} \\
&+ \left( {\frac {31}{2}}-{\frac {159}{4\,{s}^{2}}}-{\frac {19}{8\,{s}^{3}}}+{\frac {285}{8\,{s}^
{4}}}+{\frac {61}{16\,{s}^{5}}}-{\frac {1639}{128\,{s}^{6}}}-{\frac {
21}{8\,{s}^{7}}}+{\frac {389}{256\,{s}^{8}}}+{\frac {215}{512\,{s}^{9}
}} \right) {q}^{3} \\
&+ \left( 56\,s-{\frac {85}{8}}-\frac{165}{s}+{\frac 
{39}{8\,{s}^{2}}}+{\frac {2141}{8\,{s}^{3}}}+\frac{6}{s^4}-{\frac {3819
}{16\,{s}^{5}}}-{\frac {5991}{256\,{s}^{6}}}+{\frac {7083}{64\,{s}^{7}
}}+{\frac {10127}{512\,{s}^{8}}}-{\frac {12691}{512\,{s}^{9}}}-{\frac 
{779}{128\,{s}^{10}}}+{\frac {133}{64\,{s}^{11}}} \right.\\
&\left. +{\frac {20047}{32768 \,{s}^{12}}} \right) {q}^{4} + O(q^5), \\
A_5 =&\, 1+ \left( -2+\frac{1}{s^2} \right) q+ \left( -6\,s+4-\frac{1}{2 s^2}-\frac{9}{2s^3}-\frac{1}{4s^4}+\frac{5}{8s^5}+{\frac {5}{64\,{s}^{6}}}+\frac{6}{s}
 \right) {q}^{2} \\
 &+ \left( 20+24\,s-16\,{s}^{2}-{\frac {237}{32\,{s}^{6}
}}+{\frac {9}{64\,{s}^{9}}}+\frac{3}{s^3}-{\frac {9}{8\,{s}^{7}}}+\frac{26}{s^4}
-\frac{12}{s}-{\frac {67}{2\,{s}^{2}}}+\frac{7}{4 s^5}+{\frac {45
}{64\,{s}^{8}}} \right) {q}^{3} \\
&+ \left( -100+86\,s+134\,{s}^{2}-56\,{s
}^{3}-{\frac {357}{32\,{s}^{6}}}-{\frac {831}{64\,{s}^{9}}}+{\frac {
1871}{8192\,{s}^{12}}}+\frac{219}{s^3}-{\frac {701}{256\,{s}^{10}}}+{
\frac {1101}{16\,{s}^{7}}}-{\frac {173}{8\,{s}^{4}}}-\frac{141}{s}+\frac{72}{s^2} \right. \\
&\left.-{\frac {5545}{32\,{s}^{5}}}+{\frac {1369}{128\,{s}^{8}}}+{
\frac {237}{256\,{s}^{11}}} \right) {q}^{4} + O(q^5), \\
A_6 =&\, 1+ \left( -2\,s-\frac{1}{2 s^2}-\frac{1}{8 s^3} \right) q+ \left( -4\,{s}^
{2}-2\,s-\frac{2}{s}-\frac{1}{2 s^2}+\frac{9}{4 s^3}+\frac{1}{s^4}-\frac{1}{4 s^5}-\frac{1}{8 s^6} \right) {q}^{2} \\
&+ \left( -8\,{s}^{2}+10\,s+6+\frac{3}{s}+\frac{13}{s^2}+{\frac {51}{8\,{s}^{3}}}-{\frac {43}{4\,{s}^{4}}}-{
\frac {217}{32\,{s}^{5}}}+{\frac {83}{32\,{s}^{6}}}+{\frac {57}{32\,{s
}^{7}}}-{\frac {27}{128\,{s}^{8}}}-{\frac {77}{512\,{s}^{9}}} \right) 
{q}^{3} \\
&+ \left( -24\,{s}^{3}+92\,{s}^{2}+52\,s+10+\frac{23}{s}-\frac{1}{2 s^2}
-{\frac {321}{4\,{s}^{3}}}-{\frac {427}{8\,{s}^{4}}}+{\frac {949
}{16\,{s}^{5}}}+{\frac {767}{16\,{s}^{6}}}-{\frac {657}{32\,{s}^{7}}}-
{\frac {2281}{128\,{s}^{8}}}+{\frac {869}{256\,{s}^{9}}} \right. \\
&\left.+{\frac {195}{
64\,{s}^{10}}}-{\frac {55}{256\,{s}^{11}}}-{\frac {101}{512\,{s}^{12}}
} \right) {q}^{4} + O(q^5), \\
A_7 =&\, 1+ \left( 24\,s-\frac{6}{s} \right) q+ \left( 360\,{s}^{2}-180+\frac{30}{s^2}-\frac{9}{4 s^4}+{\frac {3}{32\,{s}^{6}}} \right) {q}^{2} \\
&+ \left( 
4160\,{s}^{3}-3200\,s+\frac{1020}{s}-\frac{210}{s^3}+{\frac {135}{4\,{s}
^{5}}}-{\frac {55}{16\,{s}^{7}}}+{\frac {5}{32\,{s}^{9}}} \right) {q}^{3} \\
&+ \left( 40560\,{s}^{4}-43380\,{s}^{2}+20280-\frac{6480}{s^2}+{\frac 
{7065}{4\,{s}^{4}}}-{\frac {6255}{16\,{s}^{6}}}+{\frac {975}{16\,{s}^{
8}}}-{\frac {735}{128\,{s}^{10}}}+{\frac {495}{2048\,{s}^{12}}}
 \right) {q}^{4} +O(q^5).
\end{align*}}}
\endgroup

\section{Nekrasov partition function} \label{Nekrasovsec}

In Section \ref{toricsec}, we considered a toric surface $S$ and we reduced $\sfZ_S(a_1,c_1,s,y,q)$ to a purely combinatorial expression on each toric patch $U_\sigma \cong \C^2$. We now study these local contributions on the toric patches in terms of the Nekrasov partition function. The content of this section is not used elsewhere in this paper. For simplicity we restrict attention to the case of virtual Euler characteristics, i.e.~$y=1$, although very similar arguments work for virtual $\chi_y$-genus. In particular, we will see that 
$$
\sfZ_S(a_1,c_1,s,q):=\sfZ_S(a_1,c_1,s,1,q)
$$ 
can be expressed in terms of four universal functions. We use the notation from \emph{Lectures on instanton counting} by Nakajima-Yoshioka \cite{NY1} (see also \cite{GNY1}).

Let $M(n)$ be the framed moduli space of pairs $(E,\Phi)$. Here $E$ is a rank 2 torsion free sheaf on $\PP^2$ with $c_2(E)=n$ and locally free in a neighbourhood of the ``line at infinity'' $\ell_\infty \subset \PP^2$. Furthermore 
$$
\Phi : E|_{\ell_\infty} \stackrel{\cong}{\rightarrow} \O_{\ell_\infty}^{\oplus 2}
$$
denotes the framing. The moduli space $M(n)$ is a fine moduli space and is a smooth quasi-projective variety of dimension $4n$. Let $T = \C^{*2}$ acting on $\C^2 = \PP^2 \setminus \ell_\infty$ by
$$
(t_1,t_2)\cdot (x,y) = (t_1 x, t_2 y).
$$
This action lifts to $M(n)$, which has an additional $\C^*$ action by scaling the framing
$$
(s_1,s_2) \in \O_{\ell_\infty}^{\oplus 2} \mapsto (e^{-1}s_1,es_2) \in \O_{\ell_\infty}^{\oplus 2}.
$$
We denote the corresponding equivariant parameters of these actions by $\epsilon_1, \epsilon_2, a$. Following the conventions of \cite{GNY1}, we write their characters as $e^{\epsilon_1}, e^{\epsilon_2}, e^{a}$. The $\Gamma = T \times \C^*$ fixed locus of $M(n)$ is given by $\Hilb^n(\C^2  \sqcup \C^2)^T$, where $\C^2 = \PP^2 \setminus \ell_{\infty}$ as described in \cite{NY1}. In particular, the fixed locus consists of finitely many isolated reduced points indexed by pairs of partitions $(\lambda,\mu)$ satisfying $|\lambda|+|\mu|=n$ as in Section \ref{toricsec}. Concretely, the pair $(\lambda,\mu)$ corresponds to the direct sum of ideal sheaves $I_{Z_\lambda} \oplus I_{Z_\mu}$.

The instanton part of the Nekrasov partition function with one adjoint matter $M$ and one fundamental matter $m$ is defined as follows
$$
\sfZ^{\mathrm{inst}}(\epsilon_1,\epsilon_2,a,m,M,q) := \sum_{n=0}^{\infty} q^n \int_{M(n)} \Eu(T_{M(n)} \otimes e^M) \, \Eu(\mathcal{V} \otimes e^m).
$$
Since $M(n)$ is non-compact, the above integral is defined by the $\Gamma$-localization formula. Here $\Eu(\cdot)$ is the equivariant Euler class with respect to two trivial torus actions with equivariant parameters $m,M$ (and, after localization, it also becomes equivariant with respect to $\Gamma$). Furthemore $T_{M(n)}$ denotes the tangent bundle and $\mathcal{V}$ denotes the rank $n$ vector bundle defined by
$$
\mathcal{V} := R^1 q_{2*} \Big( \mathcal{E} \otimes q_1^* \O(-\ell{_\infty}) \Big),
$$
where $\mathcal{E}$ is the universal sheaf on $\PP^2 \times M(n)$ and $q_i$ is projection to the $i$th factor. We note that the term in $\sfZ^{\mathrm{inst}}$ corresponding to $n=0$ is equal to 1.

In the previous section, we encountered the following expression
\begin{align}
\begin{split} \label{Zcomb}
\sfZ(\epsilon_1,\epsilon_2,a_1,c_1,s,q) := \, &\sum_{(\lambda,\mu)}  q^{|\lambda|+|\mu|} \frac{\Eu(H^0(\O(a_1)|_{Z_{\lambda}}))}{\Eu(T_{Z_{\lambda}})} \frac{\Eu(H^0(\O(c_1-a_1)|_{Z_{\mu}}) \otimes \s^2)}{\Eu(T_{Z_{\mu}})} \\
&\times \frac{c(E_{n_1,n_2}|_{(Z_{\lambda},Z_{\mu})})}{\Eu(E_{n_1,n_2}|_{(Z_{\lambda},Z_{\mu})} - T_{Z_{\lambda}} - T_{Z_{\mu}})},
\end{split}
\end{align}
where the sum is over all pairs of partitions $(\lambda, \mu)$. Moreover, we view $a_1, c_1$ as equivariant parameters (of trivial torus actions) by replacing $\O(c_1-a_1)$ by $e^{c_1-a_1}$ etc. In this section, all Chern classes and Euler classes are equivariant with respect to all tori. 
\begin{proposition}
$$
\sfZ(\epsilon_1,\epsilon_2,a_1,c_1,s,q) = \sfZ^{\mathrm{inst}}\Big(\epsilon_1,\epsilon_2,s+\frac{c_1-2a_1}{2},s+\frac{c_1}{2},0,q \Big).
$$
\end{proposition}
\begin{proof}
The $\Gamma$-representation of $T_{(E,\Phi)} M(n) = \Ext^1_{\PP^2}(E,E(-\ell_\infty))$ at a $\Gamma$-fixed point is described in \cite[Thm.~3.2]{NY1}. After a bit of rewriting, much like in \cite{GNY1}, it becomes
\begin{align*}
-R\Hom_{\C^2}(I_{Z_\lambda}, I_{Z_\lambda})_0 -R\Hom_{\C^2}(I_{Z_\mu}, I_{Z_\mu})_0 &+ R\Gamma(\O_{\C^2} \otimes e^{2a}) - R\Hom_{\C^2}(I_{Z_\lambda}, I_{Z_\mu} \otimes e^{2a}) \\
&+ R\Gamma(\O_{\C^2} \otimes e^{-2a}) - R\Hom_{\C^2}(I_{Z_\mu}, I_{Z_\lambda} \otimes e^{-2a}).
\end{align*}
Referring to Definition \ref{defZX}, we conclude that we want to specialize $a = s + \frac{c_1-2a_1}{2}$. The specialization $m = s+\frac{c_1}{2}$ can be deduced from the fact that the fibre of $\mathcal{V}$ over $Z_\lambda \oplus Z_\mu \in \Hilb^n(\C^2 \sqcup \C^2)^T$ is given by
$$
H^1(I_{Z_\lambda}(-\ell_\infty)) \oplus H^1(I_{Z_\mu}(-\ell_\infty)) \cong H^0(\O_{Z_\lambda}) \oplus H^0(\O_{Z_\mu}), 
$$
where the first factor has weight $e^{-a}$ and the second $e^a$ with respect to the framing action. The specialization $M=0$ comes from the fact that we are interested in total Chern class, i.e.~the virtual Euler characteristic specialization.
\end{proof}

\begin{remark}
Similar to \cite{NY1}, we define
$$
F^{\mathrm{inst}}(\epsilon_1,\epsilon_2,a,m,M,q) := \log \sfZ^{\mathrm{inst}}(\epsilon_1,\epsilon_2,a,m,M,q).
$$
The Nekrasov conjecture, originally formulated in \cite{Nek2} and studied in various contexts e.g.~in \cite{NY2, NO, BE}, states that $F^{\mathrm{inst}}$ (in its original setting) is regular at 
$(\epsilon_1,\epsilon_2) = (0,0)$ and identifies its value at $(0,0)$ with the corresponding Seiberg-Witten prepotential, an expression in terms of the periods of the corresponding Seiberg-Witten curve, which is typically a family of elliptic curves. In the case of the partition function with one adjoint and one fundamental matter however, the Seiberg-Witten curve, and thus the  Seiberg-Witten prepotential, are not available. Nevertheless, it is natural to conjecture\footnote{We verified this using the toric calculations of Section \ref{toricsec} for (1) $M=0$ and order $q^{\leq 2}$ and (2) $M=0$, $\epsilon_1 = -\epsilon_2$, and order $q^3$.} 
\begin{equation} \label{weakNEK}
\epsilon_1 \epsilon_2 F^{\mathrm{inst}}(\epsilon_1,\epsilon_2,a,m,M,q) \textrm{ is regular at } (\epsilon_1,\epsilon_2) = (0,0).
\end{equation}
Since $F^{\mathrm{inst}}(\epsilon_1,\epsilon_2,a,m,M,q)$ is symmetric under $\epsilon_1 \leftrightarrow \epsilon_2$, \eqref{weakNEK} allows us to write
\begin{align}
\begin{split} \label{leadnekr}
\epsilon_1 \epsilon_2 F^{\mathrm{inst}}(\epsilon_1,\epsilon_2,a,m,M,q) = \, &F_0(a,m,M,q) + (\epsilon_1+\epsilon_2) H(a,m,M,q) + \\
&(\epsilon_1+\epsilon_2)^2 G_1(a,m,M,q) + \epsilon_1 \epsilon_2 G_2(a,m,M,q) + \cdots,
\end{split}
\end{align}
where $\cdots$ stands for terms of order $\epsilon_1^i \epsilon_2^j$ with $i+j \geq 3$.
\end{remark}

\begin{remark} \label{4univ}
For any toric surface $S$ and $a_1, c_1 \in A^1_{T}(S)$, we have (\ref{defZX}, \eqref{Zcomb})
\begin{equation} \label{prodcharts}
\sfZ_S(a_1,c_1,s,1,q) = \prod_{\sigma=1}^{e(S)} \sfZ(v_\sigma,w_\sigma,(a_1)_\sigma,(c_1)_\sigma,s,q).
\end{equation}
Assume \eqref{weakNEK} holds. Combining \eqref{prodcharts}, \eqref{leadnekr} with localization on $S$ gives
\begin{align}
\begin{split} \label{logZ}
\log \sfZ_S(a_1,c_1,s,q) =\, &\frac{1}{8}(c_1-2a_1)^2 \, \frac{\partial^2 F_0}{\partial a^2}  + \frac{1}{4}(c_1-2a_1)c_1 \, \frac{\partial^2 F_0}{\partial a \partial m}   + \frac{1}{8} c_1^2 \, \frac{\partial^2 F_0}{\partial m^2} \\
&+\frac{1}{2}(c_1-2a_1) K_S \,  \frac{\partial H}{\partial a} +\frac{1}{2} c_1 K_S \, \frac{\partial H}{\partial m}   +K_S^2 \, G_1   + c_2(S) \, G_2 .
\end{split}
\end{align}
Here the right hand side is evaluated at $(a,m,M,q) = (s,s,0,q)$.
See \cite[Proof of Thm.~4.2]{GNY1} for a similar, but more complicated, calculation. 

Assume \eqref{weakNEK} holds. From Proposition \ref{univ} and the fact that the $A_i$ are determined on $\PP^2$, $\PP^1 \times \PP^1$, we deduce that \eqref{logZ} holds for \emph{any} smooth projective surface $S$ and $a_1, c_1 \in A^1(S)$. Therefore $ \sfZ_S(a_1,c_1,s,1,q)$ is determined by four universal functions $F_0,H,G_1,G_2$. 
We do not know the statement of the Nekrasov conjecture in this context, which would be an explicit conjectural formula for $F_0,H,G_1,G_2$, possibly in terms of periods of a family of elliptic curves, as studied in many cases e.g.~in \cite{NY2, NO, BE, GNY3}. 
One possibility to approach the $y=1$ specialization of Conjecture \ref{conj} (and Conjectures \ref{numconj}, \ref{generalsurfconj} below), would be via first finding a formulation of the Nekrasov conjecture in this context, and then a solution, employing  strategies  somehow related to the ones used in \cite{NY2, NO, BE, GNY3}. However, this seems to be very difficult, because the corresponding Seiberg-Witten curve is not available.
\end{remark}

\section{\\ Vafa-Witten formula with $\mu$-classes \\
\emph{by Lothar G\"ottsche and Hiraku Nakajima}} \label{appC}

Let $S$ be a projective algebraic surface with $b_1(S)=0$ and $p_g(S)>0$. For an ample line bundle $H$ on $S$, we denote again $M=M^H_S(c_1,c_2)$ the 
moduli space of $H$-semistable rank $2$ sheaves on $S$ with Chern classes $c_1\in H^2(S,\Z)$ and $c_2\in H^4(S,\Z)=\Z$. We assume that there are no rank $2$ strictly Gieseker $H$-semistable sheaves with first Chern class $c_1$.
Let $\E$ be a universal sheaf over $M$. For $\beta\in H_k(S,\Q)$ we denote $$\mu(\beta):=\pi_{M*}\big((c_2(\E)-c_1(\E)^2/4)\cap \pi_S^*\beta\big)\in H^{4-k}(M,\Q).$$
Formally we can write $(c_2(\E)-c_1(\E)^2/4)=-\ch_2(\E\otimes \det(\E)^{-1/2})$.

The Witten conjecture for $S$ expresses the generating function of the Donaldson invariants of $S$   in terms of the Seiberg-Witten invariants of $S$.
On the other hand the version of the Vafa-Witten conjecture of this paper (Conjecture \ref{generalsurfconj} with $y=1$) expresses the generating function of the  virtual Euler numbers of the $M^H_S(c_1,c_2)$
in terms the Seiberg-Witten invariants of $S$. 
We want to give a common conjectural generalization of both formulas. 
For $\alpha\in H_2(S,\Q)$ and $p\in H_0(S,\Z)$ the class of a point, the Donaldson generating function is
\begin{align*}\sfD^S_{c_1}(\alpha z+p u, x):=&\,\sum_{n} \int_{[M^H_S(c_1,n)]^\vir}\exp\big(\mu(\alpha z+p u)\big) x^{\vd(S,c_1,n)}\\=&
\sum_{n}\sum_{l+2m=\vd(S,c_1,n)}\int_{[M^H_S(c_1,n)]^\vir}\mu(\alpha)^l\mu(p)^m\frac{z^l}{l!}\frac{u^m}{m!}x^{l+2m}.
\end{align*}
Here again $\vd(S,c_1,n)=4n-c_1^2-3\chi(\O_S)$ is the expected dimension of the moduli space $M^H_S(c_1,n)$.
The Vafa-Witten partition function is 
$$\sfZ^S_{c_1}(x):=\sum_{n} \int_{[M^H_S(c_1,n)]^\vir} c_{\vd(S,c_1,n)} (T^\vir_{M^H_S(c_1,n)} )x^{\vd(S,c_1,n)}.$$
We consider the following new generating function
$$\sfZ^S_{c_1}(\alpha z+p u,t,x):=\sum_{n} \int_{[M^H_S(c_1,n)]^\vir} c^{\C^*}_{\vd(S,c_1,n)} (T^\vir_{M^H_S(c_1,n)}\otimes \frak{t}) \exp\big(\mu(\alpha z+p u)\big)x^{\vd(S,c_1,n)}.$$
Here 
$\frak{t}$ is the trivial line bundle with equivariant first Chern class $t$.
Note that 
\begin{align*}
\sfZ^S_{c_1}(\alpha z+p u,t,x)&=\sum_{n,k}\int_{[M^H_S(c_1,n)]^\vir}c _k(T^\vir_{M^H_S(c_1,n)})t^{\vd(S,c_1,n)-k} \exp\big(\mu(\alpha z+p u)\big)x^{\vd(S,c_1,n)}\\
&=\sum_{n} \int_{[M^H_S(c_1,n)]^\vir}c (T^\vir_{M^H_S(c_1,n)}) \exp\big(\mu(\alpha zt+p ut^2)\big)x^{\vd(S,c_1,n)}\\&=\sfZ^S_{c_1}(\alpha zt+p ut^2,1,x),
\end{align*}
thus the variable $t$ is redundant, but it serves to interpolate between $\sfZ^S_{c_1}(x)$ and $\sfD^S_{c_1}(\alpha z+p u, x)$.
In fact it is obvious from the above that
\begin{align*}
\sfZ^S_{c_1}(\alpha z+ p u,t,x)|_{t=0}&=\sfZ^S_{c_1}(x),\\
\Coeff_{t^{k}}\big[ \sfZ^S_{c_1}(\alpha z+ p u,t^{-1},xt)\big]&=\sum_{n}\int_{[M^H_S(c_1,n)]^\vir}c _k(T^\vir_{M^H_S(c_1,n)}) \exp\big(\mu(\alpha z+p u)\big)x^{\vd(S,c_1,n)}.
\end{align*} In particular 
$\sfZ^S_{c_1}(\alpha z+ p u,t^{-1},xt)|_{t=0}=\sfD_{c_1}^S(\alpha z+pu,x).$

Let $$\overline G_2(x)=G_2(x)+\frac{1}{24}=\sum_{n>0}\sigma_1(n)  x^n$$
be the Eisenstein series of weight $2$, write $D:=x\frac{d}{dx}$.
Let 
\begin{align*}G_p(x)&:=2\overline G_2(x^2)\quad
G_Q(x)=\frac{1}{2}(D G_2)(x^2),\quad
G_{S}(x):=(G_2(x)-G_2(-x))/2.
\end{align*}

We denote $Q(\alpha)=\int_S \alpha^2$, the quadratic form and for $C\in H^2(S,\Z)$ let
$\<C,\alpha\>$ be the intersection product.
Let $a_1,\ldots,a_s$ be the Seiberg-Witten basic classes of $S$ (in the Mochizuki notation, the corresponding characteristic  cohomology classes are $\widetilde a_i=K_S-2a_i$).
Denote
\begin{align*}\widetilde \sfZ^S_{c_1}(\alpha z+pu,t,x)=&\,4 \left(\frac{1}{2\overline{\eta}(x^2)^{12}}\right)^{\chi(\O_S)}\left( \frac{2 \overline{\eta}(x^4)^2}{\theta_3(x)} \right)^{K_{S}^2}\exp\big(G_Q(x)Q(\alpha)z^2t^2+G_p(x)(ut^2-\<K_S,\alpha\>zt)\big)\\
&\cdot\sum_{i=1}^s \SW(a_i)(-1)^{\<c_1,a_i\>} \Big(\frac{\theta_3(x)}{\theta_3(-x)}\Big)^{\<K_S,a_i\>}\exp\big(G_S(x) \<K_S-2 a_i,\alpha\>zt\big).
\end{align*}
Then we have the conjecture
\begin{conjecture} \label{WVW}
$$\sfZ^S_{c_1}(\alpha z+pu,t,x)=\frac{1}{2}\big(\widetilde \sfZ^S_{c_1}(\alpha z+pu,t,x)+(\sqrt{-1})^{c_1^2-\chi(\O_S)}\widetilde \sfZ^S_{c_1}(\alpha z+pu,t,\sqrt{-1}x)\big).$$
\end{conjecture}
By definition it is obvious that at $t=0$, Conjecture \ref{WVW} specializes to Conjecture \ref{generalsurfconj}.
Now we look at the specialization at $t^{-1}=0$.
\begin{proposition} Conjecture \ref {WVW} is true for $\sfZ^S_{c_1}(\alpha z+ p u,t^{-1},xt)$ modulo $t^2$.
\end{proposition}
\begin{proof}Specializing $t^{-1}=0$ in $\sfZ^S_{c_1}(\alpha z+ p u,t^{-1},xt)$, Conjecture \ref {WVW} says that 
$$\sfD^S_{c_1}(\alpha z+pu,x)=\frac{1}{2}\big(\widetilde \sfD^S_{c_1}(\alpha z+pu,x)+(\sqrt{-1})^{c_1^2-\chi(\O_S)}\widetilde \sfD^S_{c_1}(\alpha z+pu,\sqrt{-1}x)\big),$$ with 
\begin{align*}
&\widetilde \sfD^S_{c_1}(\alpha z+pu,x):=\widetilde \sfZ^S_{c_1}(\alpha z+pu,t^{-1},xt)|_{t=0}\\&=2^{2+K_S^2-\chi(\O_S)}\exp\Big(\frac{Q(\alpha)}{2}x^2z^2+2x^2u)\Big)
\sum_{i=1}^s \SW(a_i)(-1)^{\<c_1,a_i\>}\exp\big(\<K_S-2 a_i,\alpha\>xz\big),\end{align*}
which is a reformulation of the Witten conjecture for Donaldson invariants of algebraic surfaces, proved in \cite{GNY3}.

By \cite[Prop.~8.3.1]{HL}, we have $$c _1(T^\vir_{M^H_S(c_1,n)})=-2\mu(K_S)$$ in $H^2(M^H_S(c_1,n),\Q)$.
Thus after the substitution $t\to t^{-1}$, $x\to xt$, and using the result for the coefficient of $t^0$, the coefficient of $t^1$ of the right hand side of Conjecture \ref {WVW} is
\begin{align*}
&\sum_{n}\int_{[M^H_S(c_1,n)]^\vir}c _1(T^\vir_{M^H_S(c_1,n)}) \exp\big(\mu(\alpha z+p u)\big)x^{\vd(S,c_1,n)}
=z^{-1} \frac{\partial}{\partial w}\sfD^S_{c_1}(\alpha z-2K_Swz+pu,x) \Big|_{w=0}\\&=\frac{1}{2}\big(F_1(x)+(\sqrt{-1})^{c_1^2-\chi(\O_S)}F_1(\sqrt{-1}x)\big)
,\end{align*}
with 
\begin{align*}
F_1(x)=&\,z^{-1} \frac{\partial}{\partial w}\widetilde \sfD^S_{c_1}(\alpha z-2K_Swz+pu,x) \Big|_{w=0}\\=&\,2^{2+K_S^2-\chi(\O_S)}\exp\Big(\frac{Q(\alpha)}{2}x^2z^2+2x^2u)\Big)\\&\cdot
\sum_{i=1}^s \big(\<-2K_S,\alpha\>x^2z +\<K_S-2a_i,-2K_S\>x\big)\SW(a_i)(-1)^{\<c_1,a_i\>}\exp\big(\<K_S-2 a_i,\alpha\>xz\big).\end{align*}
And for the left hand side of Conjecture \ref {WVW} we get
$$
\frac{\partial}{\partial t}\sfZ^S_{c_1}(\alpha z+pu,t^{-1},xt)|_{t=0}=\frac{1}{2}\big(F_2(x)+(\sqrt{-1})^{c_1^2-\chi(\O_S)}F_2(\sqrt{-1}x)\big)
$$ 
with 
\begin{align*}
F_2(x)=\,&\frac{\partial}{\partial t}\widetilde \sfZ^S_{c_1}(\alpha z+pu,t^{-1},xt)|_{t=0}=2^{2+K_S^2-\chi(\O_S)}\exp\Big(\frac{Q(\alpha)}{2}x^2z^2+2x^2u)\Big)\\&\cdot
\sum_{i=1}^s \big(-2K_S^2x -\<2K_S,\alpha\>x^2z+4\<a_i,K_S\>x\big)\SW(a_i)(-1)^{\<c_1,a_i\>}\exp\big(\<K_S-2 a_i,\alpha\>xz\big).\end{align*}
\end{proof}

\begin{definition} Let $\pi$ and  $q$ be the projections from $S^{[n_1]}\times S^{[n_2]}\times S$ to $S^{[n_1]}\times S^{[n_2]}$  and $S$.
For $\beta\in H_*(S,\Q)$ $a_1,c_1 \in A^1(S)$,  put 
$$\nu(\beta):=\pi_{*}^{\C^*}\big(-\ch_2^{\C^*}\big(\I_1(a_1-c_1/2)\otimes \s\oplus \I_2(c_1/2-a_1)\otimes \s^{-1}\big)\cap q^*(\beta)\big).$$

For $\alpha\in H_2(S,\Q)$ and $p\in H_0(S,\Z)$ the class of a point, we put 
\begin{align*}&\sfZ_S(a_1,c_1,s,\alpha z+pu,q):=\\
&\sum_{n_1, n_2 \geq 0} q^{n_1+n_2} \int_{S^{[n_1]} \times S^{[n_2]}} \frac{c^{\C^*}(E_{n_1,n_2})\exp\big(\nu(\alpha)zt+\nu(p)ut^2\big) \, \Eu(\O(a_1)^{[n_1]}) \, \Eu(\O(c_1-a_1)^{[n_2]} \otimes \s^2)}{\Eu(E_{n_1,n_2} - \pi_1^* \, T_{S^{[n_1]}} - \pi_2^* \, T_{S^{[n_2]}})}.
\end{align*}
\end{definition}
We denote by $\widetilde\A(a_1,c_1-a_1,c_2,s)$ the expression from \eqref{cA}, in the case $$P(\E)=c(T_M^\vir)\exp(\mu(\alpha) zt +\mu(p)u t^2).$$
As in \eqref{keyexpr} we find that
\begin{equation*}
\begin{split}
\sum_{c_2\in \Z}\widetilde\A(a_1,c_1-a_1,c_2,s)q^{c_2}=\,&\sfZ_S(a_1,c_1,s,\alpha z+pu,q) \\
&\cdot (2s)^{\chi(\O_S)}\left(\frac{2s}{1+2s}\right)^{\chi(c_1-2a_1)}\left(\frac{-2s}{1-2s}\right)^{\chi(2a_1-c_1)}q^{a_1(c_1-a_1)}.
\end{split}
\end{equation*}
Now we have the following analogue of Proposition \ref{univ}.

\begin{proposition}
For $i=1,\ldots,7$ denote $A_i:=A_i(s,1,q)$ for $A_i(s,y,q)$ the universal functions from Proposition \ref{univ}.
There exist universal functions $A_8(s,q), \ldots, A_{12}(s,q)\in 1+q\Q(\!(s)\!)[[q]]$, such that for any smooth projective surface $S$ and any $a_1,c_1 \in A^1(S)$, $\alpha\in H_2(S,\Q)$, we have
$$\sfZ_S(a_1,c_1,s,\alpha z+pu,q)=A_1^{a_1^2} A_2^{a_1c_1}A_3^{c_1^2}A_4^{a_1K_S}A_5^{c_1K_S}A_6^{K_S^2}A_7^{\chi(\O_S)}A_8^{\alpha^2 z^2 t^2} A_9^{a_1\alpha z t}A_{10}^{c_1\alpha z t}
A_{11}^{K_S\alpha z t} A_{12}^{ut^2}.$$
\end{proposition}
The proof is a simple adaptation of the proof of  Proposition \ref{univ} (which is an adaptation of \cite{GNY1}), note that already in \cite{GNY1} it was shown how to deal with the classes $\nu(\beta)$.
Write $M=M^H_S(c_1,c_2)$. 
The same proof as that of Corollary \ref{univcor} shows that under the assumptions of Corollary \ref{univcor} we have
\begin{align*}
&\int_{[M^\vir]}
c_{\vd(M)}^{\C^*}(T^{\vir}_M\otimes {\frak t})\exp\big(\mu(\alpha)z+\mu(p)u\big)=\int_{[M^\vir]}
c(T^{\vir}_M)\exp\big(\mu(\alpha)zt+\mu(p)ut^2\big)\\&=\Coeff_{s^0q^{c_2}}
\Bigg[\sum_{{a_1\in H^2(S,\Z)}\atop a_1H<{(c_1-a_1)H}}
\SW(a_1)\sfA_{(a_1^2,a_1c_1,c_1^2,a_1K_S,c_1K_S,K_S^2,\chi(\O_S))}A_8^{\alpha^2 z^2 t^2} A_9^{a_1\alpha z t}A_{10}^{c_1\alpha z t}
A_{11}^{K_S\alpha z t} A_{12}^{ut^2}\Bigg].
\end{align*}
Let  $F$, $G$ be the classes of the fibres of the two projections of $\PP^1\times \PP^1$. The functions $A_8,\ldots,A_{12}$ are determined by $A_1,\ldots, A_7$ and
$\sfZ_{\PP^1\times \PP^1}(a_1,c_1,s,\alpha z,q)$, for
$$(a_1,c_1,\alpha)=(0,0,G-F), \  (0,0,-F),\ (G,G,F-G),\ \ (0,G,F-G),\ $$
and by $\sfZ_{\PP^1\times \PP^1}(0,0,s,pu ,q)$.

In these cases $S$ is a toric surface with torus $T=\C^{*2}$. We use the notations and conventions from Section 4.
Thus for $\sigma=1,\ldots, e(S)$ let $p_\sigma$ be a fixpoint of the $T$-action on $S$, let $U_\sigma$ be a maximal $T$-invariant open affine  neighbourhood of $p_\sigma$, with coordinates $x_\sigma,y_\sigma$ such that $T$ acts by
$$t\cdot (x_\sigma,y_\sigma)=(\chi(v_\sigma)x_\sigma,\chi(w_\sigma)y_\sigma),$$ where $\chi(v_\sigma)$ is a character with weight $v_\sigma$. 
Then in the notations of Section 4 the $T$-fixed points on $S^{[n_1]}\times S^{[n_2]}$ are parametrized by pairs
$(\bslambda, \bsmu )$ with 
$\bslambda=\{\lambda^{(\sigma)}\}_{\sigma=1,\ldots,e(S)},$ $\bsmu =\{\mu^{(\sigma)}\}_{\sigma=1,\ldots,e(S)}$ tuples of partitions.
Writing $\lambda^{(\sigma)}=(\lambda_1,\ldots,\lambda_\ell)$, here $\lambda^{(\sigma)}$ corresponds to the subscheme $Z_{\lambda^{(\sigma)}}$ supported in 
$p_\sigma$ with ideal $$I_{\lambda^{(\sigma)}}=(y_\sigma^{\lambda_1},x_\sigma y_\sigma^{\lambda_2},\ldots,x_\sigma^{\ell-1}y_\sigma^{\lambda_\ell},x_\sigma^\ell),$$ and similar for the $\mu^{(\sigma)}$.
The fixpoint of $S^{[n_1]}\times S^{[n_2]}$ corresponding to $(\bslambda, \bsmu )$ is the pair $(Z_{\bslambda},Z_{\bsmu})$ with $Z_{\bslambda}$ the disjoint union of the $Z_{\lambda^{(\sigma)}}$
and similar for $Z_{\bsmu}$. 
A $T$-equivariant divisor $a$ is again on $U_\sigma$ given by a character of weight $a_\sigma$.
 In the $T$-equivariant $K$-group $K_0^T(pt)$, the fibre of $\I_1(a)\oplus \I_2(-a)$ at the point $(p_\sigma,Z_{\lambda^{(\sigma)}},Z_{\mu^{(\sigma)}})$ is 
$$\chi(a_\sigma) \big(1-(1-\chi(v_\sigma))(1-\chi(w_\sigma))Z_{\lambda^{(\sigma)}}\big)+\chi(-a_\sigma) \big(1-(1-\chi(v_\sigma))(1-\chi(w_\sigma))Z_{\mu^{(\sigma)}}\big).$$
In particular, denoting $i^*_{(p_\sigma,Z_{\lambda^{(\sigma)}},Z_{\mu^{(\sigma)}})}$, the equivariant pullback to a point via the embedding of the fixpoint $(p_\sigma,Z_{\lambda^{(\sigma)}},Z_{\mu^{(\sigma)}})$, we get 
$$i^*_{(p_\sigma,Z_{\lambda^{(\sigma)}},Z_{\mu^{(\sigma)}})}\ch^T_2(\I_1(a)\oplus \I_2(-a))=a_\sigma^2-v_\sigma w_\sigma (|\lambda^{(\sigma)}|+|\mu^{(\sigma)}|),$$
and thus for $b$ an equivariant divisor on $S$, we get
$$i^*_{(Z_{\bslambda},Z_{\bsmu})}\pi_{S^{[n_1]}\times S^{[n_2]} *} \big(\ch_2^T(\I_1(a)\oplus \I_2(-a))  \pi_S^* c_1(b)\big)=\sum_{\sigma=1}^{e(S)}  \Big(\frac{b_\sigma a_\sigma^2}{v_\sigma w_\sigma}-b_\sigma (|\lambda^{(\sigma)}|+|\mu^{(\sigma)}|\Big).$$
On $\PP^1\times \PP^1$ we can represent the class of a point by $c_2(\PP^1\times \PP^1)/4$, which gives
$$i^*_{(Z_{\bslambda},Z_{\bsmu})}\pi_{S^{[n_1]}\times S^{[n_2]} *} \big(\ch_2^T(\I_1(a)\oplus \I_2(-a)) \pi_S^*(pt)\big)=\sum_{\sigma=1}^{e(S)}  \frac{1}{4}\big({a_\sigma^2}-v_\sigma w_\sigma (|\lambda^{(\sigma)}|+|\mu^{(\sigma)}|)\big).$$
Combining this with the formulas of Section 4 we can compute in the above cases $\sfZ_{\PP^1\times \PP^1}(a_1,c_1,s,\alpha z,q)$ by localization.

This has been implemented as  as a PARI/GP program, which computed   $A_8(s,q)$, $\ldots$, $A_{12}(s,q)$ up to order $q^{10}$ and any order in $s$. 
We then checked Conjecture \ref{WVW} in  a number of cases, in all these cases $\alpha$ is an arbitrary class in $H_2(S,\Q)$. In all these cases the conjecture has been verified with the bounds of Section 6 adapted accordingly.
\begin{enumerate}
\item $K3$ surfaces, for $c_1^2=0,\ldots,30$.
\item Blowup $\widehat S$ of a $K3$ surface $S$ in a point with exceptional divisor $E$. Here we take  $c_1=L+kE$ with  $L$  the pullback of a line bundle on $S$ with $L^2=0,\ldots,30$ and $k=0,1$.
\item Elliptic surfaces $S\to \PP^1$ of type $E(n)$, as in Section 6.2, for $n=3,\ldots,7$ with $c_1=kF$ or $B+kF$, with $B$ the class of a section with $B^2=-n$ and $F$ the class of a fibre and $k=0,\ldots 7$. 
\item Double cover of $\PP^2$, branched along a curve of degree $2d$ for $d=4,5$, with $c_1$ the pullback of $k$-times with hyperplane class for $k=-1\ldots,4$.
\item General quintic in $\PP^3$, with $c_1$ the restriction of $k$ times the hyperplane class with $k=0,1,2,3$.
\end{enumerate}




{\tt{gottsche@ictp.it, m.kool1@uu.nl}}

\begin{thebibliography}{MNOP}
\bibitem[BLN]{BLN} L.~Baulieu, A.~Losev, and N.~Nekrasov, \textit{Chern-Simons and twisted supersymmetry in various dimensions}, Nuclear Phys.~B 522 (1998) 82--104. 
\bibitem[BE]{BE} A.~Braverman and P.~Etingof, \textit{Instanton counting via affine Lie algebras II: from Whittaker vectors to the Seiberg-Witten prepotential}, Studies in Lie theory, Progr.~Math.~243, 61--78, Birkh\"auser, Boston (2006).
\bibitem[CK]{CK} H.-l.~Chang and Y.-H.~Kiem, \textit{Poincar\'e invariants are Seiberg-Witten invariants}, Geom.~and Topol.~17 (2013) 1149--1163. 
\bibitem[CFK]{CFK} I.~Ciocan-Fontanine and M.~Kapranov, \textit{Virtual fundamental classes via dg-manifolds}, Geom.~Topol.~13 (2009) 1779--1804.
\bibitem[DPS]{DPS} R.~Dijkgraaf, J.-.~Park, and B.~J.~Schroers, \textit{$N=4$ supersymemtric Yang-Mills theory on a K\"ahler surface}, hep-th/9801066 ITFA-97-09.
\bibitem[DKO]{DKO} M.~D\"urr, A.~Kabanov, and C. Okonek, \textit{Poincar\'e invariants}, Topol.~46 (2007) 225--294.
\bibitem[EGL]{EGL} G.~Ellingsrud, L.~G\"ottsche, and M.~Lehn, \textit{On the cobordism class of the Hilbert scheme of a surface}, Jour.~Alg.~Geom.~10 (2001) 81--100. 
\bibitem[FG]{FG} B.~Fantechi and L.~G\"ottsche, \textit{Riemann-Roch theorems and elliptic genus for virtually smooth schemes}, Geom.~Topol.~14 (2010) 83--115.
\bibitem[FM]{FM} R.~Friedman and J.~W.~Morgan, \textit{Obstruction bundles, semiregularity and Seiberg-Witten invariants}, Comm.~Anal.~Geom. 7 (1999) 451--495.
\bibitem[GSY]{GSY} A.~Gholampour, A.~Sheshmani, and S.-Y.~Yau, \textit{Localized Donaldson-Thomas theory of surfaces}, Amer.~Jour.~Math.~142 2 (2020).
\bibitem[Got1]{Got1} L.~G\"ottsche, \textit{Change of polarization and Hodge numbers of moduli spaces of torsion free sheaves on surfaces}, Math.~Z.~223 (1996) 247--260.
\bibitem[Got2]{Got2} L.~G\"ottsche, \textit{Theta functions and Hodge numbers of moduli spaces of sheaves on rational surfaces}, Comm.~Math.~Phys.~206 (1999) 105--136.
\bibitem[GH]{GH} L. G\"ottsche and D.~Huybrechts, \textit{Hodge numbers of moduli spaces of stable bundles on K3 surfaces}, Int.~J.~Math.~7 (1996) 359--372.
\bibitem[GK]{GK} L.~G\"ottsche and M.~Kool, \textit{A rank 2 Dijkgraaf-Moore-Verlinde-Verlinde formula}, Comm.~Numb.~Th.~and Phys.~13 (2019) 165--201.
\bibitem[GNY1]{GNY1} L.~G\"ottsche, H.~Nakajima, and K.~Yoshioka, \textit{Instanton counting and Donaldson invariants}, J.~Diff.~Geom.~80 (2008) 343--390.
\bibitem[GNY2]{GNY2} L.~G\"ottsche, H.~Nakajima, and K.~Yoshioka, \textit{K-theoretic Donaldson invariants via instanton counting}, Pure Appl.~Math.~Q.~5 (2009) 1029--1111.
\bibitem[GNY3]{GNY3} L.~G\"ottsche, H.~Nakajima, and K.~Yoshioka, \textit{Donaldson = Seiberg-Witten from Mochizuki's formula and instanton counting}, Publ.~Res.~Inst.~Math.~Sci.~47 (2011) 307--359.
\bibitem[GS]{GS} L.~G\"ottsche, W.~Soergel, \textit{Perverse sheaves and the cohomology of Hilbert schemes of smooth algebraic surfaces}, Math.~Ann.~296 (1993) 235--245.
\bibitem[Huy]{Huy}  D.~Huybrechts,  \textit{Compact hyper-K\"ahler manifolds: basic results},  Invent.~Math.~135 (1999) 63--113.
\bibitem[HL]{HL} D.~Huybrechts, M.~Lehn, \textit{The geometry of moduli spaces of sheaves}, Cambridge University Press (2010).
\bibitem[Joy]{Joy} D.~Joyce, \textit{Configurations in abelian categories. IV. Invariants and changing stability conditions}, Adv.~Math.~217 (2008) 125--204.
\bibitem[Kim]{Kim}S.-O.~Kim, \textit{Noether-Lefschetz locus for surfaces}, Trans.~Amer.~Math.~Soc.~324 (1991) 369--384. 
\bibitem[Kly]{Kly} A.~A.~Klyachko, \textit{Vector bundles and torsion free sheaves on the projective plane}, preprint Max Planck Institut f\"ur Mathematik (1991).
\bibitem[LLZ]{LLZ} J.~Li, K.~Liu, and J.~Zhou, \textit{Topological string partition functions as equivariant indices}, Asian J.~Math.~10 (2006) 81--114.
\bibitem[LQ1]{LQ1} W.-P.~Li and Z.~Qin, \textit{On blowup formulae for the $S$-duality conjecture of Vafa and Witten}, Invent.~Math.~136 (1999) 451--482.
\bibitem[LQ2]{LQ2}  W.-P.~Li and Z.~Qin, \textit{On blowup formulae for the $S$-duality conjecture of Vafa and Witten II: the universal functions}, Math.~Res.~Lett.~5 (1998) 439--453.
\bibitem[LNS]{LNS} A.~Losev, N.~Nekrasov, and S.~Shatashvili, \textit{Issues in topological gauge theory}, Nuclear Phys.~B 534 (1998) 549--611.
\bibitem[MNOP]{MNOP} D.~Maulik, N.~Nekrasov, A.~Okounkov, and R.~Pandharipande, \textit{Gromov-{W}itten theory and {D}onaldson-{T}homas theory, {I}}, Compos.~Math.~142 (2006) 1263--1285. 
\bibitem[Moc]{Moc} T.~Mochizuki, \textit{Donaldson type invariants for algebraic surfaces}, Lecture Notes in Math.~1972, Springer-Verlag, Berlin (2009). 
\bibitem[Mor]{Mor} J.~W.~Morgan, \textit{The Seiberg-Witten equations and applications to the topology of smooth four-manifolds}, Math.~Notes 44, Princeton Univ.~Press, Princeton NJ (1996).
\bibitem[NY1]{NY1} H.~Nakajima and K.~Yoshioka, \textit{Lectures on instanton counting}, Algebraic structures and moduli spaces, CRM Proc.~Lecture Notes vol.~38, AMS Providence RI, 31--101 (2004).
\bibitem[NY2]{NY2} H.~Nakajima and K.~Yoshioka, \textit{Instanton counting on blowup. I. 4-dimensional pure gauge theory}, Invent.~Math.~162 (2005) 313--355.
\bibitem[Nek1]{Nek1} N.~Nekrasov,  \emph{Four dimensional holomorphic theories}, PhD Thesis, Princeton (1996).
\bibitem[Nek2]{Nek2} N.~Nekrasov, \textit{Seiberg-Witten prepotential from instanton counting}, Adv.~Theor.~Math.~Phys. 7 (2003) 831--864. 
\bibitem[NO]{NO} N.~Nekrasov and A.~Okounkov, \textit{Seiberg-Witten prepotential and random partitions}, The unity of mathematics, 525--596, Progr.~Math.~244, Birkh\"auser, Boston (2006).
\bibitem[Pan]{Pan} R.~Pandharipande, \textit{Descendents for stable pairs on 3-folds}, in: ``Modern geometry: a celebration of the work of Simon Donaldson'', editors: V.~Mu\~noz, I.~Smith, R.~P.~Thomas, Proc.~Symp.~Pure Math.~99 251--288, AMS (2018).
\bibitem[She]{She} J.~Shen, \textit{Cobordism invariants of the moduli space of stable pairs}, J.~London Math.~Soc.~94 (2016) 427--446.
\bibitem[TT1]{TT1} Y.~Tanaka and R.~P.~Thomas, \textit{Vafa-Witten invariants for projective surfaces I: stable case}, Jour.~Alg.~Geom.~(2019) doi.org/10.1090/jag/738. 
\bibitem[TT2]{TT2} Y.~Tanaka and R.~P.~Thomas, \textit{Vafa-Witten invariants for projective surfaces II: semistable case}, Pure Appl.~Math.~Quart.~13 (2017) 517--562.
\bibitem[VW]{VW} C.~Vafa and E.~Witten, \textit{A strong coupling test of $S$-duality}, Nucl.~Phys.~B 431 (1994) 3--77.
\bibitem[Yos1]{Yos1} K.~Yoshioka, \textit{The Betti numbers of the moduli space of stable sheaves of rank 2 on $\PP^2$}, J.~Reine Angew.~Math.~453 (1994) 193--220.
\bibitem[Yos2]{Yos2} K.~Yoshioka, \textit{The Betti numbers of the moduli space of stable sheaves of rank 2 on a ruled surface}, Math.~Ann.~302 (1995) 519--540.
\bibitem[Yos3]{Yos3} K.~Yoshioka, \textit{Number of $\FF_q$-rational points of the moduli of stable sheaves on elliptic surfaces}, Moduli of vector bundles, editor: M.~Maruyama, Lect.~Notes in Pure and Appl.~Math.~179, Marcel Dekker, New York (1996).
\bibitem[Yos4]{Yos4} K.~Yoshioka, \textit{Some examples of Mukai's reflections on K3 surfaces},  J.~Reine~Angew.~Math.~515 (1999) 97--123.
\end{thebibliography}
\end{document}